\documentclass[10pt]{amsart}
\usepackage[T1]{fontenc}
\usepackage{style}  
\usepackage{graphicx}    
\graphicspath{ {./EU-symbol.jpg} }  

\title[Real syntomic cohomology]{Real syntomic cohomology} 
\author[G. Angelini-Knoll]{Gabriel Angelini-Knoll}
\address{Department of Mathematics, Applied Mathematics, and Statistics, Case
Western Reserve University, Cleveland, OH, USA}
\email{gabriel.angelini-knoll@case.edu}
\author[H.J. Kong]{Hana Jia Kong}
\address{School of Mathematical Sciences, Zhejiang University, Hangzhou, Zhejiang, China}
\email{hana.jia.kong@gmail.com}
\author[J.D. Quigley]{J.D. Quigley}
\address{Department of Mathematics\\
University of Virginia\\  Charlottesville, Virginia, U.S.A.}\email{mbp6pj@virginia.edu}
    
\begin{document}  
 
\begin{abstract}  
We introduce a theory of syntomic cohomology for ring spectra with involution, which we call Real syntomic cohomology. 
We show that our construction extends the theory of syntomic cohomology for rings with involution due to Park.
Our construction also refines syntomic cohomology as developed by Bhatt--Morrow--Scholze, Morin, Bhatt--Lurie, and 
Hahn--Raksit--Wilson. We compute the Real syntomic cohomology of Real topological K-theory 
and topological modular forms with level structure. 
\end{abstract}

\maketitle  

\setcounter{tocdepth}{1}
\tableofcontents

\section{Introduction}\label{sec:intro}
Real algebraic K-theory is an invariant of rings with involution which recovers Grothendieck--Witt theory and L-theory. Grothendieck--Witt theory has been studied extensively with applications to quadratic and symmetric bilinear forms, for example~\cite{KV71} and~\cite{KSW16}. L-theory has has also been thoroughly investigated in the study of surgery on manifolds, for example~\cite{Wal73,Ran78}. Real algebraic K-theory remains a subject of active interest, for instance~\cite{CBDHHLMNNS23,CBDHHLMNNSII,CBDHHLMNNSIII}, but there has been very little progress in computing it for ring spectra with involution that are not Eilenberg--MacLane. 

Nonequivariantly, since the late 1980s a strategy known as trace methods has been the primary tool used to compute algebraic K-theory. In the trace methods approach to algebraic K-theory, one first computes topological Hochschild homology, which is accessible using techniques from homological algebra. One can then use extra structure on topological Hochschild homology known as cyclotomic structure to produce topological cyclic homology, which closely approximates algebraic K-theory by the Dundas--Goodwillie--McCarthy theorem~\cite{DGM13}*{Theorem~7.0.0.2}. This approach has been the only way to access algebraic K-theory of ring spectra that are not Eilenberg--MacLane, for example~\cite{AR02,AR12,AKACHR25}. 

In 2019, work of Bhatt--Morrow--Scholze~\cite{BMS19} made this approach more effective by producing a motivic filtration on topological cyclic homology. The associated graded of this filtration is computable and interesting in its own right because it recovers syntomic cohomology, which was first introduced by Fontaine--Messing~\cite{FM87} and 
Kato~\cite{Kat87}. This was extended to commutative ring spectra, also known as $\mathbb{E}_\infty$-rings, by Hahn--Raksit--Wilson~\cite{HRW22}. This has led to new progress on computations of algebraic K-theory of ring spectra that were not possible before, for example~\cite{AKAR23,AKHW24}. 

There has been significant work towards extending trace methods to the setting of Real algebraic K-theory beginning in the 2010s. There is a notion of Real topological Hochschild homology~\cite{DMPR21} and a notion of Real cyclotomic structure~\cite{Hog16, DMP24, QS21b} leading to a notion of Real topological cyclic homology. Harpaz--Nikolaus--Shah~\cite{HNS} also provides the analogue of the Dundas--Goodwille--McCarthy theorem for Real algebraic K-theory. Nevertheless, there remain very few computations of Real algebraic K-theory in the literature suggesting that new computational tools are needed. 

A motivic filtration on Real topological cyclic homology was produced in the setting of discrete commutative rings with involution in work of~\cite{Par23}. The main goal of the present paper is to produce a new computational tool that allows us to gain access to Real algebraic K-theory of ring spectra with involution. We accomplish this by producing a motivic filtration on Real topological cyclic homology that refines the motivic filtration on topological cyclic homology. This filtration is built on the notion of strong evenness in $C_2$-equivariant homotopy theory. Strong evenness controls the regular slice tower and RO$(C_2)$-graded behavior, and is purely equivariant with no nonequivariant counterpart. 

To illustrate the power of these new techniques, we compute the Real syntomic cohomology in the case of Real truncated Brown--Peterson spectra, for example Real topological K-theory and topological modular forms with level structure. Unlike the nonequivariant case, the RO$(C_2)$-graded homotopy groups have nontrivial gap pattern, and thus the computations require more careful and finer analysis, as well as the full Real motivic filtration machinery developed in this paper.
As a consequence, we prove a version of the Lichtenbaum--Quillen conjecture for certain ring spectra with involution, the first result of its kind. This gives some evidence for an analogue of the redshift conjecture of Ausoni--Rognes~\cite{AR08} in the setting of ring spectra with involution. 

\subsection{Results}

We now state the main results of the paper in context. 

\subsubsection{The strongly even filtration and descent}

We begin by producing a motivic filtration on Real topological Hochschild homology and related invariants. 
Each of these motivic filtrations are variants of what we call the \emph{strongly even filtration}. 
First, let $\rho$ denote the regular representation of $C_2$, the cyclic group of order two. 
Let $\pi_{V}^{C_2}X:=\pi_{0}\Map(S^V,X)^{C_2}$ be $\pi_0$ of the space of $C_2$-equivariant maps from the one-point 
compactification of a $C_2$-representation $V$ into a spectrum $X$. We say a $C_2$-spectrum $B$ is \emph{strongly even}
if $\pi_{\rho k-i}^{C_2}B=0$ for $i=1,2,3$; see Lemma~\ref{lem:criteria}. When $A$ is a $C_2$-commutative ring spectrum,
\footnote{For the sake of this introduction, we use the term ``$C_2$-commutative'' as a place-holder for the 
term ``$C_2$-$\bE_{\infty}$'' in the sense of~\cite{NS22}. See Section~\ref{sec:equivariant-multiplicative} for details.} 
we define the \emph{strongly even filtration} as the homotopy limit
\[ 
\filsev^{\bullet}A:=\lim_{A\to B} P_{2\bullet}B
\]
over all maps of $C_2$-commutative ring spectra from $A\to B$ such that $B$ is strongly even. 
We then write $\grsev^{*}A:=\filsev^{*}A/\filsev^{*+1}A $ for the associated graded of this filtered $C_2$-spectrum. 

Just as in topology where one can compute singular homology of a nice topological space using a \v{C}ech cover, 
we can compute the strongly even filtration using certain nice covers, which we call strongly evenly faithfully flat (or seff) 
maps; see Definition~\ref{def:seff}. The following result therefore gives us a foothold to compute the strongly even filtration. 

\begin{thmx}\label{thm:A}
If $A\to B$ is a seff map of $C_2$-commutative ring spectra and $B$ is strongly even, then there is an equivalence
\[ 
\filsev^{\bullet}A\simeq \lim_{q \in \Delta}P_{2\bullet}(B^{ \otimes_{A}q+1})\,.
\]
of filtered $C_2$-spectra. 
\end{thmx} 

As a consequence of Theorem~\ref{thm:A}, we prove that the strongly even filtration on 
Real topological Hochschild homology recovers the HKR filtration on Real topological Hochschild homology of 
Hornbostel--Park~\cite{HP23} and Yang~\cite{Yan25} under certain conditions. 
This is also used to prove that the constructions of Real syntomic cohomology and Real prismatic cohomology 
in the present paper recover the theory of syntomic cohomology and prismatic cohomology of rings with involution from 
Park~\cite{Par23}. See Section~\ref{sec:comparison} for these results and further discussion.

\subsubsection{Artin--Tate Real motivic cohomology}
Morel--Voevodsky~\cite{MV99} developed the stable motivic homotopy category in order to study cohomology theories for varieties.\footnote{In this introduction, we use the term category as a placeholder for the term $\infty$-category.} 
For example, we can consider stable motivic homotopy theory over the real numbers $\bR$. 
The Artin--Tate category over $\bR$ is then the subcategory generated by finite \'etale motives along with the 
projective spaces $\bP_{\bR}^{1}$ and $(\bP_{\bR}^{1})^{-1}$~\cite{BHS22}. 
As a consequence of our work, we identify modules over the strongly even filtration of the sphere with the 
Artin--Tate category over $\mathbb{R}$.  
 
\begin{thmx}\label{thm:B}
There is an equivalence of categories between the modules over the $2$-complete sphere in the 
Artin--Tate category over $\mathbb{R}$ and modules over the $2$-complete  strongly even filtration of the sphere inside of 
filtered $C_2$-spectra. 

Consequently, we can produce modules over the associated graded of the strongly even filtration 
associated to the constant Mackey functor of $M$ for any even $\MU_*\MU_2$-comodule $M$.

Moreover, when $R$ is a $C_2$-$\bE_{\infty}$-ring such that $\MUR\otimes R$ is 
strongly even and $\MU\otimes R^e$ is bounded below, then we can identify 
\[ 
(\mathrm{fil}_{\sev}^*R)_2^{\wedge}\simeq \nu_{\mathbb{R}}(R)_2^{\wedge} 
\]
where $\nu_{\mathbb{R}}$ is the synthetric analogue functor from Burklund--Hahn--Senger~\cite{BHS22} 
that associates an Artin--Tate $\mathbb{R}$-motivic spectrum to a $C_2$-spectrum.
\end{thmx}

In the previous theorem, we use the term ``$C_2$-$\bE_\infty$-ring" for the first time. To properly discuss equivariant multiplicative structure in higher categories, we require the use of \emph{parametrized} $\infty$-categories. As there are not many explicit applications of this framework in the literature, we hope our work will be of independent interest. 

\subsubsection{The motivic filtration}
We define the motivic filtration 
\[
\filmot^\bullet\TCR(A)
\] 
on the Real topological cyclic homology $\TCR(A)$ of a $C_2$-commutative ring spectrum $A$ as a variant of the even filtration. 
We let $\grmot^*\TCR(A)$ denote its associated graded. 
We refer to
\[
\pi_{V}^{C_2}\grmot^w\TCR(A)\]
as the \emph{Real syntomic cohomology} of $A$ for a real $C_2$-representation $V$ and an integer $w$.  
Associated to the motivic filtration, there is a motivic spectral sequence 
\[ 
E_2^{\star,\ast}=\pi_{\star}^{C_2}\grmot^*\TCR(A) \implies \pi_{\star}^{C_2}\TCR(A)
\]
which can be used to compute the Real topological cyclic homology of $A$.

Our main examples of interest are the Real truncated Brown--Peterson spectra $\BPR\langle n\rangle$, 
which play a fundamental role in chromatic equivariant homotopy theory. 
Specifically, when $n=-1$ we have $\BPR\langle -1\rangle=\bF_2$ with trivial involution, 
when $n=0$ we have $\BPR\langle 0\rangle=\bZ_{(2)}$ with trivial involution, when $n=1$ we have $\BPR\langle 1\rangle =\kr$, 
and when $n=2$ we have $\BPR\langle 2\rangle =\tmf_1(3)$. Here $\kr$ denotes connective Real topological K-theory, 
which classifies complex vector bundles equipped with their complex conjugation action and $\tmf_1(3)$ denotes connective
topological modular forms with level structure $\Gamma_1(3)$. 

\begin{thmx}\label{thm:C}
There is a strongly even $C_2$-commutative ring spectrum $\MWR$ and a seff map of $C_2$-commutative algebras in 
Real $2$-cyclotomic spectra 
\[ 
\THR(\BPR\langle n\rangle;\bZ_2)\to \THR(\BPR\langle n\rangle/\MWR;\bZ_2)
\]
for $-1\le n\le 2$. 
\end{thmx}

In order to compute the motivic filtration on Real topological cyclic homology $\TCR(A)$ it is necessary to have a seff map 
of $C_2$-commutative algebras in Real $p$-cyclotomic spectra. 
Here Real $p$-cyclotomic structure is the necessary structure to build Real topological cyclic homology from 
Real topological Hochschild homology, see Section~\ref{sec:trace-background} for details. 
This theorem is therefore essential for computing Real syntomic cohomology.

\subsubsection{Real syntomic cohomology of Real truncated Brown--Peterson spectra}
By Theorem~\ref{thm:B}, we can construct Smith--Toda complexes in the category of modules over the strongly even 
filtration of the sphere spectrum. We write $\gr_{\sev}^*\bS/(\ol{v}_0,\cdots,\ol{v}_n)$ for the $(\gr_{\sev}^*\bS)_2$-module 
corresponding to the object $\MU_*/(2,v_1,\cdots ,v_n)\otimes \mZ$. 
We refer to Real syntomic cohomology with coefficients in $\gr_{\sev}^*\bS/(\ol{v}_0,\cdots,\ol{v}_n)$ as 
\emph{mod $(\ol{v}_0,\ol{v}_1,\cdots ,\ol{v}_n)$ Real syntomic cohomology}. 
Our main application is a computation of the Real syntomic cohomology of $H\mF_2$, $H\mZ_2$, 
Real topological K-theory $\kr$ and topological modular forms with level structure $\Gamma_1(3)$. 
Note that these are $C_2$-commutative forms of Real truncated Brown--Peterson spectra 
$\BPR\langle n\rangle$ for $n=-1$, $0$, $1$, and $2$ respectively. 

\begin{thmx}\label{thm:D}
Let $M_2:=\pi_{\star}^{C_2}H\mF_2$ and let $-1\le n\le 2$.  
The mod $(\ol{v}_0,\cdots,\ol{v}_{n})$ Real syntomic cohomology of any $C_2$-commutative ring spectrum 
form of $\BPR\langle n\rangle$ is 
\[ 
M_2[\ov_{n+1}]\langle\delta ,\ol{\lambda}_{1},\cdots ,\ol{\lambda}_{n+1}\rangle\oplus 
\bigoplus_{s=1}^{n+1}  M_2[\ov_{n+1}]\langle \lambda_s :s\in \{1,\cdots ,n\}-s\rangle \{\ol{\Xi}_{d,s} : 0 < d<p\} \,.
\]
Here $|\olambda_s|=(2^s\rho-1,1)$, $|\ol{\Xi}_{d,i}|=((2^s-2^{s-1}d)\rho-1,1)$, and $|\partial|=(-1,1)$,
 and $M_2 \langle x \rangle$ denotes a exterior algebra over $M_2$ on a class $x$ and $M_2[y]$ denotes a polynomial algebra over $M_2$ on a class $y$. 
\end{thmx}

We also recover the computations of $\pi_{\star}^{C_2}\TCR(\bF_2)$ from~\cite{QS21a,DMP24} 
and $\pi_{\star}^{C_2}\TCR(\bZ_2)/2$, which could be deduced from \cite{HLN21} or~\cite{DMP24}. 

\subsubsection{Lichtenbaum--Quillen for Real syntomic cohomology}
In the spirit of the Lichtenbaum--Quillen conjectures~\cite{Lic73,Qui75}, and their generalization due to 
Ausoni--Rognes~\cite{AR08}, we ask the following question. 
For the statement, let $\mathrm{Thick}(\mF_2)$ denote the thick subcategory of $C_2$-spectra generated by the 
constant Mackey functor $\mF_2$. We will say that a connective $C_2$-commutative ring spectrum is regular if  
\begin{enumerate}
\item $\upi_0R$ is coherent (every finitely generated  ideal in the sense of~\cite{Nak12} is finitely presented), 
\item $\upi_iR$ is finitely presented as a $\upi_0R$-modules, and 
\item the $H\pi_0R$-module $H\pi_0R\otimes_{R}H\pi_0R$ is in the thick subcategory generated by $H\pi_0R$. 
\end{enumerate}
(cf.~\cite{BL14}). 

\begin{quest}[Lichtenbaum--Quillen]
Suppose that $R$ is a regular $C_2$-commutative ring spectrum and assume that $\grsev^*R/(\ol{v}_0,\ol{v}_1,\cdots,\ol{v}_n)$ 
is in $\mathrm{Thick}(\mF_2)$, then is 
\[
\grmot^*\TCR(R)/(\ol{v}_0,\ol{v}_1,\cdots,\ol{v}_{n+1})
\] 
in $\mathrm{Thick}(\mF_2)$? 
\end{quest}

Our Theorem~\ref{thm:D} answers the Lichtenbaum--Quillen question above in the affirmative in the case of 
$\BPR\langle n\rangle$ for $-1\le n\le 2$. The only obstruction to extending this result to $n>2$ is the fact 
that $\BPR\langle n\rangle$ is not even known to be an $\bE_{\sigma}$-algebra for $n>2$ and this structure is necessary for 
Real topological Hochschild of  $\BPR\langle n\rangle$ to be defined.

\begin{remark}
The Real analogues of the Lichtenbaum--Quillen conjecture have previously been studied by~\cite{BKSO15,KRO20} 
for rings with involution, but our work is the first work to study higher chromatic analogues of Lichtenbaum--Quillen 
for Real algebraic K-theory of $C_2$-commutative ring spectra along the lines of the Ausoni--Rognes conjectures~\cite{Rog00,AR08}. 
For related work of a different flavor, see \cite{Lan22} and~\cite{Lev25}. 
\end{remark}

\subsection{Outline} 
In Section~\ref{sec:filtration}, we introduce the strongly even filtration and prove 
Theorem~\ref{thm:A}. In Section~\ref{sec:motivic-filtrations}, we introduce the theory of 
Real trace methods and extend the strongly even filtration to this setting in order to define 
Real motivic filtrations. We also introduce the Real Hochschild--May spectral sequence and 
discuss equivariant suspension maps, both of which are used for our computations in 
Section~\ref{sec:THR}. In  Section~\ref{sec:motivic}, we prove a parametrized refinement of 
Theorem~\ref{thm:B}. The remainder of Section~\ref{sec:realmotivic} is included for potential future applications, 
but we do not make essential use of them in the present work. The goal of Section~\ref{sec:base} is to prove 
Theorem~\ref{thm:C}.  Section~\ref{sec:comparison} is independent of the rest of the paper, 
but it is included to prove that our construction recovers the version of syntomic cohomology introduced by~\cite{Par23} and the HKR theorem of \cite{HP23} and~\cite{Yan25}.
In Sections~\ref{sec:THR}, \ref{sec:tool}, and~\ref{sec:syntomic}, we prove Theorem~\ref{thm:D} and also compute 
Real topological cyclic homology of $\bF_2$ and mod $2$ Real topological cyclic homology of $\bZ_2$. 
We discuss the foundations of equivariant homotopy theory that we use in this paper in Appendix~\ref{sec:equivariant}.

\subsection{Conventions} 
We freely use the theory of $\infty$-categories, or quasicategories~\cite{Joy02}, and refer the reader to \cite{Lur09}, 
\cite{Lur17}, and~\cite{Lur18} for these foundations. 
For example, all limits and colimits should be interpreted as homotopy limits and homotopy colimits. 
\begin{itemize}
    \item   We write $\Top$ for the $\infty$-category of Kan complexes, which we simply refer to as spaces from now on. 
            We write $\Top_*$ for the $\infty$-category of pointed spaces, and $\Sp$ for the $\infty$-category of spectra. 
            Given a symmetric monoidal $\infty$-category $C$, we write $\CAlg(C)$ for the $\infty$-category of $\bE_\infty$-algebras 
            in $C$.
    \item   We freely use the theory of $G$-$\infty$-categories and refer the reader to Appendix~\ref{sec:equivariant} 
            for further notation in that context and a concise overview. We write $\cC_{p}$ for the $\infty$-category of 
            $p$-complete objects in an $\infty$-category $\cC$. We use the same notation in the context of $G$-$\infty$-categories. 
    \item   We let $\bZ_{\le}$ denote the the integers regarded as a partially ordered set with respect to $\le$ and let 
            $\mathrm{Dis}(\bZ)$ denote the discrete category associated to the set of integers. Given an $\infty$-category $\cC$, 
            we write $\Fil(\cC):=\Fun(\bZ_{\le}^{\op},\cC)$ and $\Gr(\cC):=\Fun(\mathrm{Dis}(\bZ),\cC)$. Given $M\in \Gr(\cC)$, 
            we write $M^t:=M(t)$ and given $F\in \Fil(\cC)$, we write $F^s:=F(s)$.
    \item   When we say a spectral sequence \emph{conditionally converges} or \emph{strongly converges}, 
            we will always mean this in the sense of~\cite{Boa99}. For example, if a spectral sequence arises from a filtered object 
            in a stable $\infty$-category $I$ and $\lim_s I^s=0$ then it conditionally converges (to $\colim_sI^s$). 
    \item   Let $R$ be an ($\RO(C_p)$-)graded ring $R$. We write $R[y]$ for a polynomial algebra over $R$ on a class $y$ 
            in some ($\RO(C_p)$-)grading degree, we write $R\langle x\rangle$ for an exterior algebra on a class $x$ in 
            some  ($\RO(C_p)$-)grading degree, and we write $R[y]\{y\}$ for the $R$-module $R\{y,y^2,\cdots \}$ where $y$ is 
            in some ($\RO(C_p)$-)grading degree. 
\end{itemize}

\subsection{Acknowledgments}
The authors would like to thank Bastiaan Cnossen, Tobias Lenz, and Lennart Meier for useful conversations during the course of 
this project. Special thanks go to Michael A. Hill and Jeremy Hahn for many helpful observations that improved this work. 

The idea for this project grew from a conference on the topic ``Equivariant Stable Homotopy Theory and p-adic Hodge Theory'' 
at the Banff International Research Station (BIRS). The authors would like to thank BIRS for providing ideal working 
conditions for the workshop as well as the organizers Andrew Blumberg, Teena Gerhardt, and Michael A. Hill. 

The authors would like to thank the Isaac Newton Institute for Mathematical Sciences, Cambridge, for support and 
hospitality during the programme Equivariant homotopy theory in context where work on this paper was undertaken. 
This work was supported by EPSRC grant no EP/K032208/1.

Angelini--Knoll and Quigley are grateful to Max Planck Institute for Mathematics Bonn for its hospitality and financial support. 
This project has received funding from the European Union's Horizon 2020 research and innovation programme under the 
Marie Sk\l{}odowska-Curie grant agreement No 1010342555.\thinspace\includegraphics[scale=.15]{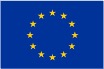} 
Quigley was partially supported by an AMS--Simons Travel Grant
and NSF grants DMS-2039316, DMS-2135884, and DMS-2414922 (formerly DMS-2203785 and DMS-2314082).

\section{The strongly even filtration}\label{sec:filtration}
We introduce the regular slice filtration in Section~\ref{sec:slice} and strongly even $C_2$-spectra in Section~\ref{sec:even} in order to define the strongly even filtration in Section~\ref{sec:strongly-even-filtration} and its $p$-complete variant in Section~\ref{sec:p-complete-strongly-even-filtration}. We then introduce the notion of strongly even faithfully flat maps and prove descent for such maps in Section~\ref{sec:descent}. In Section~\ref{sec:convergence}, we consider convergence and $p$-completion of the strongly even filtration. We refer the reader to Appendix~\ref{sec:equivariant} for background and our notation in equivariant stable homotopy theory. 

\subsection{The regular slice filtration}\label{sec:slice}
Following \cite[\S~2--\S~3]{Ull-thesis}, we define the \emph{regular slice cells} of dimension $k$ to be the 
$C_2$-spectra~${C_2}_+\otimes_{H}\bS^{n\rho_H}$  for $n|H|=k$ and $H\subset C_2$ a subgroup. 
Let $\tau_{k}$ denote the localizing subcategory of $C_2$-spectra generated by the regular slice cells of dimension $\ge k$. 
Let $\tau_{k}^{\perp}$ denote the full subcategory of $\Sp^{C_2}$ on those $C_2$-spectra $X$ such that $[Y,X]=0$ for 
all $Y\in \tau_{k}$.

\begin{construction}
 We define the functor $P^k$ to be the left adjoint to the inclusion $\tau_{k}^{\perp}\subset \Sp^{C_2}$ and 
 we define $P_k:=\fib(\id \to P^{k-1})$. These functors assemble to a functor
\[ 
P_{\bullet} : \Sp^{C_2}\longrightarrow \Fil(\Sp^{C_2})
\]
from the category of $C_2$-spectra to the category of filtered $C_2$-spectra called the \emph{regular slice filtration}
 with associated graded functor
\[ 
P_{\ast}^{\ast}:=\gr^{\ast}\circ P_{\bullet}: \Sp^{C_2} \longrightarrow \Gr(\Sp^{C_2})
\]
from the category of $C_2$-spectra to the category of graded $C_2$-spectra. 
\end{construction}

\begin{remark}
Note that at the underlying level, $P_n$ is simply the $n$-th Postnikov cover. 
\end{remark}

By~\cite[Proposition~3.1.20, Corollary~3.1.13]{Yan25}, the $C_2$-$\infty$-category of filtered $C_2$-spectra and 
the $C_2$-$\infty$-category of graded $C_2$-spectra admit canonical $C_2$-monoidal structures given by Day convolution, 
and the associated graded functor is $C_2$-symmetric monoidal. We can therefore make sense of the following definition. 

\begin{defin}
Let $\CAlg^{C_2}_{\fil}:=\CAlg^{C_2}(\Fil(\Sp)^{C_2})$ and let $\CAlg^{C_2}_{\gr}:=\CAlg^{C_2}(\Gr(\Sp)^{C_2})$ . 
\end{defin}

\begin{lem}
The regular slice filtration functors $P_\bullet$ and $P_\ast^\ast$ restrict to functors
\[ 
P_{\bullet} : \CAlg^{C_2}\longrightarrow \CAlg^{C_2}_{\fil}\,,\]
\[
P_{*}^{*} : \CAlg^{C_2}\longrightarrow \CAlg^{C_2}_{\gr}\,,
\]
respectively. Moreover, the regular slice filtration satisfies 
\[ 
\lim P_s\simeq 0  \,,
\text{ and } 
\colim P_s\simeq \id \,.
\]
\end{lem}
\begin{proof}
The first statement follows from~\cite[pp.~5]{Ull13}, or in our setting see~\cite[Proposition~3.2.19]{Yan25}; 
see also~\cite[Construction~4.2]{BHS22}. 
For the second statement this follows by passing to geometric fixed points and using the fact that the usual 
Postnikov t-structure on spectra is right and left separated; see~\cite[Proposition~3.2.22]{Yan25}.  
\end{proof}

The following lemma is immediate from the construction. 

\begin{lem}\label{lem:property-of-slices}
The functor 
\[ 
P_*^* : \Sp^{C_2}\longrightarrow  \Gr(\Sp^{C_2})
\]
commutes with filtered colimits and 
finite products. 
\end{lem}


\subsection{Strongly even $C_2$-spectra}\label{sec:even}
Let $\m{M}$ be a $C_2$-Mackey functor. 
We write 
\[ 
    \cF_1\m{M} \subset \m{M}
\]
for the sub-Mackey functor generated by $M(C_2/e)$ together with its Weyl group action following~\cite{Ull13}. 

We can then describe the associated graded of the regular slice filtration.  
\begin{prop}\label{prop:slices}
        For any $C_2$-spectrum $E$, the regular slice filtration satisfies
        \[
            P_{n}^{n}(E)=\begin{cases}
                \Sigma^{k\rho} \m{\pi}_{k\rho}(E) \quad & n=2k,\\
                \Sigma^{k\rho-1}\mathcal{F}_1\upi_{k\rho-1}(E) \quad & n=2k-1.
            \end{cases}
        \]
\end{prop} 

\begin{proof}
As in the proof of~\cite[Theorem~2.18]{Hil12}, the regular slice filtration satisfies 
\[ 
P_{2k+n}^{2k+n}(\Sigma^{k \rho }X)= \Sigma^{k \rho}P_{n}^{n}(X)
\]
by~\cite[pp.~1745]{Ull13}. 
Applying this in the case $X=\Sigma^{-k\rho}E$ and $n=-1$
, we determine that 
\[ 
P_{2k-1}^{2k-1}(E)=\Sigma^{k \rho}P_{-1}^{-1}(\Sigma^{-k\rho}E)\,.
\]
Applying this in the case $X=\Sigma^{-k\rho}E$ and $n=0$, we determine that 
\[ 
P_{2k}^{2k}(E)= \Sigma^{k \rho}P_{0}^{0}(\Sigma^{-k\rho}E)
\]
so it suffices to compute $P_{-1}^{-1}(\Sigma^{-k\rho}E)$ and $P_0^0(\Sigma^{-k\rho}E)$. 
By~\cite[Corollary~8.9]{Ull-thesis}, we have
\[
P_0^0(\Sigma^{-k\rho}E)=\upi_0(\Sigma^{-k\rho}E)=\upi_{k\rho}E \,,
\]
\[
P_{-1}^{-1}(\Sigma^{-k\rho}E)\simeq \Sigma^{-1}H\mathcal{F}_1\upi_{-1}\Sigma^{-k\rho}E
=\Sigma^{-1}H\mathcal{F}_1 \upi_{k\rho-1}E \,.
\]
\end{proof}

The following is the most important definition for our purposes. 

\begin{defin}[{\cite[Definition~3.1]{HM17}}]\label{def:strongly-even}
We say that $E$ is \emph{even} if 
\[\upi_{k \rho -1}E=0\] 
for all integers~$k$. 
We say that $E$ is \emph{strongly even} if
it is even and the restriction map 
\[
\res_e^{C_2}:\pi_{\rho k}^{C_2}E\to \pi_{2k}^eE
\] 
is an isomorphism, in other words $\upi_{k \rho }E$ is the constant Mackey functor $\m{\pi_{2k}^eE}$, 
for all integers~$k$. We say a Mackey functor is even if $H\mM$ is an even $C_2$-spectrum. 
\end{defin}

\begin{remark}\label{rem:slices}
If $E$ is strongly even in the sense of Definition~\eqref{def:strongly-even}, then 
\[ 
P_{\ast}^{\ast}E=\bigoplus \Sigma^{k\rho}  \m{\pi_{2k}^e E}
\]
by Proposition~\ref{prop:slices}. 
\end{remark}

\begin{remark}
There is an unfortunate clash of terminology here. 
We use the term strongly even in the sense of~\cite[Definition~3.1]{HM17} and not in the sense of~\cite[Definition~3.2.2]{HRW22}. In~\cite[Definition~3.9]{HP23}, they instead use the term \emph{very even} for what we call strongly even.
\end{remark}

The following lemma provides an alternative characterization of strongly even $C_2$-spectra. 

\begin{lem}[{\cite[Lemma~1.2~(ii)]{Gre18}}]\label{lem:criteria}
A $C_2$-spectrum $E$ is strongly even if and only if $\pi_{\rho k-i}^{C_2}E=0$ for $i=1,2,3$.
\end{lem}
\begin{warning}
Note that in~\cite[Lemma~1.2(i)]{Gre18},  Greenlees states that a $C_2$-spectrum $X$ is strongly even 
if $\pi_{\rho k-1}^{C_2}X=0$ and the restriction map  is an isomorphism. 
This definition is distinct from the definition of strongly even from~\cite[Definition~3.1]{HM17}, 
which we prefer to adhere to in this paper. We therefore specifically cite~\cite[Lemma~1.2(ii)]{Gre18} 
for Lemma~\ref{lem:criteria}. 
\end{warning}

\begin{exm}
The simplest example of a strongly even $C_2$-spectrum is the Eilenberg--MacLane spectrum $H\mF_2$, 
whose $\RO(C_2)$-graded homotopy groups are depicted in Figure~\ref{fig1}.
\end{exm}

\begin{figure}[ht!]
\resizebox{.5\textwidth}{!}{
\begin{tikzpicture}[radius=.08,scale=1]
\foreach \n in {-5,-4,-3,-2,-1,0,...,5} \node [below] at (\n+.1,-.1) {$\n$};
\foreach \s in {-5,-4,-3,-2,-1,1,2,...,5} \node [left] at (-.1,\s) {$\s$};
\draw [thin,color=lightgray] (-5,-5) grid (5,5);
\node [right] at (0,0) {1};
\node [right] at (0,-1) {$a_{\sigma}$};
\node [right] at (1,-1) {$u_{\sigma}$};
\node [right] at (-2,2) {$\theta$};


\foreach \n in {-5,-4,...,0} \draw [fill] (0,\n) circle;
\foreach \n in {-5,-4,...,-1} \draw [fill] (1,\n) circle;
\foreach \n in {-5,-4,...,-2} \draw [fill] (2,\n) circle;
\foreach \n in {-5,-4,-3} \draw [fill] (3,\n) circle;
\foreach \n in {-5,-4} \draw [fill] (4,\n) circle;
\foreach \n in {-5} \draw [fill] (5,\n) circle;

\draw (0,0) -- (0,-5);
\draw (0,0) -- (5,-5);
\draw (1,-1) -- (1,-5);
\draw (0,-1) -- (4,-5);
\draw (2,-2) -- (2,-5);
\draw (0,-2) -- (3,-5);
\draw (3,-3) -- (3,-5);
\draw (0,-3) -- (2,-5);
\draw (4,-4) -- (4,-5);


\foreach \n in {2,3,4,5} \draw [fill] (-2,\n) circle;
\foreach \n in {3,4,5} \draw [fill] (-3,\n) circle;
\foreach \n in {4,5} \draw [fill] (-4,\n) circle;
\foreach \n in {5} \draw [fill] (-5,\n) circle;

\draw (-2,2) -- (-5,5);
\draw (-2,3) -- (-4,5);
\draw (-2,4) -- (-3,5);
\draw (-2,5) -- (-2,5);
\draw (-2,2) -- (-2,5);
\draw (-3,3) -- (-3,5);
\draw (-4,4) -- (-4,5);


\draw [color=blue, dashed] (-5,-4) -- (4,5);
\draw [color=blue, dashed] (-5,-3) -- (3,5);
\draw [color=blue, dashed] (-5,-2) -- (2,5);

\end{tikzpicture}
}
\caption{
The $\RO(C_2)$-graded homotopy groups of $H\mF_2$. The group $\pi^{C_2}_{m+n\sigma}(H\mF_2)$ is displayed in bidegree $(m,n)$. 
Each bullet $\bullet$ represents a copy of $\bF_2$. The blue dashed lines indicate the gap, which implies that the 
$C_2$-spectrum $H\mF_2$ is strongly even.  
\label{fig1}
}
\end{figure}

\subsection{The strongly even filtration}\label{sec:strongly-even-filtration}
We first introduce notation. 
\begin{notation}
Let $\CAlg^{C_2}_{\sev}$ denote the full sub-$\infty$-category of $\CAlg^{C_2}$ spanned by the objects whose 
underlying object in $\Sp^{C_2}$ is strongly even. 
\end{notation}
We will need the following smallness result for our main construction. 
\begin{prop}\label{prop:acc1}
The sub-$\infty$-category $\CAlg^{C_2}_{\sev}$  of $\CAlg^{C_2}$ is accessible. 
\end{prop}

\begin{proof}
Let $\Ab^{\RO(C_2)}$ denote the $\infty$-category of $\RO(C_2)$-graded abelian groups and let $\Ab^{\RO(C_2)}_{\sev}$ 
denote the $\infty$-category of $\RO(C_2)$-graded abelian groups which vanish in gradings $\rho k-i$ for $i=1,2,3$ and $k$ 
any integer. 
We can identify $\CAlg^{C_2}_{\sev}$ with the pullback 
\[
\begin{tikzcd}
	 \CAlg^{C_2}_{\sev} \arrow{r}{\inc} \ar[d] & \CAlg^{C_2} \ar[d,"\pi_{\star}^{C_2}"] \\ 
     \Ab^{\RO(C_2)}_{\sev}\ar[r,hookrightarrow,"i"] & \Ab^{\RO(C_2)}
\end{tikzcd}
\]   
where the right vertical map is given by the associated graded of the regular slice filtration and the bottom map is the 
canonical inclusion. The functors $\pi_{\star}^{C_2}$ and $i$ preserve filtered colimits and consequently they are accessible.  
Therefore, the $\infty$-category~$\CAlg^{C_2}_{\sev}$ is accessible by~\cite[Proposition~5.4.6.6]{Lur09}. 
\end{proof}

\begin{defin}\label{def:sev-filt}
Let $A$ be a $C_2$-$\bE_\infty$-ring. We define the \emph{strongly even filtration} as the right Kan extension 
\[
\begin{tikzcd}
\CAlg^{C_2}_{\sev} \arrow[d,hookrightarrow] \arrow{r}{P_{2\bullet}} & \CAlg^{C_2}_{\fil} \\
\CAlg^{C_2} \arrow{ur}[swap]{\filsev^{\bullet}} & 
\end{tikzcd}
\]
and we write $\grsev^* := \gr^* \circ \filsev^{\bullet}$ for its associated graded. 
\end{defin}

\begin{remark}
Since the regular slice filtration is complete, the strongly even filtration is complete, 
i.e. for a $C_2$-$\bE_\infty$-ring $A$ we have $\lim_s\filsev^{s}A=0$. 
\end{remark}

\begin{remark}
If $R$ is a strongly even $C_2$-$\bE_{\infty}$-ring, then $\filsev^{\bullet}R=P_{2\bullet}R$. 
More generally, there is a natural map 
$P_{2\bullet}R\to \filsev^{\bullet}R$
inducing a map of $C_2$-$\bE_\infty$-rings
\[
R\longrightarrow \colim_{s}\filsev^{s}R\,.
\]
We provide a sufficient condition for this map to be an equivalence in Proposition~\ref{prop:convergence}. 
\end{remark}

\subsection{A $p$-complete variant}\label{sec:p-complete-strongly-even-filtration}
First, we discuss the notion of bounded $p$-power torsion in Mackey functors, which we need for our $p$-completed variant of 
the strongly even filtration.  

\begin{defin}
Given a Mackey functor $\m{M}$ we write $\m{M}[p^n]$ for the kernel of the map 
\[
\m{M} \cong \m{M} \square \m{A} \xrightarrow{\m{M}\square p^n} \m{M}\square \mA \cong \m{M}
\]
in the abelian category of Mackey functors. 
We say a Mackey functor has \emph{bounded $p$-power torsion} if the sequence 
\[ 
\m{M}[p]\to \m{M}[p^2]\to \m{M}[p^3] \to \dots 
\]
is eventually constant. 
\end{defin}

\begin{lem}\label{lem:acessible-p-power}
The full subcategory of Mackey functors with bounded $p$-power torsion is an $\aleph_1$-accessible sub-category of 
$p$-complete Mackey functors. 
\end{lem}

\begin{proof}
Let $\{\m{M}_{i}\}_{i\in I}$ be a $\aleph_{1}$-filtered diagram of Mackey functors such that each $M_{i}$ has 
bounded $p$-power torsion. 
For the sake of contradiction, suppose the colimit~$\m{M}$ of this diagram does not have bounded $p$-power torsion. 
Then for each positive integer~$n$, we can find a pair~$(i_{n},x)$ where $i_{n}\in I$ and $x : \mA \to \m{M}_{i_{n}}[p^{n}]$ 
such that there exists an extension
\[
\begin{tikzcd}
 \mA \ar[r,"x"] \ar[rrd] & \m{M}_{i_{n}}[p^{n}] \ar[r]  & \m{M}[p^{n}] \ar[d] \\ 
&&  \m{M}[p^{n}]/\m{M}[p^{n-1}]  \,.
\end{tikzcd}
\]
Since $I$ is $\aleph_{1}$-filtered, there exists an integer $j$ and maps $i_{n}\to j$ for each positive integer~$n$ 
so that the map $\mA\to \mM_{i_{n}}\to \mM_{j}[p^{n}]$ extends to a map 
\[
\mA\to \mM_{i_{n}}\to \mM_{j}[p^{n}]\to \mM_{j}[p^{n}]/\mM_{j}[p^{n-1}]\,.
\] 
This contradicts the assumption that $\m{M}_{j}$ has bounded $p$-power torsion. 
\end{proof}

\begin{defin} Fix a prime $p$. 
\begin{itemize}
\item Let $\cD(\m{A})_p :=\Mod(\Sp^{C_2}_p,\m{A})$ denote the $\infty$-category of $\m{A}$-modules, where $\m{A}$ is the 
Burnside Mackey functor. 
\item Let $\cD(\m{A})_{\bdd-p}$ denote the $\aleph_1$-accessible full sub-$\infty$-category of $\cD(\m{A})_p$ spanned 
by chain complexes of Mackey functors that have bounded $p$-power torsion level-wise. 
\item Let $\Sp_{\bdd-p}^{C_2}$ denote the pullback 
\[
\begin{tikzcd}
\Sp_{\bdd-p}^{C_2} \ar[r] \ar[d] & \Sp_{p}^{C_2} \ar[d,"P_{*}^{*}"] \\
\cD(\m{A})_{\bdd-p} \ar[r] & \cD(\m{A})_p 
\end{tikzcd}
\]
of $\infty$-categories. We say a $p$-complete $C_2$-spectrum has bounded $p$-power torsion if it is in $\Sp_{\bdd-p}^{C_2}$

\item Let  $\CAlg_{p}^{C_2}\subset \CAlg^{C_2}$
be the full sub-$\infty$-category of $C_2$-$\bE_{\infty}$-algebras whose underlying object in $\Sp^{C_2}$  is $p$-complete. 

\item Let  $\CAlg_{\fil,\,p}^{C_2}\subset \CAlg^{C_2}(\Fil(\Sp^{C_2}))$
be the full sub-$\infty$-category of $C_2$-$\bE_{\infty}$-algebras in $\Fil(\Sp^{C_2})$ whose underlying object in 
$\Fil(\Sp^{C_2})$ is $p$-complete. 

\item Let  $\CAlg_{\gr,\,p}^{C_2}\subset \CAlg^{C_2}(\Gr(\Sp^{C_2}))$
be the full sub-$\infty$-category of $C_2$-$\bE_{\infty}$-algebras in $\Gr(\Sp^{C_2})$ whose 
underlying object in $\Gr(\Sp^{C_2})$ is $p$-complete.

\item Let $\CAlg_{\sev,\,p}^{C_2}\subset \CAlg^{C_2}_{\sev}$
be the full sub-$\infty$-category spanned by the $C_2$-$\bE_{\infty}$-algebras 
whose underlying $C_2$-spectrum is $p$-complete and has bounded $p$-power torsion and there is a canonical 
inclusion~$\CAlg_{\sev,\,p}^{C_2}\subset \CAlg_{p}^{C_2}$
\end{itemize}
\end{defin}

\begin{prop}\label{prop:accessibile-Borel}
The sub-$\infty$-category 
$\CAlg_{\sev,\,p}^{C_2}$ of $ \CAlg_{p}^{C_2}$
is accessible.
\end{prop}

\begin{proof}
This follows from the same argument as the proof of Proposition~\ref{prop:acc1} using Lemma~\ref{lem:acessible-p-power}. 
\end{proof}

The proposition above allows us to make sense of the following definition. 

\begin{defin}
Let $R \in \CAlg_{p}^{C_2}$.
We define the \emph{$p$-complete strongly even filtration} to be the right Kan extension of $P_{2\bullet}$ along the 
canonical inclusion, 
\[
\begin{tikzcd}
\CAlg_{\sev,\,p}^{C_2} \arrow{r}{P_{2\bullet}} \arrow{d}[hookrightarrow]{\inc} & \CAlg^{C_2}_{\fil,\,p} \\ 
\CAlg_{p}^{C_2} \arrow{ur}[swap]{\fil_{\sev,p}^{\bullet}} & 
\end{tikzcd}
\]
and we write $\gr^*_{\sev,p} := \gr^* \circ \fil_{\sev,p}^\bullet$ for its associated graded. 
\end{defin}

\begin{remark}
If $R$ is a $p$-complete and strongly even $C_2$-$\bE_\infty$-algebra whose underlying $C_2$-spectrum lies in 
$\Sp_{\bdd-p}^{C_2}$, then $\fil_{\sev,p}^{\bullet}R=P_{2\bullet}R$.  
\end{remark}

\begin{notation}\label{not:deg-Aw-suspension}
Given a $\gr_{\sev,p}^*\bS_p$-module $M^*$ and an element  $x\in \pi_{V}^{C_2}M^{t}$, we define the \emph{degree} of 
$x$ to be $V$ and the \emph{Adams weight} of $x$ to be $t\rho-V$. We write $\| x \| = (V ,t \rho -V)$. 
Given a $\gr_{\sev,p}^*\bS_p$-module $M$, we write $\Sigma^{V,t}M$
for the shift by stem $V$ and Adams weight $t$ in other words 
\[ 
\pi_{a}(\Sigma^{V,t}M)^b=\pi_{W}M^s
\]
where $W=a-V$ and $2s-W=(2b-a)-t$. 
\end{notation}

\begin{exm}
Let $\kr$ denote connective Real K-theory (cf. Example~\ref{Ex:BPRn}). Consider the example
\[M^*:=\gr_{\sev,2}^*(\kr)_2=P_{2*}^{2*}(\kr)_2= \Sigma^{\rho *}H\m{\bZ}_2\,.\] 
Then $\overline{v}_1\in \pi_{\rho}^{C_2}M^1=\pi_{\rho}^{C_2}\Sigma^{\rho}H\mZ_2$ 
and $1\rho-\rho=0$, so $\|\ol{v}_1\|=(\rho,0)$.  
\end{exm}

\subsection{Descent}\label{sec:descent}
We start by introducing terminology following~\cite[\S 1.6(7)]{HRW22}.  

\begin{defin}
By a \emph{commutative graded ring}, we mean a commutative algebra in the $\infty$-category of graded abelian groups equipped 
with Day convolution symmetric monoidal structure. 
\end{defin}

\begin{remark}
This should be distinguished from the usual notion of a graded commutative ring, which is equipped with the Koszul sign 
convention. As a consequence, a commutative graded ring in this sense has an underlying commutative ring. 
\end{remark}

\begin{defin}
Let $A$ be a commutative graded ring and let $M$ be a graded $A$-module. We say that $M$ is faithfully flat over $A$ if the 
functor 
\[\bigoplus_{n\in \bZ}M^n\otimes_{\bigoplus_{n\in \bZ}A^n}\--\] 
is exact and faithful. 
\end{defin}

\begin{defin}
We call a map of $C_2$-$\bE_\infty$-rings $R\to S$ \emph{faithfully flat} if 
the map of commutative graded rings $\bigoplus_{n\in \bZ}\pi^e_{2n}R\to \bigoplus_{n\in \bZ}\pi^e_{2n}S$ is faithfully flat, 
i.e., the $\bigoplus_{n\in \bZ}\pi_{2n}^eR$-module $\bigoplus_{n\in \bZ}\pi_{2n}^eS$ is faithfully flat. 
We call a map of $p$-complete $C_2$-$\bE_{\infty}$-rings $R\to S$ \emph{$p$-completely faithfully flat} if the map of 
commutative rings $\bigoplus_{n\in \bZ}\pi_{2n}^eR\to \bigoplus_{n\in \bZ}\pi_{2n}^eS$ is $p$-completely faithfully flat 
in the sense of~\cite[Def. 2.2.4]{HRW22}. 
\end{defin}

\begin{defin}\label{def:seff}
A map of $C_2$-$\bE_\infty$-rings $A\to B$  is 
\begin{itemize}
    \item \emph{strongly evenly faithfully flat (seff)} if for any strongly even $C_2$-$\bE_\infty$-ring $C$ equipped with a 
    map of $C_2$-$\bE_\infty$-rings $A\to C$, the $C_2$-spectrum $C\otimes_{A}B$ is strongly even and faithfully flat over $C$. 
    \item \emph{$p$-completely seff} if for any $p$-complete strongly even $C_2$-$\mathbb{E}_{\infty}$-ring $C$ 
    with bounded $p$-power torsion equipped with a map of $C_2$-$\mathbb{E}_{\infty}$-rings $A\to C$, then $(C\otimes_{A}B)_p$ 
    is strongly even and $p$-completely faithfully flat over $C$. 
    \item \emph{pure} if for any non-zero strongly even $C_2$-$\bE_\infty$-ring $C$ equipped with a map of 
    $C_2$-$\bE_\infty$-rings $A\to C$, then $B\otimes_{A}C$ is equivalent to a wedge of $\rho$-divisible suspensions of $C$. 
    \item \emph{$p$-completely pure} if for any non-zero $p$-complete strongly even $C_2$-$\bE_\infty$-ring $C$ 
    equipped with a map of $C_2$-$\bE_\infty$-rings $A\to C$, then $(B\otimes_{A}C)_p$ is equivalent to the $p$-completion 
    of a wedge of $\rho$-divisible suspensions of $C$. 
\end{itemize}
\end{defin}
 
\begin{prop}\label{prop:seff-map-MUR}
The unit map $\bS \longrightarrow \MUR$ is pure.  
\end{prop}

\begin{proof}
Every even $C_2$-$\bE_{\infty}$-ring $C$ is Real oriented by~\cite[Lemma~2.3]{HM17} and consequently 
$\upi_\star (C\otimes_{\bS}\MUR) \cong \upi_{\star} C[\bar{b}_i : i\ge 1]$ by \cite[Theorem~2.25]{HK01} 
where $|b_i|=i\rho$. Choosing a basis as a $\upi_{\star}C$-module, we can produce a map from a wedge of $\rho$-divisible 
suspensions of $C$ that induces a $\upi_{\star}$-isomorphism and consequently is an equivalence as desired. 
\end{proof}

\begin{remark}
If $A\to B$ is ($p$-completely) seff in the sense of Definition~\ref{def:seff}, then $A^{e}\to B^{e}$ is ($p$-completely) eff 
and if $A\to B$ is ($p$-completely) pure, then $A^e\to B^e$ is ($p$-completely) evenly free in the 
sense of~\cite[Definition~2.2.13]{HRW22}. 
\end{remark}

\begin{remark}
If $A\to B$ is a seff map of $C_2$-$\bE_\infty$-rings, then $A\to B$ is $p$-completely seff. 
If $A\to B$ is strongly evenly pure, then $A\to B$ is $p$-completely strongly evenly pure. 
\end{remark}

\begin{defin}
We say that a sieve on $A$ in $(\CAlg^{C_2}_{\sev})^{\op}$ is a \emph{seff covering sieve} 
if it contains a finite collection of maps $\{ A\to B_i \}_{1\le i\le n}$ in $\CAlg^{C_2}_{\sev}$ 
such that $A \to \prod_i B_i$ is seff. 
We say that a sieve on $A$ in $(\CAlg_{\sev,\,p}^{C_2})^{\op}$ is a \emph{$p$-completely seff covering sieve} 
if it contains a finite collection of maps $\{ A\to B_i \}_{1\le i\le n}$ in $\CAlg^{C_2}_{\sev}$ such that 
$A \to \prod_i B_i$ is $p$-completely seff. 
\end{defin}

\begin{prop}\label{Prop:Topologies}
The following hold. 
\begin{enumerate}
\item The seff covering sieves form a Grothendieck topology on $(\CAlg^{C_2}_{\sev})^{\op}$. 
For a category $\cC$ admitting small limits, a functor 
\[ 
F : \CAlg^{C_2}_{\sev}\to \cC
\]
is a sheaf for this topology if it preserves finite products and for any seff map $A\to B$ in $\CAlg^{C_2}_{\sev}$ 
the canonical map 
\[ 
F(A)\to \lim_{q\in \Delta} F(B^{\otimes_{A}q+1})
\]
is an equivalence. 

\item The $p$-completely seff covering sieves form a Grothendieck topology on $(\CAlg^{C_2}_{\sev,\,p} )^{\op}$. 
For a category $\cC$ admitting small limits, a functor 
\[ 
F : \CAlg^{C_2}_{\sev,\,p}\to \cC
\]
is a sheaf for this topology if it preserves finite products and for any $p$-completely seff map $A\to B$ 
in $\CAlg_{C_2,\,p}^{\sev}$ the canonical map 
\[ 
F(A)\to \lim_{q\in\Delta} F(B^{\otimes_{A}q+1})
\]
is an equivalence. 
\end{enumerate}
\end{prop}

\begin{proof}
    It is elementary to check the conditions in the proofs of~\cite[A.3.2.1,~A.3.3.1]{Lur18} since pushouts in $\CAlg^{C_2}_{\sev}$ and $\CAlg_{C_2,\,p}^{\sev}$ along seff maps 
    exist and the only pushouts we need to check the conditions for are such pushouts. 
\end{proof}

\begin{defin}
    We call the Grothendieck topologies on $(\CAlg^{C_2}_{\sev})^{\op}$ and $(\CAlg_{\sev,\,p}^{C_2})^{\op}$ from 
    Proposition~\ref{Prop:Topologies} the \emph{seff topology} and the \emph{$p$-completely seff topology}, respectively. 
\end{defin}

\begin{prop}\label{prop:babysheaf}
The functor
  \[
\CAlg^{C_2}_{\sev}\overset{P_{\ast}^{\ast}}{\longrightarrow} \CAlg^{C_2}_{\gr}
  \]
    is a sheaf for the seff topology and the functor
  \[
    \CAlg_{\sev,\,p}^{C_2}\overset{P_{\ast}^{\ast}}{\longrightarrow} \CAlg_{\gr,\,p}^{C_2} 
  \]
    is a sheaf for the $p$-completely seff topology. 
\end{prop}

\begin{proof}
We will prove the $p$-completely seff case; the seff case follows \emph{mutatis mutandis}. 
By Lemma~\ref{lem:property-of-slices}, we know $P_*^*$ commutes with finite products since the 
inclusion $\CAlg^{C_2}_{\sev,\,p}\subset \CAlg_{p}^{C_2}$ commutes with finite products. 
Let $A\to B$ be $p$-completely seff. 
Consider the canonical map
\begin{equation}\label{can map for sheaf 1} 
P_{\ast}^{\ast}A\to \lim_{q\in \Delta}  P_{\ast}^{\ast} \left ( B^{\otimes_{A} q+1} \right )_p \,.
\end{equation}
Since $A\to B$ is $p$-completely seff, we know that $(B^{\otimes_{A} \bullet+1})_p$ is strongly even. 
Since $A$ is also strongly even, it suffices to check that the map 
\[ P_{*}^{*}A=\Sigma^{\rho *}H\underline{\pi_{2*}^eA}\to 
\lim_{q\in \Delta} \Sigma^{\rho *}H\underline{\pi_{2*}^e(B^{\otimes_{A} q +1}})_p 
=\lim_{q\in \Delta}P_{*}^{*}(B^{\otimes_{A} q+1})_p
\]
is an equivalence. It further suffices to check that 
\[ 
H\underline{\pi_{2k}^eA}\to \lim_{q\in \Delta} H\underline{\pi_{2k}^e(B^{\otimes_{A} q +1}})_p
\]
is an equivalence for each integer $k$. Since $A^e\to B^e$ is $p$-completely faithfully flat and 
$\pi_*$ is a sheaf for the flat topology in the sense of \cite[Proposition~2.2.8]{HRW22} by~\cite[Lemma~2.2.11]{HRW22}, 
we know that 
\[ 
\pi_{2k}^eA\cong\lim_{\Delta}\pi_{2k}^eB^{\otimes_A \bullet+1}
\]
so $R^1\lim_{q\in \Delta}
\pi_{2k}^eB^{\otimes_A q+1}= 0$ by the Milnor sequence. 
Since limits of Mackey functors are computed pointwise, the integer-graded homotopy Mackey functors satisfy
\begin{align}\label{eq:derived-vanishing}
R^1\lim_{q\in \Delta}
\upi_{*}H\underline{\pi_{2k}^e(B^{\otimes_A q+1}})= R^1\lim_{q\in \Delta}
\underline{\pi_{2k}^e(B^{\otimes_A q+1}})= 0\,. 
\end{align}
Together, this implies that the integer-graded homotopy Mackey functors satisfy
\begin{align*}
\upi_*\lim_{q\in \Delta} H\underline{\pi_{2k}^eB^{\otimes_A q+1}} 
&\cong \lim_{q\in \Delta} \upi_*H\underline{\pi_{2k}^e (B^{\otimes_A q+1}} )\\
& \cong \lim_{q\in \Delta} \underline{\pi_{2k}^e\left (B^{\otimes_A q+1} \right )}\\
&\cong\underline{\pi}_{2k}^eA \\
&=\upi_*H\underline{\pi_{2k}^eA}\,.
\end{align*}
Since integer-graded homotopy Mackey functors detect equivalences of $C_2$-spectra, this proves the claim. 
\end{proof}

\begin{thm}\label{thm:sheaf}
The following hold:
\begin{enumerate}
\item The functor
 \[
 \CAlg^{C_2}_{\sev}\overset{P_{2*}}{\longrightarrow} \CAlg^{C_2}_{\fil}
 \]
is a sheaf for the seff topology. 
\item The functor
 \[
 \CAlg^{C_2}_{\sev,\,p}\overset{P_{2*}}{\longrightarrow} \CAlg^{C_2}_{\fil,\,p}
 \]
is a sheaf for the $p$-completely seff topology. 
\end{enumerate}
\end{thm}

\begin{proof}
Since the filtration is complete, it suffices to check the statement on the associated graded, which follows from Proposition~\ref{prop:babysheaf}.
\end{proof}

\begin{thm}[Strongly evenly faithfully flat descent]\label{thm:descent}
The following hold:
\begin{enumerate}
\item \label{it1:thm:descent} If $A\to B$ is a seff map of $C_2$-$\bE_\infty$-rings, then the map 
\begin{align}\label{it1-display}
\filsev^{\bullet}A\to \lim_{q\in \Delta}\filsev^{\bullet}(B^{\otimes_A q+1})
\end{align}
is an equivalence. 
\item \label{it2:thm:descent} If $A\to B$ be a $p$-completely seff map of $C_2$-$\bE_\infty$-rings, then the map 
\[ 
\fil_{\sev,p}^{\bullet}A\to \lim_{q\in \Delta}\fil_{\sev,p}^{\bullet}(B^{\otimes_A q+1})_p
\]
is an equivalence.
\end{enumerate}
\end{thm}

\begin{proof}
We will prove \eqref{it1:thm:descent}; \eqref{it2:thm:descent} is proven in the same way. 
The left-hand side is
\[
\filsev^{\bullet}A = \lim_{\substack{A\to C \\ C\in \CAlg^{C_2}_{\sev}}}P_{2*} C
\]
by construction. Since $A\to B$ is seff, for any strongly even $C_2$-$\mathbb{E}_{\infty}$-algebra $C$ equipped with a 
map $A\to C$, we know $C\to C\otimes_A B$ is faithfully flat. By Proposition~\ref{prop:babysheaf}, we have
\[
P_{2*}C = \lim_{q\in \Delta} P_{2*} ((B\otimes_A C)^{\otimes_{C}q +1})= \lim_{q\in \Delta}P_{2*}(  B^{\otimes_{A}q +1}\otimes_A C)\,.
\]
Therefore, the left-hand side of \eqref{it1-display}, after commuting limits, is 
\[
\filsev^{\bullet}A = \lim_{q\in \Delta}\lim_{\substack{A\to C \\ C\in \CAlg^{C_2}_{\sev}}}P_{2*}( B^{\otimes_{A}q +1}\otimes_A C).
\]
The right-hand side of \eqref{it1-display}, by definition, is computed by 
\[
\lim_{q\in \Delta} \lim_{\substack{B^{\otimes_{A}q +1}\to D\\
D\in \CAlg^{C_2}_{\sev}}}P_{2*}(B^{\otimes_{A}q+1}\otimes_{B^{\otimes_{A}q+1}}D)
\simeq \lim_{q\in \Delta} \lim_{\substack{B^{\otimes_{A}q +1}\to D\\
D\in \CAlg^{C_2}_{\sev}}}P_{2*}D \,.
\]
We will show that 
\[
\lim_{
\substack{
A\to C \\ 
C\in \CAlg^{C_2}_{\sev}
}
} 
P_{2*} (B^{\otimes_{A}q+1}\otimes_A C)
=
\lim_{
\substack{
B^{\otimes_{A}q+1}\to D \\
D\in \CAlg^{C_2}_{\sev}}
} P_{2*}( D) \,.
\]
For each $A\to C\in \CAlg^{C_2}_{\sev}$, we obtain a map 
$B^{\otimes_A q+1} \simeq B^{\otimes_A q+1} \otimes_A A \to B^{\otimes_A q+1}\otimes_A C$ in $\CAlg^{C_2}_{\sev}$. 
Thus, it suffices to show that the functor 
\begin{equation}
\begin{aligned}
\mathscr C :=\{A\to C, C\in \CAlg^{C_2}_{\sev}\}&\to \{B^{\otimes_A q+1}\to D, D\in \CAlg^{C_2}_{\sev}
\}=:\mathscr D  \\
(A\to C)&\mapsto (B^{\otimes_A q+1}\to B^{\otimes_A q+1}\otimes_A C)
\end{aligned}
\end{equation}
is final. By Quillen's theorem A, it suffices to show that 
\[ 
\mathscr{C}\times_{\mathscr{D}}\mathscr{D}_{{/} X }
\]
is weakly contractible where $X=(B^{\otimes_A q+1}\to D)$ for some $D\in \CAlg^{C_2}_{\sev}$. 
This holds because this $\infty$-category has a terminal object 
$
Y=\{
A\to D, 
B^{\otimes_{A} q+1}\to B^{\otimes_A q+1}\otimes_AD \to D 
\}
$. 
\end{proof}

\begin{cor}[Real Novikov descent]
The following statements hold:
\begin{enumerate}
\item Let $A$ be a $C_2$-$\bE_\infty$-ring, then 
\[ \filsev^{\bullet}A\simeq \lim_{\Delta}\filsev^{\bullet}(A\otimes \MUR^{\otimes  \bullet+1})
\]
\item Let $A$ be a $p$-complete $C_2$-$\bE_\infty$-ring, then 
\[ \fil_{\sev,\,p}^{\bullet}A\simeq \lim_{\Delta}\fil_{\sev,\,p}^{\bullet}(A\otimes \MUR^{\otimes  \bullet+1})_p
\]
\end{enumerate}
\end{cor}
\begin{proof}
This directly follows from Theorem~\ref{thm:descent} and Proposition~\ref{prop:seff-map-MUR}. 
\end{proof}
\begin{exm}
We can identify
\[ \fil_{\sev,\,2}^{s}\bS_2\simeq \lim_{\Delta} P_{2s}((\MUR^{\otimes  \bullet+1})_2).
\]
This identifies  $\fil_{\sev,2}^{\bullet}\bS_2$ on objects; the structure maps 
$\fil^{s+1}_{\sev,\,2} \bS_2 \to \fil^s_{\sev,\,2} \bS_2$ are induced by the natural transformations 
$P_{2s+2} \to P_{2s}$. 
\end{exm}

\subsection{Convergence and $p$-completion}\label{sec:convergence}
We do not determine all cases where the strongly even filtration conditionally converges, but merely point out a special case. 

\begin{prop}\label{prop:convergence}
The following statements hold:
\begin{enumerate}
\item \label{it1:convergence} If there exists a seff map $A\to B$ of $C_2$-$\bE_\infty$-rings such that $P^0\fib(A\to B)=0$ and $B$ is strongly even, 
then the canonical maps
\[
A\to \lim_{\Delta} B^{\otimes_A\bullet +1} \quad \text{ and } \quad A\longrightarrow \colim_s\filsev^{s}A 
\]
are equivalences. 
\item \label{it2:convergence} If there exists a $p$-completely seff map $A\to B$ of  $C_2$-$\bE_\infty$-rings such that 
$P^0\fib(A\to B)=0$ and $B$ is strongly even, then the canonical maps
\[
A_p\to \lim_{\Delta} (B^{\otimes_A\bullet +1})_p \quad \text{ and } \quad  A_p\longrightarrow \colim_s\fil_{\sev,p}^{s}A_p
\] 
are equivalences. 
\end{enumerate}
\end{prop}

\begin{proof}
We first prove \eqref{it1:convergence}.
Set $I:=\fib (A\to B)$, and observe that since $P^0I=0$, then $P^{n}I^{\otimes_A n+1}=0$ by~\cite[Corollary~4.2]{Ull-thesis}. 
Consequently, we have $P_{n+1}(I^{\otimes_A n+1})\simeq I^{\otimes_A n+1}$. 
By~\cite[Proposition~2.14]{MNN17}, there is a fiber sequence
\[ 
I^{\otimes_A n+1}\to A\to \lim_{\Delta_{\le n}} B^{\otimes_{A}\bullet+1}
\]
and this allows us to conclude the first statement since $\lim_n I^{\otimes_A n+1}=0$. 
We further deduce that there is a fiber sequence
\[
\begin{tikzcd}
\lim_n P_{\bullet}(I^{\otimes_A n+1}) \longrightarrow  \lim_n  P_{\bullet}(A) \longrightarrow \lim_{\Delta_{\le n}}P_{\bullet}(B^{\otimes_{A}\bullet+1}) 
\end{tikzcd}
\]
of filtered spectra since 
$\lim_{\Delta_{\le n}}P_{\bullet}(B^{\otimes_{A}\bullet+1})\simeq P_{\bullet}(\lim_{\Delta_{\le n}}B^{\otimes_{A}\bullet+1})$. 
Since 
\[ 
\lim_{n} P_{s}(I^{\otimes_A n+1}) =\lim_{n} P_{\text{max}(s,n)}(I^{\otimes_A n+1}) = 0 
\]
the second claim follows. 
The $p$-complete statement ~\eqref{it2:convergence} follows from the same argument since $p$-completion commutes with limits. 
\end{proof}

\begin{remark}
Inspecting the previous proof, we see that analogous statements hold whenever $A\to B$ is seff and 
$\lim_n \fib(A\to B)^{\otimes_{A}n+1}\simeq 0$.
\end{remark}

We also note how the strongly even filtration interacts with $p$-completion.  

\begin{prop}
If $A$ is a  $C_2$-$\bE_\infty$-ring and there exists a seff map $A\to B$ where $B$ is strongly even and 
$B_p$ is strongly even and has bounded $p$-power torsion, then the canonical map 
\[ 
(\filsev^{\bullet}A)_p\simeq \fil_{\sev,\,p}^{\bullet}A_p
\]
is an equivalence. 
\end{prop}

\begin{proof}
This follows from Theorem~\ref{thm:descent} in light of the equivalence 
\[
\left ( P_{2*}(B^{\otimes_{A}q+1} ) \right)_p\simeq 
P_{2*}\left ( (B^{\otimes_{A}q+1})_p\right )
\]
which holds under our hypotheses. 
\end{proof}

\section{Real trace methods and Real motivic filtrations}\label{sec:motivic-filtrations}
In Section~\ref{sec:trace-background}, we introduce the theory of Real trace methods. 
In Section~\ref{sec:trace-motivic} we define motivic filtrations and in Section~\ref{sec:trace-Tate} 
we define the Nygaard filtration. In Section~\ref{sec:motivic-filtrationsB} we discuss descent for the motivic filtration. 
In Section~\ref{sec:RHMSS}, we define a computational tool for computing Real topological Hochschild homology and in 
Section~\ref{sec:suspension} we discuss the theory of equivariant suspension maps.

\subsection{Real trace methods}\label{sec:trace-background}

We now recall the theory of trace methods following~\cite{QS21a}. 
First we fix a choice of splitting of the extension 
\[ 
1\to \SO(2)\to \OO(2)\to C_2\to 1
\]
by fixing an isomorphism $\OO(2)\cong \SO(2)\rtimes C_2$ where
\[ 
	C_2=\left \{ \left ( \begin{matrix} 1 & 0 \\ 0 & 1\end{matrix} \right ) , \, 
    \left ( \begin{matrix}0 & 1 \\ 1 & 0\end{matrix}  \right ) \right \} \subset \OO(2)\,. 
\]

\begin{notation}
Let $\CAlg_{p}^{h_{C_2}S^1}$ denote the full sub-$\infty$-category of $C_2$-$\bE_\infty$-rings with 
$C_2$-twisted $S^1$-action (cf.~Definition~\ref{Def:Twisted}) whose  underlying $C_2$-spectra are $p$-complete. 
\end{notation}

\begin{defin}[{\cite[Definition~1.20]{QS21a}}]
A \emph{Real $p$-cyclotomic spectrum} is an object 
$X\in \Sp_p^{h_{C_2}S^1}$ together with a map
\[ 
\varphi_p : X\to X^{t_{C_2}\mu_p} 
\]
in $\Sp_p^{h_{C_2}S^1}$.
\end{defin}

\begin{notation}
 We let $\RCyc_p$ denote the $C_2$-symmetric monoidal $C_2$-$\infty$-category of 
 Real $p$-cyclotomic spectra and let $\CAlg^{\rcyc}_{p}:= \CAlg^{C_2}(\RCyc_{p})$. 
\end{notation}

\begin{remark}
To account for the difference in phrasing between~\cite[Definition~2.18]{QS21a} and the definition above, 
note that 
\[ 
\Sp_p^{h_{C_2}S^1}\simeq  \Sp_p^{h_{C_2}\mu_{p^{\infty}}}
\]
since we are working in the $p$-complete setting.
 Here we use notation from Section~\ref{sec:equivariant-Tate}. 
\end{remark}

\begin{defin}\label{def:tcr}
Given a Real $p$-cyclotomic spectrum $X$, we define the Real topological cyclic homology of $X$ to be the equalizer
\[
\TCR(X,\bZ_p):=\mathrm{eq} \left ( \can,\varphi_p^{h_{C_2}S^1}: \left ( X^{h_{C_2}S^1} \right )_p^\wedge \longrightarrow \left (
(X^{t_{C_2}\mu_p})^{h_{C_2}S^1}
\right )_p^\wedge \right)  \,. 
\]
Let $(\--)^{\triv}$ denote the right adjoint to $\mathrm{TCR} : \RCyc_p \longrightarrow \Sp^{C_2}$. 
\end{defin}

The main example of a Real $p$-cyclotomic spectrum, namely Real topological Hochschild homology, 
uses the notion of tensoring of a $C_2$-$\bE_\infty$-algebra with a $C_2$-space. 
See~\cite[Footnote~35 in Definition~5.2]{QS21a} for further details.  

\begin{notation}
We write
\[ 
\-- \odot \-- : \Top^{C_2}\times \CAlg^{C_2}\to \CAlg^{C_2}
\]
for the tensoring of a $C_2$-$\bE_{\infty}$-algebra with a $C_2$-space. 
\end{notation}

\begin{defin}[{\cite[Definition~5.2]{QS21a}}]
Let $B$ be a $C_2$-$\bE_\infty$-ring. The \emph{Real topological Hochschild homology of $B$} is 
\[
\THR(B) := S^{\sigma}\odot B \,.
\]
Let $\THR(B;\bZ_p):=\THR(B)_p^\wedge$ as an object in $\CAlg_{p}^{\rcyc}$.
\end{defin}

\begin{remark}
By~\cite[Remark~5.4]{QS21a}, this definition coincides as associative algebras in $\Sp^{C_2}$ with the definition of Real topological Hochschild homology via the dihedral bar construction as in~\cite{DMPR21}; see also ~\cite[Proposition~4.9]{AKGH25}.
\end{remark}

We will need the following relative version of Real topological Hochschild homology:

\begin{defin}
Given a $C_2$-$\bE_{\infty}$-ring $A$ and a $C_2$-$\bE_\infty$-$A$-algebra $B$, we define 
\[ 
\THR(B/A):=  \THR(B)\otimes_{\THR(A)}A 
\] 
and write $\THR(B/A;\bZ_p):= \THR(B/A)_p^\wedge$. 
\end{defin}

\begin{defin}
We define 
\begin{align*}
 \TCR^{+}(A/B;\bZ_p) & := \left ( \THR(A/B;\bZ_p)_{h_{C_2}S^1} \right )_p^{\wedge} \,, \\    
 \TCR^{-}(A/B;\bZ_p) & := \left ( \THR(A/B;\bZ_p)^{h_{C_2}S^1}\right )_p^{\wedge} \,, \\
 \TPR(A/B;\bZ_p) & := \left ( \THR(A/B;\bZ_p)^{t_{C_2}S^1} \right )_p^{\wedge} \,. 
\end{align*}
\end{defin}

\begin{remark}
When $B$ is a $C_2$-$\bE_{\infty}$-ring then there is a canonical map 
\[
\THR(B)\longrightarrow B^{\triv}
\] of $C_2$-$\bE_{\infty}$-algebras in Real $p$-cyclotomic spectra induced by the collapse map $S^\sigma\longrightarrow C_2/C_2$. 
\end{remark}

\begin{defin}\label{def:cyc-base}
A $C_2$-$\bE_\infty$-ring $B$ is a \emph{Real $p$-cyclotomic base} if the canonical map $\THR(B)\to B^{\triv}$ of $C_2$-$\bE_\infty$-rings is a map of Real $p$-cyclotomic spectra, where $B$ has trivial Real $p$-cyclotomic structure. Equivalently, $B$ is a Real $p$-cyclotomic base if there exists an extension
\[ 
    \begin{tikzcd}
   \THR(B)\ar[d] \ar[r,"\varphi"] & \THR(B)^{t_{C_2}\mu_p} \ar[r] &  B^{t_{C_2}\mu_p}\\
    B \ar[urr, dashed]& &
    \end{tikzcd}
\]
in the $\infty$-category $\CAlg^{h_{C_2}S^1}$. 
\end{defin}

\begin{remark}
If $B$ is a Real $p$-cyclotomic base, then $\THR(A/B)$ is Real $p$-cyclotomic. 
\end{remark}

\begin{construction}
When $\THR(A/B)^e$
is bounded below, then there is an equivalence 
\[ 
G : (\THR(A/B)^{t_{C_2}\mu_p})^{h_{C_2}S^1}\simeq \TPR(A/B;\mZ_p)
\]
by~\cite[Proposition~4.4]{QS21a}. We can then define 
\[ \varphi = G\circ (\varphi_p)^{h_{C_2}S^1} : \TCR^{-}(A/B;\mZ_p)\to \TPR(A/B;\mZ_p) 
\]
and define Real topological cyclic homology as in~\cite[Proposition~2.23]{QS21a}:
\[ 
\TCR(A/B;\bZ_p):= \eq \left (\can,\varphi:  \TCR^{-}(A/B;\bZ_p)\longrightarrow \TPR (A/B;\bZ_p) \right )
\,.
\]
We observe that $\TCR(A/B;\bZ_p)\simeq \TCR(\THR(A/B);\bZ_p)$ in the sense of Definition~\ref{def:tcr}. 
\end{construction}

\subsection{Real motivic filtrations}\label{sec:trace-motivic}
We begin by introducing some shorthand for our $\infty$-categories of interest. 

\begin{defin}
We define the following $\infty$-categories: 
\begin{itemize}
\item Let  $\CAlg_{\sev,\,p}^{h_{C_2}S^1} \subset \CAlg_p^{h_{C_2}S^1}$
denote the full sub-$\infty$-category spanned by the objects whose underlying $C_2$-spectrum is strongly even and has 
bounded $p$-power torsion. 
\item Let $\CAlg^{\rcyc}_{\sev,\,p} \subset \CAlg_{p}^{\rcyc}$
denote the full sub-$\infty$-category spanned by the objects whose underlying $C_2$-spectrum is strongly even and has 
bounded $p$-power torsion.  
\end{itemize}

\end{defin}

\begin{prop}
Each of the inclusions
\begin{align*}
\CAlg_{\sev,\,p}^{h_{C_2}S^1} & \subset \CAlg_p^{h_{C_2}S^1} \,,  \\
\CAlg^{\rcyc}_{\sev,\,p} & \subset \CAlg_{C_2,p}^{\rcyc}
\end{align*}
are inclusions of accessible sub-$\infty$-categories. 
\end{prop}

\begin{proof}
This follows from the same argument as Proposition~\ref{prop:acc1} using Lemma~\ref{lem:acessible-p-power}. 
\end{proof}

This allows us to make sense of the following definitions.

\begin{defin}
We define 
\[ 
\fil_{\sev,\,h_{C_2}S^1,p}^{\bullet}R \coloneqq 
\lim_{R\to B,B\in \CAlg_{\sev,\,p}^{h_{C_2}S^1}} P_{2\bullet}\left (B^{h_{C_{2}}S^{1}} \right )_p^\wedge
\]
where the limit is taken in $\CAlg_p^{h_{C_2}S^1}$. 

We define 
\[ 
\fil_{\sev,\,t_{C_2}S^1,p}^{\bullet}R\coloneqq 
\lim_{R\to B,B\in \CAlg_{\sev,\,p}^{h_{C_2}S^1}} P_{2\bullet}\left ( B^{t_{C_{2}}S^{1}} \right )_p^\wedge
\]
where the limit is taken in $\CAlg_p^{h_{C_2}S^1}$. 

We define 
\[ 
\fil_{\sev,\,\TCR}^{\bullet}R\coloneqq \lim_{R\to B,B\in\CAlg^{\rcyc}_{\sev,\,p}} 
\mathrm{eq} \left ( \can,\varphi :
\xymatrix{ 
P_{2\bullet}
\left ( 
B^{h_{C_{2}}S^{1}} \right )_{p}^\wedge
\longrightarrow P_{2\bullet} 
 \left ( B^{t_{C_{2}}S^{1}}\right )_{p}^\wedge
}  
\right ) 
\]
where the limit is taken in $\CAlg^{\rcyc}_{p}$. 
\end{defin}

\begin{defin}\label{def:bifiltered}
Let $R \in \CAlg_p^{h_{C_2}S^1}$. 

\begin{itemize}
\item We define a filtration on $\fil_{\sev,\,h_{C_2}S^1,\,p}^{\bullet}R$ on objects by 
\[ 
\fil^s_+\fil_{\sev,\,h_{C_2}S^1,\,p}^{\bullet}R \coloneqq \lim_{R\to B,B\in
 \CAlg_{\sev,\,p}^{h_{C_2}S^1}} P_{2\bullet}\left ( \left ((P_{2s}B)^{h_{C_{2}}S^{1}} \right )_p\right ) 
\]
and on morphisms in the obvious way.

\item We define a filtration on $\fil_{\sev,\,t_{C_2}S^1,\,p}^{\bullet}R$ on objects by 
\[ 
\fil^s_+\fil_{\sev,\,t_{C_2}S^1,\,p}^{\bullet}R \coloneqq \lim_{R\to B,B\in
 \CAlg_{\sev,\,p}^{h_{C_2}S^1}} P_{2\bullet}\left ( \left ((P_{2s}B)^{t_{C_{2}}S^{1}} \right )_p\right ) 
\]
and on morphisms in the obvious way.
\end{itemize}
This produces functors 
\begin{align*}
\fil_+^{\smallblacksquare}\fil_{\sev,\,h_{C_2}S^1,\,p}^{\bullet} &:\CAlg^{h_{C_2}S^1}_{p}
\longrightarrow \CAlg^{C_2}(\Fil(\Fil(\Sp^{C_2}))) \,, \\ 
\fil_+^{\smallblacksquare}\fil_{\sev,\,t_{C_2}S^1,\,p}^{\bullet} &:\CAlg^{h_{C_2}S^1}_{p}
\longrightarrow \CAlg^{C_2}(\Fil(\Fil(\Sp^{C_2})))  \,.
\end{align*}

\end{defin}

\begin{construction}
Let $R$ be $C_2$-$\bE_\infty$-ring. 
Define the \emph{Real motivic filtrations} on $\THR(R;\bZ_p)$, $\TCR^{-}(R;\bZ_p)$, $\TPR(R;\bZ_p)$, and $\TCR(R\bZ_p)$ 
as follows:
\begin{align*}
 	\filmot^{\bullet}\THR(R;\bZ_p)&\coloneqq \fil_{\sev,p}^{\bullet}\THR(R;\bZ_p) \,, \\
	\filmot^{\bullet}\TCR^{-}(R;\bZ_p)&\coloneqq  \fil_{\sev,\,h_{C_2}S^1,\,p}^{\bullet}\THR(R;\bZ_p) \,, \\
	 \filmot^{\bullet}\TPR(R;\bZ_p) &\coloneqq \fil_{\sev,\,t_{C_2}S^1,\,p}^{\bullet}\THR(R;\bZ_p) \,, \\
  \filmot^{\bullet}\TCR(R;\bZ_p) &\coloneqq \fil_{\sev,\TCR}^{\bullet}\THR(R;\bZ_p)\,.
 \end{align*}
Moreover, for  
$\cF\in \{\THR(\--;\bZ_p)\,,
\TCR^{-}(\--;\bZ_p)\,,\TPR(\--;\bZ_p)\,,\TCR(\--;\bZ_p)\}$, 
we write 
\[ 
\grmot^{\ast}\cF(R):= \filmot^{\ast}\mathcal{F}(R)/\filmot^{\ast+1}\mathcal{F}(R)\,.
\]
\end{construction}

\begin{remark}
    We define
    \[
\filmot^{\bullet}\THR(R;\bZ_p)^{t_{C_2}\mu_p} 
    \]
    to be the pushout of the diagram 
\[
\begin{tikzcd}
\filmot^{\bullet}\TCR^{-}(R;\bZ_p) \ar[r] \ar[d] & \filmot^{\bullet}\TPR(R;\bZ_p)   \\
\filmot^{\bullet}\THH(R;\bZ_p)  &    
\end{tikzcd}
\]
    for filtered $C_2$-$\mathbb{E}_\infty$-rings, 
    and we set
    \[
    \gr^*_{\mot} \THR(R;\bZ_p)^{t_{C_2}\mu_p} := 
    \fil^*_{\mot} \THR(R;\bZ_p)^{t_{C_2}\mu_p} / \fil^{*+1}_{\mot} \THR(R;\bZ_p)^{t_{C_2}\mu_p}\,.
    \]
\end{remark}

The following construction will be the key input for Section~\ref{sec:trace-Tate}, which is needed for our computations. 

\begin{construction}\label{const:Nygaard-filtration}
We further define filtrations
\begin{align*}
\fil^{\smallblacksquare}_{\Nyg}\grmot^*\TPR(R;\bZ_p):=
\fil_+^{\smallblacksquare}\fil_{\sev,\,t_{C_2}S^1,\,p}^{*}\THR(R;\bZ_p)/\fil_+^{\smallblacksquare}\fil_{\sev,\,t_{C_2}S^1,\,p}^{*+1}\THR(R;\bZ_p) \,, \\
\fil^{\smallblacksquare}_{\Nyg}\grmot^*\TCR^{-}(R;\bZ_p):=
\fil_+^{\smallblacksquare}\fil_{\sev,\,h_{C_2}S^1,\,p}^{*}\THR(R;\bZ_p)/\fil_+^{\smallblacksquare}\fil_{\sev,\,h_{C_2}S^1,\,p}^{*+1}\THR(R;\bZ_p) \,, \\
\fil_{\loc}^{\smallblacksquare}\grmot^*\TPR(R;\bZ_p):=
\fil_+^{\smallblacksquare}\fil_{\sev,\,h_{C_2}S^1,\,p}^{*}\THR(R;\bZ_p)^{t_{C_2}\mu_p}/\fil^{\smallblacksquare}_+\fil_{\sev,\,h_{C_2}S^1,\,p}^{*+1}\THR(R;\bZ_p)^{t_{C_2}\mu_p}
\end{align*}
on 
\[ 
\grmot^*\TPR(R;\bZ_p)\,, \quad  \grmot^*\TCR^{-}(R;\bZ_p)\, \quad \text{ and } \quad \grmot^*\TPR(R;\bZ_p) \,,
\]
respectively, where for the filtration $\fil_{\loc}^{\smallblacksquare}\grmot^*\TPR(R;\bZ_p)$, 
we additionally assume that $\THH(R^e;\bZ_p)$ is bounded below. 
We also write 
\begin{align*}
\gr^s_{\Nyg}\grmot^*\TPR(R;\bZ_p):=\fil_{\Nyg}^{s}\grmot^*\TPR(R;\bZ_p)/\fil_{\Nyg}^{s+1}\grmot^*\TPR(R;\bZ_p) \,, \\
\gr_{\Nyg}^{s}\grmot^*\TCR^{-}(R;\bZ_p):= \fil^{s}_{\Nyg}\grmot^*\TCR^{-}(R;\bZ_p)/\fil_{\Nyg}^{s+1}\grmot^*\TCR^{-}(R;\bZ_p) \,, \\
\gr_{\loc}^{s}\grmot^*\TPR(R):= \fil_{\loc}^{s}\grmot^*\TPR(R)/\fil_{\loc}^{s+1}\grmot^*\TPR(R)
\end{align*}
for the associated graded of these filtrations at an integer $s$. 
\end{construction}

\begin{defin}
For 
\[
\cF\in \{\THR(\--;\bZ_p),\THR(\--;\bZ_p)^{t_{C_2}\mu_2},\TCR^{-}(\--;\bZ_p),\TPR(\--;\bZ_p),\TCR(\--;\bZ_p)\}\,,
\] 
the \emph{motivic spectral sequence} is the spectral sequence associated to the filtered $C_2$-spectrum 
$\filmot^{\bullet}\cF(R)$. This spectral sequence has signature 
\[ 
\EE_2^{V,t\rho-V}=\pi_{V}^{C_2}\grmot^t\cF(R) \implies \pi_{V}^{C_2}\colim_{s}\filmot^{s}\cF(R)\
\]
with differential convention 
\[ 
d_r :\EE_r^{V,t\rho-V}\to \EE_r^{V-1,t\rho-V+r} \,.
\]
Since the slice filtration is exhaustive, this spectral sequence conditionally converges. 
\end{defin}

\begin{remark}
When there exists a seff map $\THR(A)\to B$ in $\CAlg_{p}^{h_{C_2}S^1}$  where 
\[P^0\fib(\THR(A)\to B)=0\,,\] 
then similar considerations to Proposition~\ref{prop:convergence} and the proof of~\cite[Theorem~3.18]{Kee25} 
imply that the motivic spectral sequence conditionally converges to 
\[ 
\pi_{\star}^{C_2}\cF(R)
\]
as desired for 
$\cF\in \{\THR(\--;\bZ_p),\THR(\--;\bZ_p)^{t_{C_2}\mu_2},\TCR^{-}(\--;\bZ_p),\TPR(\--;\bZ_p),\TCR(\--;\bZ_p)\}$.
\end{remark}

\begin{remark}
The motivic spectral sequence strongly converges whenever the $\mathrm{E}_2$-page is a finite type $\pi_{\star}^{C_2}\bF_2$-module which is finitely generated as a $\pi_{\star}^{C_2}\bF_2$-algebra. This will be the case in all of our examples of interest.
\end{remark}

\subsection{Amitsur--Dress--Tate cohomology}\label{sec:trace-Tate}
Recall from~\cite{Ara79} and~\cite[Theorem~2.10]{HK01}, that the $C_2$-equivariant Bredon cohomology 
with coefficients in a Mackey functor $\mM$ of $\mathbb{C}P_{\bR}^{\infty}$ is given by
\[
\mrH^\star_{C_2} (\mathbb{C}P_{\bR}^{\infty}; \mM) \cong \upi_\star \mM[\ol{t}] 
\]
where $|t| = -\rho$. Here we write $\mathbb{C}P_{\bR}^{\infty}$ for 
$\mathbb{C}P^\infty$ with complex conjugation action, see Example~\ref{exm:parametrized-CPinfty}. 

Amitsur--Dress and Amitsur--Dress--Tate cohomology are the parametrized analogues of group cohomology and Tate cohomology, 
respectively; see~\cite[\S~21]{GM95} and~\cite[\S~3.2]{MNN17} for definitions. 
The Bredon cohomology computation above yields the Amistur--Dress(--Tate) cohomology computations
\begin{align*}
\widehat{\mrH}^{*}_{\AD}(\mathbb{C}P_{\bR}^{\infty}; \pi_{\star}^{C_2}\mM)&=\upi_{\star}^{C_2}\mM[\ol{t},\ol{t}^{-1}] \,,  \\
\mrH^{*}_{\AD}(\mathbb{C}P_{\bR}^{\infty}; \pi_{\star}^{C_2}\mM)&= \upi_{\star}^{C_2}\mM[\ol{t}] \,, 
\end{align*}
where $|t|=-\rho$. 
For $X\in \CAlg^{h_{C_2}S^1}$, we obtain multiplicative spectral sequences
\begin{align}
\label{prismaticSS}
\widehat{\EE}_2^{\star,*,*} 
:=\widehat{\mrH}^{*}_{\AD}(\mathbb{C}P_{\bR}^{\infty},\upi_{\star}X)& \implies  \upi_{\star}X^{t_{C_2}S^1} \,, \\
\label{TC-SS}
\EE_2^{\star,*,*}  :=\mrH_{\AD}^*(\mathbb{C}P_{\bR}^{\infty},\upi_{\star}X) & \implies  \upi_{\star}X^{h_{C_2}S^1} 
\end{align}
that we call the \emph{parametrized Tate spectral sequence} and the \emph{parametrized homotopy fixed point spectral sequence}, 
respectively; see~\cite[Remark~4.32]{QS21a} for the identification of these spectral sequences with certain 
spectral sequences studied by Greenlees and May~\cite{GM95}. 

For a $C_2$-$\bE_{\infty}$-ring $R$ such that $\THH(R^e)$ is bounded below and $\grmot^w\THR(R)$ is 
even for each $w$ after an appropriate suspension, the filtrations from~Construction~\ref{const:Nygaard-filtration} 
give rise to multiplicative spectral sequences
\begin{align}
\label{prismaticSS}
\widehat{\EE}_2^{\star,*,*} :=\pi_{\star}^{C_2}\grmot^{*}\THR(R)[t,t^{-1}]& \implies  \pi_{\star}^{C_2}\grmot^{*}\TPR(R) \,, \\
\label{TC-SS}
\EE_2^{\star,*,*}  :=\pi_{\star}^{C_2}\grmot^{*}\THR(R)[t] & \implies  \pi_{\star}^{C_2}\grmot^{*}\TCR^{-}(R) \,, \\
\label{hS1-Tate-SS}
_{\textup{loc}}\EE_2^{\star,*,*}:=\pi_{\star}^{C_2}\grmot^{*}\THR(R)^{t_{C_2}\mu_2}[\bar{t}] & \implies   \pi_{\star}^{C_2}\grmot^{*}\TPR(R) \,,
\end{align}
which we refer to as the \emph{periodic $\overline{t}$-Bockstein}, \emph{$\overline{t}$-Bockstein}, 
and \emph{localized $\overline{t}$-Bockstein}, respectively.\footnote{It is perhaps more appropriate to call 
this the twisted (periodic/localized) $a_{\lambda}$-Bockstein spectral sequence where $a_{\lambda}$ is the 
Euler class of the standard $S^1$-equivariant representation, but we nevertheless adhere to the terminolology 
from~\cite{HRW22}.}
These spectral sequences fit into a commutative diagram of spectral sequences
\begin{equation}\label{eq:diagram-of-ss}
\xymatrix{
\widehat{E}_2^{\star,*,*}  
\ar@{=>}[d]  & \ar[l] E_2^{\star,*,*} \ar[r] \ar@{=>}[d] & _{\textup{loc}}E_2^{\star,*,*} \ar@{=>}[d] \\
\pi_{\star}^{C_2}\grmot^{*}\TPR(R) & \ar[l]_{\textup{can}}
\pi_{\star}^{C_2}\grmot^{*}\TCR^{-}(R) \ar[r]^{\varphi} & 
\pi_{\star}^{C_2}\grmot^{*}\TPR(R) 
}
\end{equation}
with differential convention
\[ 
d_r :\mathrm{E}_r^{V,w,f}\to  \mathrm{E}_r^{V-1,w+1,f+r}  
\]
where $V$ is the \emph{degree}, $w$ is the \emph{Adams weight}, and $f$ is the \emph{Nygaard filtration}. 
For example, the trigrading of $\ol{t}$ is given by 
\begin{align*} 
\|\ol{t}\|&=(-\rho,0,1)\,,
\end{align*}
and if $x\in \pi_V\grmot^t\THR(R)$, then 
\begin{align*} 
\|x\|&=(V,t\rho -V,0)\,.
\end{align*}

Observe that there is a short exact sequence 
\[ 
0 \to \mathrm{Nyg}_{\ge 1}^h\to \pi_{\star}^{C_2}\grmot^{*}\TCR^{-}(A) \to  \mathrm{Nyg}_{=0}^h \to 0 
\]
where $\mathrm{Nyg}_{\ge 1}^h$ is the subgroup of elements in Nygaard filtration $\ge 1$ and $\mathrm{Nyg}_{=0}^h$ is
defined as the quotient by this subgroup. We also define a compatible exact sequence 
\[ 
0 \to \mathrm{Nyg}_{\ge 1}^t\to \pi_{\star}^{C_2}\grmot^{*}\TPR(A) \to  \mathrm{Nyg}_{\le 0}^t \to 0 
\]
where $\mathrm{Nyg}_{\ge 1}^t=\mathrm{can}(\mathrm{Nyg}_{\ge 1})$ consists of elements in Nygaard filtration $\ge 1$ 
and $\mathrm{Nyg}_{\le 0}^t$ denotes the quotient. We further consider an exact sequence 
\[ 
0 \to \mathrm{Nyg}_{=0}^t \to \mathrm{Nyg}_{\le 0}^t\to \mathrm{Nyg}_{\le 1}^t \to 0 
\]
where $\mathrm{Nyg}_{=0}^t$ consists of elements in $\pi_{\star}^{C_2}\gr_{\mot}^*\TCR^{-}(R)$ in 
Nygaard filtration exactly zero.

\subsection{Descent for motivic filtrations}\label{sec:motivic-filtrationsB}
In this section, we prove descent for the motivic filtration.  

\begin{defin}
We say a map $A\to B$ in  
$\CAlg^{h_{C_2}S^1}_{\sev,\,p}$ and  
 $\CAlg^{\rcyc}_{\sev,\,p}$ is \emph{$p$-completely seff} if the underlying map in $\CAlg^{C_2}_{\sev,\,p}$ is seff. 
\end{defin}
\begin{defin}
 We say that a sieve on $A$ in 
$(\CAlg^{h_{C_s}S^1}_{\sev,\,p})^{\op}$ and  
 $(\CAlg^{\rcyc}_{\sev,\,p})^{\op}$ is a \emph{$p$-completely seff covering sieve} if it contains a finite 
 collection of maps $\{A\to B_i\}_{1\le i\le n}$ in  
$(\CAlg^{h_{C_2}S^1}_{\sev,\,p})$ and  
 $(\CAlg^{\rcyc}_{\sev,\,p})$ such that $A\longrightarrow \prod_{i}B_i$ is seff (resp. $p$-completely seff). 
\end{defin}

\begin{proposition}
The $p$-completely seff covering sieves generate a Grothendieck topology on the $\infty$-categories $(\CAlg^{h_{C_2}S^1}_{\sev,\,p})^{\op}$ 
and $(\CAlg^{\rcyc}_{\sev,\,p})^{\op}$.
\end{proposition}
\begin{proof}
Since pushouts in 
$\CAlg^{h_{C_2}S^1}_{\sev,\,p}$ and  
 $\CAlg^{\rcyc}_{\sev,\,p}$ are computed in $\CAlg^{\sev}_p$, the same argument as Proposition~\ref{Prop:Topologies} 
 produces the the $p$-completely seff topology. 
\end{proof}

\begin{thm}\label{thm:sheaf}
The following statements hold. 
\begin{enumerate}
\item \label{it1:sheaf} The functors $P_{2\ast}\left ((\--)^{h_{C_2}S^1}\right )$ and 
$P_{2\ast}\left ((\--)^{t_{C_2}S^1}\right)$ are sheaves for the $p$-completely seff topology 
on $\CAlg^{h_{C_2}S^1}_{\sev,\,p}$ and the $p$-completely seff topology on  $\CAlg^{\rcyc,\,\sev}_{p}$. 
\item The functors $P_{2\ast}\left ((P_{2\ast}\--)^{h_{C_2}S^1}\right )$ and 
$P_{2\ast}\left ((P_{2\ast}\--)^{t_{C_2}S^1}\right)$ are sheaves for the $p$-completely seff topology 
on $\CAlg^{h_{C_2}S^1}_{\sev,\,p}$. 
\item The functor 
\[ 
\mathrm{eq} \left ( \varphi,\can : P_{2\bullet}\left ( (\--)^{h_{C_2}S^1} \right ) \longrightarrow P_{2\bullet}\left ( (\--)^{t_{C_2}S^1} \right ) \right )
\]
is a sheaf for the $p$-completely seff topology on $\CAlg^{\rcyc,\,\sev}_{p}$.
\end{enumerate}
\end{thm}

\begin{proof}\
\begin{enumerate}
\item By passing to associated graded and using the fact that these filtrations are complete, 
it suffices to prove that the functors 
\[
P_*^*((\--)^{h_{C_2}S^1}) \quad \text{ and } \quad P_*^*((\--)^{t_{C_2}S^1})
\]
are sheaves for the $p$-completely seff topology. Given a $p$-completely seff map $A\to B$ in $\CAlg_{p}^{h,\,\sev}$, 
it suffices to prove that the canonical maps
\begin{align*}
P_*^*(A^{h_{C_2}S^1})&\to \lim_{\Delta} P_*^*((B^{\otimes_{A}\bullet+1})_p^{h_{C_2}S^1}) \,, \\ 
P_*^*(A^{t_{C_2}S^1})&\to \lim_{\Delta} P_*^*((B^{\otimes_{A}\bullet+1})_p^{t_{C_2}S^1})
\end{align*}
are equivalences. 
Since $A$ and $(B^{\otimes_{A}q+1})_p$ are strongly even, it suffices to show that the canonical maps
\begin{align*}
\Sigma^{\rho n}H\m{\pi_{2n}(A^{hS^1})}&\to \lim_{\Delta} \Sigma^{\rho n}\m{\pi_{2n}((B^{\otimes_{A}\bullet+1})_p^{hS^1})} \,, \\ 
\Sigma^{\rho n}\m{\pi_{2n}(A^{tS^1})}&\to \lim_{\Delta} \Sigma^{n\rho}\m{\pi_{2n} ((B^{\otimes_{A}\bullet+1})_p^{tS^1})}
\end{align*}
are equivalences, but this follows from~\cite[Lemma~2.2.13(b)]{HRW22} since limits of Mackey functors are computed point-wise. 
The same argument applies in the category of Real $p$-cyclotomic spectra. 
\item The argument is essentially the same as in the previous case. 
\item This follows from \eqref{it1:sheaf}. 
\end{enumerate}
\end{proof}

\begin{cor}\label{thm:motivic-descent}
The following statements hold.
\begin{enumerate}
\item Let $\THR(A;\bZ_p) \longrightarrow B$ be a $p$-completely seff map in $\CAlg_{p}^{h_{C_2}S^1}$. 
Then 
\begin{align*}
    \fil_{\mot}^*\TCR^{-}(A;\bZ_p)& \simeq \lim_{\Delta} \fil_{\sev,h_{C_2}S^1,p}^*((B^{\otimes_{\THR(A;\bZ_p)}\bullet+1})^{h_{C_2}S^1})_p \,, \\ 
    \fil_{\mot}^*\TPR(A;\bZ_p)& \simeq \lim_{\Delta}\fil_{\sev,t_{C_2}S^1,p}^*((B^{\otimes_{\THR(A;\bZ_p)}\bullet+1})^{t_{C_2}S^1})_p\,.
\end{align*}
\item Let $\THR(A;\bZ_p)\to B$ be a $p$-completely seff map in $\CAlg_{p}^{\rcyc}$. 
Then 
\[
\filmot^*\TCR(A;\bZ_p) \simeq 
\mathrm{eq} \left (\can,\varphi : \filmot^*\TCR^{-}(A;\bZ_p) \longrightarrow \filmot^*\TPR(A;\bZ_p) \right ) \,.
\]
\end{enumerate}
\end{cor}

\begin{proof}
The proof is entirely analogous to the proof of Theorem~\ref{thm:descent}. 
\end{proof}

\begin{cor}\label{thm:can-varphi}
If $A$ is a $C_2$-$\bE_\infty$-algebra in Real $p$-cyclotomic spectra that admits a $p$-completely seff cover $A\to B$ with 
$B\in \CAlg^{\rcyc}_{\sev,p}$, then there are filtered maps 
\[
\can,\varphi : \fil_{\sev,h_{C_2}S^1,p}^*A\longrightarrow \fil_{\sev,t_{C_2}S^1,p}^*A
\]
that converge to $\can$ and $\varphi$ as in~\cite{QS21a} and we can identify 
\[ 
\fil_{\sev,\TCR}^*A\simeq \mathrm{eq}\left (\can,\varphi:\fil_{\sev,h_{C_2}S^1,p}^*A
\longrightarrow \fil_{\sev,t_{C_2}S^1,p}^*A \right )
\]
\end{cor}
\begin{proof}
Let $G$ and $H$ be the right Kan extensions of the functors 
\[ P_{2*}((\--)^{h_{C_2}S^1}) \,, P_{2*}((\--)^{t_{C_2}S^1}) :  \CAlg^{\rcyc}_{\sev,p}\longrightarrow \CAlg^{C_2}_{\fil,p} \, , \]
respectively. Then consider the natural transformations $P_{2*}\can$ and $P_{2*}\varphi$, 
which right Kan extend to functors $\can,\varphi : G\to H$. Observe that by definition 
\[ 
\fil_{\sev,\TCR}^{*}A=\eq (\can,\varphi : G\to H)\,.
\]
By construction there are natural transformations 
\[ 
G(A)\to \fil_{\sev,h_{C_2}S^1,p}^*A \,,H(A)\to \fil_{\sev,t_{C_2}S^1,p}^*A
\]
and by Theorem~\ref{thm:sheaf}, there are equivalences 
\[ 
G(A) \simeq \lim_{\Delta} G(B^{\otimes_{A}\bullet+1})
\simeq \lim_{\Delta} \fil_{\sev,h_{C_2}S^1,p}^*(B^{\otimes_{A}\bullet+1})\simeq \fil_{\sev,h_{C_2}S^1,p}^*A
\,, \]
\[ H(A) \simeq \lim_{\Delta} H(B^{\otimes_{A}\bullet+1})\simeq 
\lim_{\Delta} \fil_{\sev,t_{C_2}S^1,p}^*(B^{\otimes_{A}\bullet+1})\simeq \fil_{\sev,t_{C_2}S^1,p}^*A \,.
\]
\end{proof}

\subsection{A Real Hochshild--May spectral sequence}\label{sec:RHMSS}

By \cite[Thm. 1.2.6]{Yan25}, the $C_2$-$\infty$-category $\Fil(\Sp^{C_2})$ admits a canonical $C_2$-symmetric monoidal 
structure given by parametrized Day convolution. Let 
\[
I\in \CAlg^{{C_2}}(\Fil(\Sp^{C_2}))
\] 
be a $C_2$-$\bE_{\infty}$-algebra in filtered $C_2$-spectra.

If $X$ is a $C_2$-space, we can form the tensoring $X\odot  I $ and it is a $C_2$-$\bE_{\infty}$-algebra.

\begin{lem}
Given $I$ and $X$ as above, there is an equivalence
\[ \gr^{*}(X\odot I) \simeq X\odot   \gr^{*}I\]
of graded $C_2$-$\bE_{\infty}$-algebras. 
\end{lem}

\begin{proof}
The associated graded $\gr^*(X \odot I)$ is a graded $C_2$-$\bE_\infty$-algebra by \cite[Thm. 1.2.6]{Yan25}. 
After forgetting the $C_2$-action, this is proven in \cite[Theorem~3.3.10]{AS18} in the $\bE_{\infty}$-context. 
Since the functor $\gr^*$ commutes with sifted colimits and direct sums, it suffices to prove the result for $X=C_2/H$, 
or in other words it suffices to prove that 
\[ 
\gr^*N_e^{C_2}i_e^*I\simeq N_e^{C_2}i_e^*\gr^*I\,,
\]
but this follows from the fact that $\gr^*$ is $C_2$-monoidal~\cite[Proposition~3.1.20]{Yan25}. 
\end{proof}

\begin{construction}
When $I\in \CAlg^{C_2}(\Fil(\Sp^{C_2}))$ and  $\lim_{k}(X\odot  I)(k)=\lim_{k}I(k)=0$ and $\colim_{k}I(k)=I_{\infty}$, 
then we produce a conditionally convergent spectral sequence 
\[ \pi_{\star} (X\odot \gr^{*}I) \implies \pi_{\star}(X\odot I_{\infty}) \]
associated to filtration $X\otimes I$.  
\end{construction}

\begin{remark}
When $G=e$ and $X=S^{1}$, then this appeared in earlier work of \cite{Bru01} in the setting of functors with smash product.  
When $G=e$, and $X$ is a more general simplicial set, the spectral sequence was first constructed by~\cite{AS18}. 
Similar constructions also appear in~\cite{Kee25,AMN22,LL23}. 
\end{remark}

\begin{defin}
Let $G=C_{2}$, let $X=S^{\sigma}$, and let $I\in \CAlg^{C_2}(\Fil(\Sp^{C_{2}}))$. 
We call the resulting spectral sequence 
\[ 
\upi_{\star}\THR(\gr^{*}I)\implies \upi_{\star}\THR(I_{0}) 
\]
the \emph{Real Hochschild--May spectral sequence}.
\end{defin}

\begin{defin}
By a \emph{tower functor}, we mean a functor 
\[ 
T : \Sp^{C_2}\to \Fil(\Sp^{C_2}) 
\]
that restricts to a functor 
\[ 
T : \CAlg^{C_2}(\Sp^{C_2})\to \CAlg^{C_2}(\Fil(\Sp^{C_2})) \,.
\]
If $\colim_{k}T(R)=R$ and $\lim_{k}T(R)=0$ then we call $T$ a \emph{convergent} tower functor. 
\end{defin}

\begin{exm}
The slice filtration $P_{\bullet}$ is a convergent tower functor. 
\end{exm}

\begin{exm}
Given a convergent tower functor 
\[ 
T : \CAlg^{C_2}\to \CAlg^{C_2}_{\fil}\,,
\] 
then there is a spectral sequence 
\[ 
\pi_{\star}^{C_2}\THR(\gr^{*}T(R))\implies \pi_{\star}^{C_2}\THR(R) 
\]
which conditionally converges.
\end{exm}

\begin{exm}
For any $C_{2}$-$\bE_{\infty}$-ring $R$, there is a Real Hochschild--May spectral sequence 
\[  \pi_{\star}^{C_2}\THR(P_{*}^{*}R)\implies \pi_{\star}^{C_2}\THR(R)\]
which conditionally converges.
\end{exm}

\subsection{Equivariant suspension maps}\label{sec:suspension}
Fix a finite group $G$ and an $\infty$-category $\cC$ tensored and cotensored over $G$-spaces. 
For each pair $X,Y\in \Top^{G}$ and $A\in \cC$, we have a map of $G$-spaces
\[ 
 \Map(X,Y)\to  \Map(X\odot  A,Y\odot  A).
\]
Taking $X=G/H$ and recalling that $G/H\odot A=N_{H}^{G}i_{H}^{*}A$, 
this produces a map of $G$-spaces
\begin{equation}\label{eq: unreduced suspension operator}
\Map(G/H,Y)_{+}\otimes N_H^{G}i_{H}^{*}A\longrightarrow Y\odot A.
\end{equation}
For example, when $\cC$ is a $C_2$-$\infty$-category and $B\in \CAlg^{G}(\mathcal{C})_{A\backslash}$, 
then we have a commutative diagram 
\begin{equation}\label{eq:front-face}
    \begin{tikzcd}
      \Map(G/H,Y)_+\otimes A \ar[r] \ar[d] &  A \ar[d] \\
     \Map(G/H,Y)_+\otimes{}_AN_{H}^Gi_H^*B \ar[r] & Y\odot (B/A) 
    \end{tikzcd}
\end{equation}
where ${}_AN_H^G$ is the $A$-relative norm functor of~\cite[Definition~2.20]{ABGHLM18} and $Y\odot (B/A)$ 
denotes the tensoring of $B$ with $X$ in $\mathrm{CAlg}_{G}(\cC)_{A/}$. Here, we use the facts that 
\[
G/H\odot (B/A)\simeq  {}_AN_{H}^Gi_H^*B \quad \text{ and } \quad Y\odot (A/A)\simeq A\simeq G/H\odot (A/A)\,.
\] 

\begin{defin}\label{def: suspension}
Let $G/H \to Y$ be a map in $\Top^{G}$. Let $R$ be a $C_{2}$-$\bE_{\infty}$-ring. 
Then since the map 
$\Map(G/H,Y)_+\otimes A \longrightarrow A$
 in the diagram \eqref{eq:front-face} splits, we have a commutative diagram 
\[ 
  \begin{tikzcd}
        \Map(G/H,Y)\otimes A \ar[r] \ar[d] &  0 \ar[d] \\
        \Map(G/H,Y)\otimes{}_{A}N_H^Gi_H^*B \ar[r] & Y \odot B/A
    \end{tikzcd}
\]
and consequently, writing ${}_{A}\overline{N_H^Gi_H^*B}=\mathrm{cof}  (A\to  {}_{A}N_H^Gi_H^*B)$, we define a map 
\[ 
\sigma^{M}_{H}\colon \Map(G/H,Y) \odot {}_{A}\overline{N_H^Gi_H^*B}  \longrightarrow Y\odot(A/R) \,.
\]
\end{defin}

\begin{warning}
The $\sigma$ in the notation $\sigma_{H}^{M}$ above 	is unrelated to the use of the symbol $\sigma$ as 
the sign representation. The former is used because it is the lower-case of the Greek letter sigma, 
whose upper-case $\Sigma$ is commonly used to denote the reduced suspension. 
\end{warning}

\begin{defin}\label{def:rho-suspension}
Let $Y=S^{\sigma}$. Let $k\in \mathrm{CAlg}^{C_2}$ and let $A$ be a $C_2$-$\bE_{\infty}$-$k$-algebra. 
Then we have a map $C_2/C_2 \to S^\sigma$ and an associated map
\[ 
    \sigma^{\mathbb{S}^{\sigma}}_{C_2}\colon \thinspace \Sigma^{\sigma}\overline{A}\to \mathrm{THR}(A/k)
\]
where $\overline{A}=\mathrm{cof} (k\to A)$. 
We write $\sigma^{\rho}$ 
for this map since if $y$ has non-trivial image 
\[
x\in \mathrm{im}(\upi_{\star}\overline{A}\to \upi_{\star}\Sigma k)=\mathrm{ker}(\upi_{\star}\Sigma k\to \upi_{\star}\Sigma A)
\]
then $|\sigma^{\rho}y|=|x|+\rho$.  
\end{defin}

\begin{remark}\label{rem:res-sigma}
It is clear from the construction that $\sigma^{\rho}$ from Definition~\ref{def:rho-suspension} restricts to $\sigma^2$ 
from~\cite[Example~A.2.4]{HW22}. 
\end{remark}

\begin{remark}
Note that there is canonical unit map $k\longrightarrow \THR(A/k)$ in $\CAlg^{h_{C_2}S^1}$ and this produces a 
unit map $k\longrightarrow \lim_{\mathbb{C}P_{\bR}^{k}}\THR(A/k)$ where $\mathbb{C}P_{\bR}^{k}$ is the $2k$-skeleton of 
the $C_2$-CW complex whose underlying space is $\mathbb{C}P^{k}$ with $C_2$-action by complex conjugation. 
\end{remark}

\begin{lem}\label{HWlemma}
There is a functorial commutative diagram 
\[
\begin{tikzcd}
\Sigma^{-1}\overline{A} \arrow{r} \arrow{d}{\Sigma^{-\rho}\sigma^{\rho}} & k \arrow{d} \\
\Sigma^{-\rho}\THR(A/k) \arrow{r}{\ol{t}} & \lim_{\mathbb{C}P_{\bR}^{1}}\THR(A/k) 
\end{tikzcd}
\] 
where $k\to \lim_{\mathbb{C}P_{\bR}^{1}}\THR(A/k)$ is the unit map. 
\end{lem}

\begin{proof}
As in the proof of~\cite[Lemma~A.4.1]{HW22}, 
we left Kan extend 
\[
\begin{tikzcd}
k\arrow{r} \arrow{d} & 0 \arrow{dd} \\ 
A \arrow{d} &   \\	
\mathrm{THR}(A) \arrow{r}{s^{\rho}} & \Sigma^{-\sigma}\THR(A)
\end{tikzcd}
\]
to the diagram 
\[
\begin{tikzcd}
k\arrow{drr} \arrow{r} & 0 \arrow{dd}  \arrow{drr} & & \\ 
& & A \arrow{dd} \arrow{r} &  \overline{A} \arrow{dd} \\	
& 0 \arrow{drr} & & \\ 
&  & \THR(A) \arrow{r} & \Sigma^{-\sigma}\THR(A)
\end{tikzcd}
\]
where 
$\overline{A} \to \Sigma^{-\sigma}\THR(A)$ is adjoint to $\sigma^{\rho}$. 
Further right Kan extension of the diagram above to a cube produces the map of fiber sequences 
\[
\begin{tikzcd}
k \arrow{r} \arrow{d} & A \arrow{r} \arrow{d} & \overline{A} \arrow{d} \\ 
\lim_{\mathbb{C}P_{\bR}^1}\THR(A) \arrow{r} & \THR(A) \arrow{r} & \Sigma^{-\sigma}\THR(A) 
\end{tikzcd}
\]
as well as the desired commuting diagram after taking horizontal fibers.  
\end{proof}

\begin{remark}
As in~\cite[Lemma~A.4.1]{HRW22}, a choice of suspension requires a choice of sign convention and we 
fix our choice of sign convention so that the lemma above holds. 
\end{remark}

\begin{remark}\label{rem:compatibility}
Lemma~\ref{HWlemma} is clearly compatible with~\cite[Lemma~A.4.1]{HW22} on underlying. 
\end{remark}

\section{Artin--Tate motivic homotopy theory and the strongly even filtration}\label{sec:realmotivic}
 The purpose of this section is twofold: to provide exposition for the reader interested in the connections to 
 real motivic homotopy theory and to provide a computational tool. 
 We do not use this tool in the proofs of the main theorems of the present work, but we expect it to have applications along the lines of~\cite{AKHW24}. 

\subsection{Artin--Tate $\mathbb{R}$-motivic homotopy theory}\label{sec:motivic}
Here we review work of Burklund--Hahn--Senger~\cite{BHS22}. 
Consider the filtered $C_2$-spectrum  
\[
\nu_{\bR}\bS_2 := \lim_{q \in \Delta} P_{2\bullet}(\MUR^{\otimes 
 q+1} )_2^\wedge  \,, 
\]
i.e. the slice d\'ecalage of the $C_2$-equivariant Adams--Novikov spectral sequence. 
Note that the associated graded of this filtration is not the $\EE_2$-page of the $C_2$-equivariant 
Adams--Novikov spectral sequence. Instead, it can be identified with 
\[ 
\pi_{\star}^{C_2}\gr^*\nu_{\bR}\bS_2 = \bigoplus \Ext_{\MU_*\MU}^{*,*}(\MU_*,\MU_*\otimes \pi_{\star}^{C_2}H\mZ_2) 
\]
by~\cite[Theorem~5.1]{BHS22}.
More generally, we write 
\[ 
\nu_{\bR}(X):=\lim_{q \in \Delta} P_{2\bullet}((X_2^\wedge 
\otimes \MUR^{\otimes q+1})_2^{\wedge}) 
\]
producing a functor 
\[ 
\nu_{\bR} : \Mod(\Sp^{C_2},\bS_2)\to \Fil(\Mod(\Sp^{C_2},\bS_2^\wedge))
\]
and we further write 
\[ 
\ol{\nu}_{\bR}(X):= \gr^*\nu_{\bR}(X) \,.
\]
In \emph{op. cit.}, it is also proven that 
\[ 
\Mod(\Fil(\Mod(\Sp^{C_2},\bS_2),\nu_{\bR}\bS)\simeq  \Mod(\SH^{\textup{AT}}(\bR),\bS_2) 
\]
where 
\[ 
\SH^{\AT}(\bR) \subset \SH(\bR)
\]
is the stable, full sub $\infty$-category of the stable $\infty$-category of stable  $\bR$-motivic 
homotopy theory closed under tensor products and colimits and generated by $\operatorname{Spec}(\bC)$, $\bG_m$, and $S^0$. 
We refer the reader to \cite{MV99} for the original definition of the $\bR$-motivic stable homotopy category 
and to~\cite[Sec. 4]{BH21} for the $\infty$-categorical analogue. 
In op.~cit., Burklund, Hahn, and Senger further produce a symmetric monoidal equivalence  
\begin{equation}\label{eq:cof-tau}
\Mod (\Gr(\Sp_{i2}^{C_2}),\ol{\nu}_{\bR}\bS) \simeq \Mod(\Sp^{C_2}_{i2},H\mZ_2)\otimes_{\Mod(\Sp,H\bZ)}\IndCoh(\cM_{\fg})   \,,
\end{equation}
where the subscript `$i2$' refers to modules over the $2$-completion of the unit and $\IndCoh(\cM_{\fg})$ is defined 
in~\cite[Definition~5.14]{BHS22}. For example, if $M$ is an $\MU_*\MU$-comodule concentrated in even degrees, 
then there is an associated sheaf $\mathcal{F}_{M}\in \mathrm{IndCoh}(\mathcal{M}_{\textup{fg}})$ and 
therefore a $\ol{\nu}_{\bR}\bS$-module associated to $\mathcal{F}_{M}\otimes \mZ $. 
We observe that this equivalence is sufficiently compatible with the $C_2$-equivariant structure to preserve 
$C_2$-$\bE_{\infty}$-ring structure.

\begin{construction}
We define $\underline{\mathrm{SH}}^{\AT}(\mathbb{R})$ as a $C_2$-$\infty$-category by defining a functor 
\[ 
c_{\mathbb{C}/\mathbb{R}}^{\textup{cat}} :\mathcal{O}_{C_2}^{\op} \to \mathrm{Cat}_{\infty} 
\]
by $C_2/H\mapsto \mathrm{SH}(\mathbb{C}^{H})$ on objects and on morphisms as follows. 
First, recall $\mathcal{O}_{C_2}^{\op}$ is a $1$-category, so its mapping spaces are descrete. A map
\[
\Delta^0 \to \mathrm{Map}(C_2/H,C_2/H')
\]
is then given by $gH'$ such that $gHg^{-1}\subset H'$. 
For any such $g$, there is a corresponding field automorphism $g:\mathbb{C}\to \mathbb{C}$, and $g$ restricts to a field 
homomorphism $\mathbb{C}^{H}\to \mathbb{C}^{H'}$. We define the associated morphism by 
\[
g^*:\mathrm{SH}(\mathbb{C}^{H'})\to \mathrm{SH}(\mathbb{C}^{H})
\]
and composition is defined by the formula $(gh)^*=h^*g^*$. 
This is essentially a categorification of~\cite[\S~4.3]{HO18}. 
In particular, the Galois group $C_2$ acts by a field homomorphism $\mathbb{C}^e\to \mathbb{C}^e$ and there is a 
corresponding Weyl group action on $\mathrm{SH}(\mathbb{C})$, which we refer to as ``complex conjugation''. 
\end{construction}

\begin{defin}\label{def:parametrized-categories}
We define a $C_2$-$\infty$-category $\m{\Mod}(\Fil(\m{\Sp}_{i2}^{C_2}),\m{\nu_{\bR}(\bS)})$ by 
\[ 
\begin{tikzcd}
\Mod(\Fil(\Sp^{C_2}),\nu_{\bR}\bS_2) \arrow{r} & \Mod(\Fil(\Sp),\nu_{\bC}\bS_2) \arrow[loop right]{r}
\end{tikzcd}
\]
and a $C_2$-$\infty$-category $\m{\Mod}(\Gr(\m{\Sp}^{C_2}),\m{\ol{\nu}_{\bR}(\bS_2)})$
by 
\[ 
\begin{tikzcd}
\Mod(\Gr(\Sp^{C_2}),\ol{\nu}_{\bR}\bS_2) \arrow{r} & \Mod(\Gr(\Sp),\ol{\nu}_{\bC}\bS_2)\arrow[loop right]{r}
\end{tikzcd}
\]
where the Weyl group action comes from the equivalences
\[
\Mod(\Fil(\Sp_{i2}),\ol{\nu}_{\bC}\bS)\simeq \Mod(\SH^{\AT}(\bC),\bS_2)
\]
and 
\[
\Mod(\Gr(\Sp),\ol{\nu}_{\bC}\bS_2)\simeq \Mod(\SH^{\AT}(\bC),(C\tau)_2)
\]
from~\cite{GIKR22} and the $C_2$-action on $\SH(\bC)$ given by complex conjugation. 

We further define 
$\m{\Mod}(\m{\Sp}^{C_2},H\mZ_2)\otimes_{\Mod(\Sp,H\bZ_2)}\Mod(\IndCoh(\cM_{\fg}),(\mathcal{O}_{\cM_{\fg}})_2)$ as 
\[
\begin{tikzcd}
\Mod(\Sp^{C_2},H\mZ_2)\otimes_{\Mod(\Sp,H\bZ_2)}\Mod(\IndCoh(\cM_{\fg}),(\mathcal{O}_{\cM_{\fg}})_2) \arrow{d}  \\ \Mod(\IndCoh(\cM_{\fg}),(\mathcal{O}_{\cM_{\fg}})_2) \arrow[loop right]{r}
\end{tikzcd}
\]
together with a compatible Weyl group action stemming from the equivalence 
\[
\Mod(\IndCoh(\cM_{\fg}),(\cO_{\cM_{\fg}})_2)\simeq \Mod(\SH^{\AT}(\bC),(C\tau)_2) 
\]
from~\cite{GWX21} and the $C_2$-action on $\SH(\bC)$ given by complex conjugation. 
\end{defin}

\begin{remark}
We say that a functor $\underline{\cC}\to \underline{\cD}$ between $C_2$-$\infty$-categories is fiberwise 
symmetric monoidal if 
$\underline{\cC}_{C_2/H}\to \underline{\cD}_{C_2/H}$ is symmetric monoidal for each subgroup $H\subset C_2$. 
\end{remark}

\begin{proposition}
There is a fiber-wise symmetric monoidal $C_2$-parametrized Betti realization functor 
\[ 
\DK : \m{\SH}^{\AT}(\bR)\longrightarrow \m{\Sp}^{C_2}  \,.
\] 
There is a fiber-wise symmetric monoidal  equivalence of $C_2$-$\infty$-categories
\[ 
\Mod(\m{\SH}^{\AT}(\bR),\bS_2^\wedge)\simeq \m{\Mod}(\Fil(\m{\Sp}^{C_2}),\m{\nu_{\bR}(\bS)}_2^\wedge)
\]
as well as fiber-wise symmetric monoidal equivalences of $C_2$-$\infty$-categories 
\begin{align*}
\m{\Mod}( \m{\SH}^{\AT}(\bR),C\ta_2^\wedge)&\simeq \m{\Mod}(\Gr(\m{\Sp}^{C_2}),\m{\ol{\nu}_{\bR}(\bS_2)}) \\ 
&\simeq \m{\Mod}(\m{\Sp}^{C_2},H\mZ_2)\otimes_{\Mod(\Sp_{i2},H\bZ)}\Mod(\IndCoh(\cM_{\fg}),(\cO_{\cM_{\fg}})_2^\wedge) \,.
\end{align*}
\end{proposition}

\begin{proof}
The first statement follows from~\cite[Theorem 4.7,~Theorem~6.4]{ES21} by restricting to $\underline{\mathrm{SH}}^{\mathrm{AT}}(\mathbb{R}) \subset \underline{\mathrm{SH}}(\mathbb{R})$. The second and third equivalences follow  from~\cite[Theorem~5.1]{BHS22} and~\cite[Corollary~1.2]{GWX21} after unraveling Definition~\ref{def:parametrized-categories}. 
\end{proof}

\begin{remark}
The $\infty$-category 
\[
\m{\Mod}(\m{\Sp}^{C_2},H\mZ_2)\otimes_{\Mod(\Sp,H\bZ_2)}\Mod(\IndCoh(\cM_{\fg}),(\cO_{\cM_{\fg}})_2)
\]
is equipped with a levelwise Chow $t$-structure whose heart can be identified with 
\[ 
\begin{tikzcd}
\CoMod\left (\m{(\MU_2)}_{2*},\m{(\MU_*\MU_2})_{2*} \right)\arrow{r}  & \CoMod\left ((\MU_2)_{2*},(\MU_*\MU_2)_{2*})\right )\arrow[loop above]{r}
\end{tikzcd}
\]
as a $C_2$-$\infty$-category where the $C_2$-action again comes from complex conjugation. 
\end{remark}

For the sake of our computations, the only consequence that is necessary for us is the following definition, made possible by the results above. 

\begin{defin}
Fix $n\ge 0$. We write $\gr_{\sev,2}^*\bS_2/(\ol{v}_0,\cdots ,\ol{v}_n)$ for the $\bE_\infty$ $\gr_{\sev,2}^*\bS_2$-algebra corresponding to the $\MU_*\MU_2$-comodule algebra $\MU_*/(2,v_1,\cdots ,v_n)\otimes_{\bZ_2}\mZ_2$ under the equivalence \eqref{eq:cof-tau} as described above. Given a $\gr_{\sev,2}^*\bS_2$-module $M$, we write
\[ 
M/(\ol{v}_0,\cdots ,\ol{v}_n) := \gr_{\sev,2}^*\bS/(\ol{v}_0,\cdots ,\ol{v}_n)\otimes_{\gr_{\sev,2}^*\bS_2}M  \,.
\]
\end{defin}

\subsection{The strongly even filtration for modules}\label{sec:filtration-modules}
Consider the following construction of an $\infty$-category of pairs $(A,M)$ where $A$ is a $C_2$-$\bE_\infty$-algebra 
and $M$ is a module over the underlying $\bE_\infty$-algebra. First, let $\Mod(\cC)$ be the $\infty$-category of 
pairs $(A,M)$ where $A$ is an $\bE_\infty$-algebra in $\cC$ and $M$ is a module over $A$ in $\cC$, 
defined as in \cite[Definition~2.5]{EM06} and \cite[Definition~4.2.1.13]{Lur17}. 
This $\infty$-category is equipped with forgetful functors  
\begin{align*}  
\cU_{\Alg} : \Mod(\cC)&\longrightarrow \CAlg(\cC) \\
(A,M) & \longmapsto A
\end{align*}
and 
\begin{align*}
\cU_{\Mod}  : \Mod(\cC)& \longrightarrow \cC \\
(A,M) & \longmapsto M  \,.
\end{align*} 

\begin{defin}\label{def:mod}
We define an $\infty$-category $\mathrm{Mod}^{C_2}$ as the pullback 
\[
\begin{tikzcd}
\Mod^{C_2} \arrow{d}{\cU_{\Alg}^{C_2}} \ar[r] & \Mod(\Sp^{C_2}) \arrow{d}{\cU_{\Alg}} \\ 
\CAlg^{C_2}\arrow{r}{\textup{forget}} & \CAlg(\Sp^{C_2}) \,.
\end{tikzcd}
\]
An object in this pullback is the data of a $C_2$-$\bE_\infty$-algebra together with a module $M$ over the 
underlying $\bE_{\infty}$-algebra. 
\end{defin}

\begin{defin}\label{def:modsev}
Let $\Mod^{C_2}_{\sev}$ denote the pullback 
\[
\begin{tikzcd}
\Mod^{C_2}_{\sev} \arrow{r}{\inc} \ar[d] &  \Mod^{C_2}\arrow{d}{\cU_{\Alg}^{C_2}} \\
\CAlg^{C_2}_{\sev}\arrow{r}{\inc} & \CAlg^{C_2}
\end{tikzcd}
\]
where $\inc$ is the canonical inclusion. 
\end{defin}

\begin{prop}
The functor 
\[
\inc : \Mod^{C_2}_{\sev}\longrightarrow \Mod^{C_2}
\]
is accessible. 
\end{prop}

\begin{proof}
 
The forgetful functors 
$\cU_{\Alg} : \Mod(\Sp_{C_2})\longrightarrow \CAlg$ and $\CAlg^{C_2}\longrightarrow  \CAlg$ 
are also accessible since they each admit an adjoint and both source and target are presentable. 
Consequently, the forgetful functor $\cU_{\Alg}^{C_2}:\Mod^{C_2}\to\CAlg^{C_2}$ is accessible. 
By Proposition~\ref{prop:acc1}, the functor $\inc:\CAlg^{C_2}_{\sev}\longrightarrow \CAlg^{C_2}$ is accessible. 
Therefore, the result follows from~\cite[Proposition~5.4.6.6]{Lur09}. 
\end{proof}

\begin{defin}\label{def:sev-filt-module}
Let $(A,M)\in \Mod^{C_2}$. We define the \emph{strongly even filtration} of $(A,M)$ as the limit 
\[ 
\fil_{\sev/A}^{\bullet}M \coloneqq \lim_{A\to B,B\in \CAlg^{C_2}_{\sev}}P_{ 2\bullet}(M\otimes_{A}B)  \,.
\]
If $M=A$, then $\fil^\bullet_{\sev}(A) \simeq  \fil^\bullet_{\sev/A}A$.

Let $(A,M)\in \Mod^{C_2}$. We define the \emph{$p$-completely strongly even filtration} of $(A,M)$ as the limit 
\[ 
\fil_{\sev/A,\,p}^{\bullet}M \coloneqq 
\lim_{A\to B,B\in \CAlg_{\sec,\,p}^{C_2}}P_{ 2\bullet}(M\otimes_{A}B)_p^\wedge  \,.
\]
If $M=A$, then $\fil^\bullet_{\sev,\,p}(A) \simeq  \fil^\bullet_{\sev/A,\,p}A$.
\end{defin}

\begin{remark}
More formally, we could alternatively define 
\[ 
\fil_{\sev/\--}^*(-) :\Mod^{C_2}\longrightarrow \Fil(\Sp_{C_2}) 
\]
to be the right Kan extension 
\[
\begin{tikzcd}
\Mod^{C_2}_{\sev} \arrow{d}{\textup{inc}} \arrow{rr}{P_{2*} \mathcal{U}_{\textup{Mod}}} & & \Fil(\Sp_{C_2}) \\
\Mod^{C_2}\arrow{urr} & &
\end{tikzcd}
\]
where the top horizontal functor is the composite of the forgetful functor $(A,M)\mapsto M$ and the functor $P_{2\bullet}$. 
Proposition~\ref{prop:acc1} implies that this right Kan extension exists and can be identified with a $\kappa_1$-filtered 
limit appearing in Definition~\ref{def:sev-filt}. 
\end{remark}

\subsection{Descent for modules}\label{sec:desent-modules}

\begin{defin}
We say that a sieve on $(A,M)$ in $(\Mod^{C_2}_{\sev})^{\op}$ is a \emph{pure covering sieve} if it 
contains a finite collection of maps $\{ (A,M) \to (B_i,M_i) \}_{1\le i\le n}$ in $\Mod^{C_2}_{\sev}$ 
such that $A\to \prod_i B_i$ is a strongly evenly pure map and for each $i$ there are equivalences $M\otimes_AB_i \simeq M_i$. 

We say that a sieve on $(A,M)$ in $(\Mod_{\sev,\,p}^{C_2})^{\op}$ is a \emph{$p$-completely pure covering sieve} 
if it contains a finite collection of maps $\{ (A,M) \to (B_i,M_i) \}_{1\le i\le n}$ in $\Mod_{\sev,\,p}^{C_2}$ 
such that $A\to \prod_i B_i$ is a $p$-completely strongly evenly pure map and for each $i$ there are equivalences 
$(M \otimes_AB_i)_p^\wedge\simeq (M_i)_p^\wedge$. 
\end{defin}

\begin{prop}\label{prop:descent-for-modules}
The pure (resp. $p$-completely pure) covering sieves form a Grothendieck topology on $(\Mod_{\sev}^{C_2})^{\op}$ 
(resp. $(\Mod_{\sev,\,p}^{C_2})^{\op}$) called the \emph{pure topology} (resp. \emph{$p$-completely pure topology}). 
For a category $\mathcal{C}$ admitting small limits, a functor 
\[ 
F : (\Mod_{\sev}^{C_2})^{\op}\to \cC 
\]
(resp. $F : (\Mod_{\sev,\,p}^{C_2})^{\op}\to \cC$)
is a sheaf for the pure (resp. $p$-completely pure) topology if it preserves finite products and for any pure map 
$A\to B$ in $\CAlg_{\sev}^{C_2}$ (resp. $\CAlg_{\sev,\,p}^{C_2}$) the canonical map 
\[ 
F(M)\to \lim_{\Delta} F(M\otimes_{A}B^{\otimes_{A}\bullet+1})
\]
is an equivalence (resp. the canonical map 
\[ 
F(M)\to \lim_{\Delta} F((M\otimes_{A}B^{\otimes_{A}\bullet+1})_p^\wedge)
\]
is an equivalence.)
\end{prop}

\begin{proof}
    It is elementary to check the conditions in the proofs 
    of~\cite[A.3.2.1,~A.3.3.1]{Lur18} since pushouts in $\Mod_{\sev}^{C_2}$ 
    (resp. $\Mod_{\sev,\,p}^{C_2}$) along maps $(A,M)\to (B,N)$ where $A\to B$ is  pure (resp. $p$-completely pure) and 
    $M\otimes_{A}B\simeq N$ (resp. $(M\otimes_{A}B)_p^\wedge \simeq N_p^\wedge$) exist. 
    As these are the only pushouts for which we need to check these conditions, the result follows. 
\end{proof}

\begin{prop}\label{prop:module-sheaves}
The composite functors
\[ 
\begin{tikzcd}
\Mod_{\sev}^{C_2} \arrow{r}{\mathcal{U}_{\textup{Mod}}} & \CAlg_{\sev}^{C_2} \arrow{r}{P_{2*}} & \Fil( \Sp^{C_2}) ,
\end{tikzcd}
\]
\[ 
\begin{tikzcd}
\Mod_{C_2,p}^{\sev} \arrow{r}{\mathcal{U}_{\textup{Mod}}} & \CAlg_p^{\sev} \arrow{r}{P_{2*}} & \Fil( \Sp_p^{C_2})
\end{tikzcd}
\]
are sheaves for the pure topology and the $p$-completely pure topology, respectively. 
\end{prop}

\begin{proof}
We just prove the case of the pure topology since the other case is similar. We consider the map of filtered graded spectra 
\begin{equation}\label{can map for sheaf 1.5} 
P_*^*P_{\ast}M\to \lim_{\Delta}  P_*^* \left ( P_{\ast}M\otimes_{P_{\ast}A}\left ({P_{\ast}B}^{\otimes_{P_{\ast}A} \bullet+1} \right ) \right ) 
\end{equation}
which is a map of complete filtered graded $C_2$-spectra whose colimit is the map 
\[ P_*^*M\to \lim_{\Delta} P_*^*\left (M\otimes_{ A}\left (B^{\otimes_{A} \bullet+1} \right ) \right ) \,.
\]
By completeness of the slice filtration, it suffices to show that the map 
\begin{equation}\label{can map for sheaf 1.5} 
P_{\ast}^{\ast}P_{\ast}^{\ast}M\longrightarrow 
\lim_{\Delta}  P_{\ast}^{\ast} \left ( P_{\ast}^{\ast}M
\otimes_{P_{\ast}^{\ast}A}\left ( P_*^*B^{\otimes_{P_*^*A} \bullet+1} \right )  \right ) 
\end{equation}
is an equivalence of bigraded spectra. We then observe that we can write this as 
\[
P_*^*M\longrightarrow \lim_{\Delta}P_*^* \left ( P_{\ast}^{\ast}M
\otimes_{\bigoplus_{s}\Sigma^{\rho s}H\m{\pi_{2s}^eA}} \left ( 
\bigoplus_{s}\Sigma^{s\rho}H\m{\pi_{2s}^eB}^{\otimes_{\bigoplus_{s}\Sigma^{s\rho}H\m{\pi_{2s}^eA}} \bullet+1} \right )  \right ) 
\]
and the result follows from the underlying faithfully flat descent~\cite[Appendix~D]{Lur18}. 
To see this, note that we can forget the grading to identify $\bigoplus_{s} \Sigma^{\rho s}H\pi_{2s}^eA$ and  $\bigoplus_{s} \Sigma^{\rho s}H\pi_{2s}^eB$ with  commutative Green functors by forgetting the grading.\footnote{Here we use Voevodsky's grading convention. See~\cite{DDIO24} for details.} Since they are also constant Mackey functors, descent is determined on the underlying. 
\end{proof}

\begin{prop}\label{prop:descent-modules}
Let $A\to B$ be a pure (resp. $p$-pure) map of $C_2$-$\bE_{\infty}$-ring spectra 
and $M$ be an $A$-module. 
There is an equivalence 
\begin{align}\label{eq:fil-module1} 
\fil_{\sev/A}^{\bullet}M  
\simeq \lim_{\Delta}\fil_{\sev/B^{\otimes_{A}\bullet+1}}^{\bullet}M\otimes_{A}B^{\otimes_{A}\bullet+1}   
\end{align}
of filtered $C_2$-spectra (resp. there is an equivalence of 
\begin{align}\label{eq:fil-module2}
\fil_{\sev/A_p,\,p}^{\bullet}M_p^\wedge  
\simeq \lim_{\Delta}\fil_{\sev/(B^{\otimes_{A}\bullet+1})_p}^{\bullet}(M\otimes_{A}B^{\otimes_{A}\bullet+1})_p^\wedge   
\end{align}
of filtered $p$-complete $C_2$-spectra).
\end{prop}
\begin{proof}
The proof is essentially the same as the proof of Theorem~\ref{thm:descent}. 
\end{proof}

\begin{defin}\label{def:projective}
Given a seff map $A\to B$, we say that an $A$-module $M$ is \emph{$B$-projective} if $M\otimes_A B$ is a pure $B$-module. We say an $A$-module is \emph{$p$-completely $B$-projective} if $(M\otimes_A B)_p^\wedge$ is a pure $B_p^\wedge$-module.
\end{defin}

To produce an interesting example, we will use the following observation. 

\begin{remark}
When $M$ is a $C_2$-spectrum and $A\to B$ is the map
$\bS\to \MUR$, then Definition~\ref{def:projective} specializes to the notion of $\MUR$-projective considered in \cite{BHS22}. 
\end{remark}

\begin{prop}\label{prop:descent-for-modules-again}
Suppose $A\to B$ is pure (resp. $p$-completely pure) and $X$ is $B$-projective (resp. $p$-completely $B$-projective). Then there is an equivalence 
\[ 
\fil_{\sev/A}^sB\simeq \lim_{\Delta} P_{2s}(X\otimes_{A} B^{\otimes_{A}\bullet+1}) 
\]
of filtered $C_2$-spectra (resp. there is an equivalence
\[ 
\fil_{\sev/A_p^\wedge,\,p}^sB_p^\wedge\simeq \lim_{\Delta} P_{2s}(X\otimes_{A_p^\wedge} 
(B_p^\wedge)^{\otimes_{A_p^\wedge}\bullet+1})_p^\wedge
\]
of filtered $C_2$-spectra.)
\end{prop}

\begin{proof}
This follows from Propositions \ref{prop:seff-map-MUR} and~\ref{prop:descent-modules}. 
\end{proof}

\begin{cor}
If $X$ is a $2$-complete $C_2$-spectrum such that $(\MUR\otimes X)_2$ has bounded $2$-torsion and $(\MUR\otimes X)_2$ is a 
wedge of $\rho$-multiple suspensions of $\MUR$, then 
\[
\fil_{\sev/\mathbb{S}_2^\wedge,\,2}^{\bullet}X\simeq \nu_{\mathbb{R}}X\,.
\] 
When $X=\bS_2^\wedge$ this is an equivalence of $C_2$-$\bE_{\infty}$-algebras.
\end{cor}

\begin{proof}
Since $\mathbb{S} \to \MUR$ is $2$-completely pure, we can identify 
\[
\fil_{\sev/\bS_2^\wedge,2}^{s}X\simeq \lim_{\Delta }P_{2s}(X\otimes_{\bS_2^\wedge} 
(\MUR)_2^\wedge)^{\otimes_{\bS_2^\wedge} \bullet +1})_2^\wedge\,.
\]
When $X=\bS_2^\wedge$ this is an identification as $C_2$-$\bE_{\infty}$-algebras in filtered spectra. 
\end{proof}

\begin{exm}
If $\bS/2$ is the mod $2$ Moore spectrum, then there are equivalences of $\gr_{\sev}^*\bS$-modules
\[ 
\gr_{\sev}^{\bullet}(\bS)/(\ol{v}_0)\simeq \gr_{\sev/\bS}^{\bullet}(\bS/2)\simeq \gr^*\nu_{\bR}(\bS/2) \,.
\]
\end{exm}

\subsection{Variants for modules}\label{sec:motivic-variants}
In this section, we note that all of these arguments also carry over to the context of spectra 
with $C_2$-twisted $S^1$-action. 

\begin{defin}
We define $\infty$-categories
\[
\Mod^{h_{C_2}S^1}_p \text{ and }\Mod_p^{\rcyc}
\]
as pullbacks just as in Definition~\ref{def:mod} along with forgetful functors 
\[ 
\cU_{\Mod}:\Mod_p^{h_{C_2}S^1}\longrightarrow \Sp_p^{h_{C_2}S^1}\text{ and } 
\cU_{\Mod}:\Mod_p^{\rcyc}\longrightarrow
\Sp_p^{\rcyc}\,.
\]
We further write 
\[
 \Mod_{\sev,\,p}^{h_{C_2}S^1}  \text{ and }  \Mod_{\sev,\,p}^{\rcyc}
\]
for the full subcategories spanned by those pairs $(A,M)$ such that the underlying object of $A$ in $\CAlg(\Sp^{C_2})$ is strongly even. 
\end{defin}

\begin{remark}
Note that the inclusions 
\[
 \Mod_{\sev,\,p}^{h_{C_2}S^1}\subset  \Mod_{p}^{h_{C_2}S^1} \text{ and }\Mod_{\sev,\,p}^{\rcyc}\subset \Mod_{p}^{\rcyc}
\]
are accessible functors by the same argument as before. 
\end{remark}
\begin{defin}
We define 
\begin{align*}
\fil_{\sev/\--,\,h_{C_2}S^1,p}^{\bullet}:&\Mod_{p}^{h_{C_2}S^1}\longrightarrow \Fil(\Sp_p^{C_2}) \,, \\
\fil_{\sev/\--,\,t_{C_2}S^1,p}^{\bullet}:&\Mod_p^{h_{C_2}S^1}\longrightarrow \Fil(\Sp^{C_2}_p)  
\end{align*}
as the right Kan extensions of 
$(\cU_{\Mod})^{h_{C_2}S^1}$ and 
$
(\cU_{\Mod})^{t_{C_2}S^1}$, respectively,
along the inclusion 
$\Mod_{\sev,\,p}^{h_{C_2}S^1}\subset \Mod_p^{h_{C_2}S^1}$. 
Similarly, we define
\[ 
\fil_{\sev/\--,\TCR}^{\bullet} : \Mod_{p}^{\rcyc}\longrightarrow \Fil(\Sp_p^{C_2}) 
\]
by right Kan extension of the equalizer 
\[\mathrm{eq} \left ( \can,\varphi : \cU_{\Mod}^{h_{C_2}S^1}\longrightarrow  \cU_{\Mod}^{t_{C_2}S^1}\right ) \]
along the inclusion 
\[
\Mod_{\sev,\,p}^{\rcyc}\subset \Mod_{p}^{\rcyc} \,.
\]
\end{defin}

\begin{thm}
The underlying $p$-completely pure covering sieves form a Grothendieck topology on $\Mod_{\sev,\,p}^{h_{C_2}S^1}$ and $\Mod_{\sev,\,p}^{\rcyc}$ called the \emph{$p$-completely pure topology} and the functors 
\[ 
P_*^*\left ( (\--)^{h_{C_2}S^1}\right )_p\,,  P_*^*\left ( (\--)^{t_{C_2}S^1}\right )_p \,, P_*^*\left ( (P_*^*\--)^{h_{C_2}S^1}\right )_p\,, \text{ and }P_*^*\left ((P_*^*\--)^{t_{C_2}S^1} \right )_p
\]
are sheaves for this topology. 
\end{thm}
\begin{proof}
The proof is the same as that of Proposition~\ref{prop:descent-for-modules} and Proposition~\ref{prop:module-sheaves}. 
\end{proof}

\begin{defin}
Let $A$ be a $C_2$-$\bE_{\infty}$-ring with $S^1$-action, let $B$ be an $\bE_\sigma$-ring such that $\THR(B)$ is an $A$-module in $C_2$-spectra with $C_2$-twisted $S^1$-action. We define 
\begin{align*}
\fil_{\mot/A}^\bullet\TCR^{-}(B;\bZ_p)&:= \fil_{\sev/A,h_{C_2}S^1,p}^\bullet \THR(B;\bZ_p) \,,  \\ 
\fil_{\mot/A}^\bullet\TPR(B;\bZ_p)&:= \fil_{\sev/A,t_{C_2}S^1,p}^\bullet \THR(B;\bZ_p)\,.  
\end{align*}
If $A$ is additionally Real $p$-cyclotomic and $\THR(B)$ is additionally a module over $A$ in Real $p$-cyclotomic spectra, then we define 
\[
\fil_{\mot/A}^\bullet\TCR(B;\bZ_p):=\fil_{\sev/A,\,\TCR,\,p}^{\bullet}\THR(B;\bZ_p) \,.
\]
\end{defin}

\begin{thm}
Let $A\to B$ be a $p$-completely pure map of $C_2$-$\bE_\infty$-rings and let $\THR(C)$ be an $S^1$-equivariant $A$-module. We have equivalences
\begin{align*}
\fil_{\mot/A}^{\smallblacksquare}\TCR^{-}(C;\bZ_p)&:=\lim_{\Delta}\fil_{\sev/B^{\otimes_{A}\bullet+1},\,h_{C_2}S^1,\,p}^{\smallblacksquare}\left (
\THR(C)\otimes_{A}B^{\otimes_{A}\bullet+1} \right ) \,, \\ 
\fil_{\mot/A}^{\smallblacksquare}\TPR(C;\bZ_p)&:=\lim_{\Delta}\fil_{\sev/B^{\otimes_{A}\bullet+1},\,t_{C_2}S^1,\,p}^{\smallblacksquare}\left (\THR(C)\otimes_{A}B^{\otimes_{A}\bullet+1} \right )
\end{align*}
of filtered $C_2$-spectra. 
\end{thm}

\begin{proof}
The proof is essentially the same as Propositions \ref{prop:descent-modules} and~\ref{prop:descent-for-modules-again}. 
\end{proof}

\section{Real cyclotomic bases}\label{sec:base}
First, we give a general discussion of $\rho$-cellular $C_2$-$\bE_\infty$-rings; see~Definition~\ref{def: cell decomp}. 
We then discuss forms of $\BPR\langle n\rangle$. 
Next, we produce examples of $C_2$-$\bE_{\infty}$-rings that admit $\rho$-cellular decompositions and use this to produce 
a useful Real cyclotomic base. We then discuss a notion of Real chromatically quasisyntomic $C_2$-$\bE_\infty$-rings and 
sufficient conditions so that the canonical and Frobenius maps lift to filtered maps. 

\subsection{$\rho$-cellular decompositions}\label{sec:base-cellular}
The main objects of study will be $C_2$-$\bE_{\infty}$-rings that admit a nice cellular filtration.

\begin{defin}\label{def: cell decomp}
A $C_2$-$\bE_\infty$-ring $A$ is \emph{$\rho$-cellular} if there exists a sequence of $C_2$-$\bE_\infty$-ring maps 
\[ 
A_i\longrightarrow A_{i+1}\to \dots \longrightarrow A_{\infty}=A
\]
where $A_{i}=\Sym_{C_2} (\bigoplus \bS^{i\rho -1})$ and $A_j$ is defined inductively by a pushout 
\[
    \begin{tikzcd}
       \Sym_{C_2}(\bigoplus \bS^{j\rho -1}) \ar[r]\ar[d] & A_j\ar[d] \\
       \bS \ar[r] & A_{j+1}
    \end{tikzcd}
\]
for $j>i$. Here $\Sym_{C_2}$ denotes the left adjoint to the forgetful functor $\CAlg_{C_2}\to \Sp^{C_2}$. 
\end{defin}

Our primary reason for focusing on $\rho$-cellular $C_2$-$\bE_\infty$-rings is the fact that obstruction theory arguments 
work particularly well, see Theorem~\ref{thm: obstruction thy} below for example. 
First, we recall a universal property of $\THR$. 

\begin{thm}[{\cite[Theorem 4.12]{AKGH25}},~{\cite[\S~5]{QS21a}}]\label{thm: universal property}
Given an object~$A\in \CAlg^{C_2}$ and an object~$B\in\CAlg^{h_{C_2}S^1}$, 
there is a natural isomorphism 
\[ 
\Hom_{h\CAlg^{h_{C_2}S^1}}(\THR(A),B)\cong \Hom_{h\CAlg^{C_2}}(A,B) \,.
\]
\end{thm}

\begin{exm}
Given a $C_2$-$\bE_\infty$-ring $A$, the identity map~$\id_A : A\to A$ produces a canonical map~$\THR(A)\to A$ 
in $\CAlg^h_{C_2}$ 
\end{exm}

\begin{thm}\label{thm: obstruction thy}
Let $A$ be a $\rho$-cellular $C_2$-$\bE_\infty$-ring. Let $\THR(A)\to R$ be a map in $\CAlg^{h_{C_2}S^1}$ where $R$ is even. 
Then the obstructions to producing an extension
\[
    \begin{tikzcd}
       \THR(A) \ar[d] \ar[r] & R \\
        A \ar[ur,dashed] & 
    \end{tikzcd}
\]
along the canonical map~$\mathrm{THR}(A)\to A$ vanish. 
\end{thm}
\begin{proof}
As in the proof of~\cite[Theorem 3.1.9]{HRW22}, we prove this by induction on the cells of $A$. 
By Theorem~\ref{thm: universal property}, asking for an extension 
\[
    \begin{tikzcd}
       \THR(\Sym_{C_2}(\bigoplus \bS^{k\rho-1})) \ar[d] \ar[r] & R \\
        \Sym_{C_2}(\bigoplus \bS^{k\rho-1}) \ar[ur,dashed] & 
    \end{tikzcd}
\]
amounts to extending the map of $C_2$-spectra
\begin{align}\label{eq: map from wedge of rho k spheres}
\bigoplus \bS^{k\rho-1}\to R
\end{align}
to map in $\Sp^{h_{C_2}S^1}$.

Since $R$ is Real oriented \cite[Lemma 2.3]{HM17},  the parametrized homotopy fixed point spectral sequence 
\begin{align}\label{eq: Tate spectral sequence for obstruciton theory proof}
  \mrH^{*}_{\AD}(\bC P^{\infty}_{\bR};\pi_{\star}^{C_2}R)\implies \pi_{\star}^{C_2}R^{h_{C_2}S^1}
\end{align}
collapses at the $\EE_2$-page and thus the edge homomorphism $\pi_{\star}^{C_2}R^{h_{C_2}S^1}\to \pi_{\star}^{C_2} R$ is 
surjective.  We conclude that there is a factorization
\[ \bigoplus \bS^{k\rho-1}\to R^{h_{C_2}S^1}\to R\]
of the map \eqref{eq: map from wedge of rho k spheres} and by adjunction this implies that the map of $C_2$-$\bE_\infty$-rings 
adjoint to \eqref{eq: map from wedge of rho k spheres} extends to a map in $\CAlg^{h_{C_2}S^1}$.

Suppose we have produced an extension 
\[
    \begin{tikzcd}
       \THR(A_j) \ar[d] \ar[r] & R \\
      A_j \ar[ur] & 
    \end{tikzcd}
\]
and we want to extend this map further to a commutative diagram
\begin{equation}\label{eq: diagram for obstruction to extending Aj to Aj+1}
    \begin{tikzcd}
        \THR(A_{j}) \ar[d]  \ar[r] & \THR(A_{j+1})\ar[d] \ar[r] & R \\
      A_{j} \ar[r] & A_{j+1} \ar[ur] & 
    \end{tikzcd}
\end{equation}
Since $\THR$ sends pushouts of $C_2$-$\bE_\infty$-rings to pushouts of $C_2$-$\bE_\infty$-rings and we can apply the 
natural transformation $\THR \Rightarrow \id$ from Theorem~\ref{thm: universal property} to the pushout defining $A_{j+1}$. 
We observe that the obstructions to extending the map
\[ 
\begin{tikzcd} A_{j} \ar[r] \ar[d] & R \\  A_{j+1} \arrow[dashed]{ur} & \end{tikzcd} 
\]
of $C_2$-$\bE_\infty$-rings along the dashed arrow lie in $\pi_{j\rho -1}^{C_2}R$, which is zero. 
To ensure that the extension $A_{j+1}\to R$ is compatible with the 
diagram~\eqref{eq: diagram for obstruction to extending Aj to Aj+1} we observe that the obstructions of this map to extend to 
a map in $\CAlg^{h_{C_2}S^1}$ are trivial, again because the spectral 
sequence~\eqref{eq: Tate spectral sequence for obstruciton theory proof} collapses. 
\end{proof}

\subsection{Forms of $\BPR\langle n\rangle$}\label{sec:forms}
Following~\cite{HM17}, we define forms of $\mathrm{BP}_\mathbb{R}\langle n\rangle$ as follows:

\begin{defin}
An even Real oriented $p$-local homotopy commutative and associative $C_2$-ring spectrum $R$ is \emph{a form of $\BPR\langle n \rangle$ at the prime $p$} if 
the map
\[ 
\mZ_{(p)}[\bar{v}_1,\cdots ,\bar{v}_n]\subset  \upi_{*\rho }\BPR \subset \upi_{*\rho }(\MUR)_{(p)} \to \upi_{*\rho}R 
\]
induced by the Real orientation of $R$ is an isomorphism of constant Mackey functors.
When $R$ is additionally a $C_2$-$\bE_{\infty}$-algebra, then we say that $R$ is a \emph{$C_2$-$\bE_{\infty}$ form of $\BPR\langle n\rangle$}. This does not depend on a choice of classes $\bar{v}_1,\cdots ,\bar{v}_n$.  
\end{defin}

\begin{exm}\label{Ex:BPRn}
When $n=-1$, $H\mF_p$ is a $C_2$-$\bE_{\infty}$-form of $\BPR\langle -1\rangle$. When $n=0$, $H\mZ_p$ is a $C_2$-$\bE_{\infty}$-form of $\BPR\langle 0\rangle$. When $n=1$, $\ku_{\mathbb{R}}$ is a $C_2$-$\bE_{\infty}$-form of $\BPR\langle 1\rangle$ at the prime $2$. When $n=2$, $\tmf_1(3)$ is a $C_2$-$\bE_{\infty}$-form of $\BPR\langle 2\rangle$ at the prime $p=2$ by~\cite{HM17}. 
\end{exm}

\begin{remark}
Although there exist $C_2$-$\bE_\infty$-forms of $\BPR\langle n\rangle$ for $-1\le n\le 2$ at the prime $2$, it is not known whether the orientations $\MUR\to \tmf_1(3)$ and $\MUR\to \kr$ are maps of $C_2$-$\bE_{\infty}$-rings, or even $\mathbb{E}_{\rho}$-rings. 
There is work in progress of Ryan Quinn which should shed light on this question. To the authors' knowledge, it is currently not known whether $\BPR\langle n\rangle$ admits a homotopy commutative ring structure for $n>2$~\cite{KLW17}, but it has been predicted that one can produce $\bE_{2\sigma +1}$-$\MUR$-algebra forms of $\BPR\langle n\rangle$~\cite[Remark~1.0.14]{HW22}. 
\end{remark}

\subsection{Examples of Real cyclotomic bases}\label{sec:realcyc}
Recall that we defined Real $p$-cyclotomic bases in Definition~\ref{def:cyc-base}. We now use the results of the previous section to produce Real cyclotomic bases.

\begin{construction}
Let $h\colon \thinspace \MUR \to \kr$ denote the Conner--Floyd real orientation of $\kr$~\cite{HS20}. By Real Bott periodicity~\cite{Ati66}, there is an equivalence $\Omega^{\infty \rho}\Sigma^{\rho}\kr\simeq BU_{\bR}$. There is therefore an equivariant infinite loop map
\[ 
\Omega^{\infty\rho}\Sigma^{\rho}h \colon \thinspace \Omega^{\infty \rho} \Sigma^{\rho}\MUR \longrightarrow \Omega^{\infty \rho}\Sigma^{\rho}\kr \simeq BU_{\bR}\longrightarrow \m{\Pic}(\Sp^{C_2}) \,,
\]
where the last map in the composite is the one discussed in~\cite[Remark~13]{HL18}.  Viewing this as a map of $C_2$-$\infty$-categories, we define $\MWR$ to be the $C_2$-colimit of the composite $C_2$-functor 
\[ 
\Omega^{\infty\rho}\Sigma^{\rho}h \colon \thinspace \Omega^{\infty \rho} \Sigma^{\rho}\MUR\longrightarrow \BUR\longrightarrow \m{\Pic}(\Sp^{C_2}) \to \underline{\CAlg},
\]
which is a model for the Thom spectrum by~\cite[Theorem~5.0.2]{HHKWZ20}. We refer the reader to~\cite[\S~5]{HHKWZ20} for further details on this construction (see also~\cite[Construction 4.9]{HW20}). 
\end{construction} 

\begin{prop}\label{prop:MW}
There is an isomorphism of $\pi_{\star}^{C_2}\MUR$-algebras 
\[ 
\pi_{\star}^{C_2}\MWR \cong \pi_{\star}^{C_2}\MUR[\ol{c}_j : j \in J ] \]
where $|\ol{c}_j|=k_j\rho$ for some integer $k_j$ for each $j\in J$, where $J$ ranges over polynomial generators of $\MW_*$ as an $\MU_*$-algebra. 
\end{prop}
\begin{proof}
Note that $\MUR_{\star}^{C_2}\MWR=\MUR_{\star}^{C_2}\Omega^{\infty}\Sigma^{\rho}\MUR$ by the Thom isomorphism. We can apply the Atiyah--Hirzebruch type spectral sequence for the filtered $C_2$-spectrum $(P_{*}\MUR)\otimes \MWR$ with $\EE_2$-page 
\[ 
\pi_{\star}^{C_2}(P_{*}^*\MUR)\otimes \MWR\cong \pi_{\star}^{C_2}H\mZ\otimes \Omega^{\infty}\Sigma^{\rho}\MUR[\ol{a}_i :i\ge 1]
\]
where $|\ol{a}_i|=i\rho $. 
 We then apply \cite[Theorem~2]{HH18} to deduce that $H\mZ\otimes \Omega^{\infty}\Sigma^{\rho}\MUR$ is strongly even and use the computation of the underlying from \cite[Theorem~3.3, Corollary~3.4]{Wil73} to deduce that 
\[
\pi_{\star}^{C_2}H\mZ\otimes \Omega^{\infty}\Sigma^{\rho}\MUR\cong \pi_{\star}^{C_2}H\mZ[\ol{c}_j : j\in J]
\]
where $|\ol{c}_j|=k_j\rho$ for some integer $k_j$. 
The $\EE_2$-page of the spectral sequence is therefore $\pi_{\star}^{C_2}H\mZ[c_j : j\in J][a_i : i\ge 1]$. From this, we can deduce that $\MWR$ is strongly even since $\pi_{\rho k-i}^{C_2}\MWR=0$ for $i=1,2,3$ is clear already at the $\EE_2$-page and the spectral sequence strongly converges. Consequently, it receives a homotopy commutative ring map $\MUR\to \MWR$ and we can consider the map of slice spectral sequences 
\[ \pi_{\star}^{C_2}P_*^*\MUR \to \pi_{\star}^{C_2}P_*^*\MUR\otimes \MWR\]
to determine that the classes $\ol{a}_i$ are infinite cycles and consequently, the target Atiyah--Hirzebruch type spectral sequence collapses at the $\EE_2$-page. We can also use this map to determine the $\MUR_{\star}$-module structure on $\MWR_{\star}$ is free as desired and there is no room for $\pi_{\star}^{C_2}\MUR$-algebra extensions, leading to the desired answer as a $\pi_{\star}^{C_2}\MUR$-algebra. 
\end{proof}


\begin{thm}\label{thm: examples of Real cyclotomic bases}
The $C_2$-$\bE_\infty$-ring $\MW_{\bR}$ satisfies the following:
\begin{enumerate}
\item $\MW_{\bR}$ is equipped with a $\rho$-cellular decomposition with exactly one $0$-cell.
\item The unit map $\bS\longrightarrow \MWR$ is strongly evenly pure. 
\item $\MW_{\bR}$ is a Real cyclotomic base. 
\item There are maps of $C_2$-$\bE_{\infty}$-ring spectra 
\[
\MWR\longrightarrow \tmf_1(3) \longrightarrow \kr \longrightarrow H\mZ_2\longrightarrow H\mF_2\,. 
\]
\end{enumerate}
\end{thm}

\begin{proof}
There is a $\rho$-cellular decomposition of $\MUR$ as a spectrum and $\MWR$ arises as the Thom construction after applying $\Omega^{\infty}\Sigma^{\rho}$. This produces a $C_2$-$\bE_\infty$-$\rho$-cellular decomposition of $\MWR$. The fact that the unit map $\bS \longrightarrow \MWR$ is strongly evenly pure follows. 
To see this, note that any even $C_2$-$\bE_{\infty}$-ring $C$ is real oriented and consequently 
\[ C\otimes \MWR\simeq C \otimes \MUR\otimes_{\MUR}\MWR \,.
\]
Since $\MWR$ is a wedge of $\rho$ multiple suspensions of $\MUR$, we know that $C\otimes \MWR$ is a wedge of $\rho$ multiple suspensions of $C$ as desired. 

The fact that $\MW_{\bR}$ is a Real cyclotomic base follows from Theorem~\ref{thm: obstruction thy}. There are $C_2$-$\bE_{\infty}$-ring maps
\[ 
\tmf_1(3)\longrightarrow \kr \longrightarrow H\mZ_2 \longrightarrow H\mF_2 \,.
\]
We start with the unit map
\[ \bS \to \tmf_1(3)\]
which is a $C_2$-$\bE_{\infty}$-ring map. The obstructions to extending the unit map to a map 
\[
\MWR\longrightarrow \tmf_1(3)
\]
of $C_2$-$\bE_{\infty}$-rings lie in degrees $\rho k-1$ for some $k$ and therefore vanish. 
\end{proof}

\subsection{Real chromatically quasisyntomic $C_2$-$\bE_\infty$-rings}\label{sec:base-quasisyntomic}
In this section, we introduce the notion of a Real chromatically quasisyntomic $C_2$-$\bE_\infty$-rings. 
\begin{defin}\label{def:real-quasisyntomic}
We say that a $C_2$-$\bE_\infty$-ring $A$ is \emph{Real chromatically quasisyntomic} if 
\begin{enumerate}
\item the spectrum $\MUR\otimes A$ is strongly even, 
\item  the commutative graded ring $\MU_{2*}A^e$ has bounded $p$-power torsion, and  
\item the algebraic cotangent complex $L^{\textup{alg}}_{\MU_{2*}A^e/\bZ}$ is concentrated in (homological) degrees $[0,1]$. 
\end{enumerate}
\end{defin}
\begin{remark}
A $C_2$-$\bE_\infty$-ring $A$ is Real chromatically quasisyntomic if and only if $\MUR \otimes A$ is strongly even and the underlying $\bE_\infty$-ring $A^e$ is chromatically quasisyntomic in the sense of \cite[Definition~1.3.1]{HRW22}. Consequently, if $A$ is a Real chromatically quasisyntomic $C_2$-$\bE_\infty$ ring, then $A^e$ is a chromatically quasisyntomic $\bE_{\infty}$-ring. 
\end{remark}

\begin{exm}
If $R$ is a quasisyntomic discrete ring, then the Eilenberg--MacLane spectrum $\m{R}$ of the associated $C_2$-Tambara functor is a Real chromatically quasisyntomic $C_2$-$\bE_\infty$-ring. 

For $-1\le n\le 2$, the $C_2$-$\bE_\infty$-rings $\BPR\langle n\rangle$ are Real chromatically quasisyntomic. 
\end{exm}

\section{Comparing filtrations}\label{sec:comparison}

In this section, we compare the motivic filtration on $\THR$ to the generalization of the Bhatt--Morrow--Scholze 
filtration defined by Park in \cite{Par23}, which we call the BMSP-filtration, and the HKR filtration from Hornbostel--Park \cite{HP23} and Yang~\cite{Yan25} 

Recall from \cite[Def. 4.10]{BMS19} that a ring $R$ is $p$-quasisyntomic if it is $p$-complete with bounded $p^\infty$-torsion 
and the cotangent complex $L_{R/\bZ_p} \in D(R)$ has $p$-complete Tor-amplitude in $[0,1]$. 

\begin{defin}
For a $p$-quasisyntomic commutative ring $R$, we write  
\[
\fil_{\BMSP}^{*}\THR(H\m{R};\bZ_p)\,, \fil_{\BMSP}^{*}\TCR^{-}(H\m{R};\bZ_p)\,, 
\fil_{\BMSP}^{*}\TPR(H\m{R};\bZ_p)\,, \fil_{\BMSP}^{*}\TCR(H\m{R};\bZ_p)
\] 
for the filtrations defined in~\cite[Theorem~10.1]{Par23}. 
\end{defin}

Let $\mbN$ denote the constant $C_{2}$-semi-Mackey functor on the natural numbers $\bN$, and for each $s \in \bN$, 
let  $\mbN[1/s^{n}]$ denote the constant $C_{2}$-semi-Mackey functor on $\bN[1/s^{n}]$. We regard both $\mbN$ and $\mbN[1/s^n]$ 
as $C_2$-$\bE_\infty$-algebras in $C_2$-spaces by equipping them with the discrete topology. 
Let 
\[
\bS[\bar{z}] := \bS[\mbN]\,, \quad \bS[\bar{z}^{1/s^n}] := \bS[\mbN[1/s^n]]\,, \quad \bS[\bar{z}^{1/s^\infty}] := \colim_n \bS[\bar{z}^{1/s^n}]
\]
be the associated $C_2$-$\bE_\infty$-rings, with underlying $C_2$-spectra 
\[
\Sigma^{\infty}_+\mbN \,,\quad \Sigma^{\infty}_+\mbN[1/s^n] \,, \quad  \colim_n \Sigma^{\infty}_+\mbN[1/s^n] \,,
\]
respectively.

\begin{lem}\label{lem:another-seff-map}
Let $X$ be a set. Then $\bigotimes_{X}\bS[z]\longrightarrow \bigotimes_{X}\bS[z^{1/s^{\infty}}]$ is seff. 
\end{lem}

\begin{proof}
It is clear that the maps $\mathbb{S}\to \bS[\mbN]$, $\mathbb{S}\to \bS[\mbN[1/s^{n}]]$, and $\bS[\mbN]\to \bS[\mbN[1/s^{n}]]$ are seff and consequently the map $\bS[\mbN]\to \bS[\mbN[1/s^{\infty}]]$ is also {seff}. By base change, this implies the case when $X$ is finite. The general claim follows by writing $X$ as a filtered colimit of finite sets. 
\end{proof}

\begin{lem}\label{lem:seff-map}
Let $k\to S$ be a map of strongly even $C_2$-$\bE_{\infty}$ rings such that $\pi_{\rho *}^{C_2}S$ is polynomial over $\pi_{\rho *}^{C_2}k$ on generators in degrees divisible by $\rho$. Then the canonical map 
\[\THR(S/k)\to S\]
is seff. 
\end{lem}

\begin{proof}
We compute that the $\mathrm{E}_2$-term of the Hochschild--May spectral sequence 
\[ 
\upi_{\star}\THR(P_*^*S/P_*^*k)\implies \upi_{\star}\THR(S/k)
\]
is 
\[\upi_{\star}\THR(P_*^*S/P_*^*k)\cong \upi_*P_*^*k[x_i :i\in I]\otimes \Lambda (d^{\sigma}x_i) 
\]
where $|d^{\sigma}x_i|=|x_i|+\sigma$. Since $S$ is strongly even, the map 
\[ 
\upi_{\star}\THR(P_*^*S/P_*^*k)\to \upi_{\star}P_*^*S
\]
sends the classes $d^{\sigma}x_i$ to zero. If $C$ is strongly even, we consider the pushout of the diagram
\[ 
P_*^*C=\THR(P_*^*C/P_*^*C)\leftarrow \THR(P_*^*S/P_*^*k) \rightarrow P_*^*S
\]
and note that since $C$ is strongly even the classes $d^{\sigma}x_i$ also map to zero in $P_*^*C$. 
Consequently, the pushout 
\[P_*^*C\otimes_{\THR(P_*^*S/P_*^*k)} P_*^*S
\]
is strongly even and applying the spectral sequence for the filtration 
\[ P_*C\otimes_{\THR(P_*S/P_*k)}P_*S\]
of $\upi_*C\otimes_{\THR(S/k)}S$ we determine that 
$\upi_*C\otimes_{\THR(S/k)}S$ is strongly even as well. It therefore suffices to check that 
\[
C^e\to (C\otimes_{\THR(S/k)}S)^e= C^e\otimes_{\THR(S/k)^e}S^e
\] 
is faithfully flat, which is clear since the underlying map of $\bE_{\infty}$-rings $\THH(S^e/k^e)\to S^e$ is eff 
by~\cite[Proposition~4.2.4]{HRW22}.
\end{proof}

Following~\cite{HP23}, 
we recall the following definition (cf.~\cite[Definition~1.5]{AKGH25}). 

\begin{defin}
Let $E$ be a $C_2$-$\bE_{\infty}$-algebra and let $R$ be a 
$C_2$-$\bE_{\infty}$-$E$-algebra. We define
\[ \HR(R/E):=\THR(R)\otimes_{\THR(E)}E\,.
\]
\end{defin}

\begin{lem}\label{lem:completion}
After $p$-completion, the map 
\[ \THR(\bS[z^{1/p^{\infty}}])\to \bS[z^{1/p^{\infty}}]\]
is an equivalence. 
\end{lem}

\begin{proof}
The unit map $\bS \to H\mZ$ is an isomorphism on Hill--Hopkins--Ravenel $0$-slices, so by a standard connectivity argument, 
the $H\mZ$-nilpotent completion of the sphere spectrum $\bS_{H\mZ}^\wedge$ is equivalent to the sphere spectrum $\bS$. 
As the $H\mZ$-nilpotent completion can be obtained through a cosimplicial resolution of $\bS$ by copies of $H\mZ$, 
we are reduced to proving the result after base change to $H\mZ$. 
That is, it suffices to show 
\[
\HR(\mZ[z^{1/p^{\infty}}]/\mZ)\to H\mZ[z^{1/p^{\infty}}]
\]
is an equivalence after $p$-completion. This follows from \cite[Theorem~4.1]{HP23} and the result that the cotangent complex is 
trivial, i.e. $L_{\bZ[z^{1/p^{\infty}}]/\bZ}\simeq 0$, by~\cite[Proposition~11.7]{BMS19}. 
\end{proof}

Recall \cite[Def. 4.20]{BMS19} that a quasisyntomic ring $S$ is \emph{quasiregular semiperfectoid} if there exists a 
map~$R \to S$ with $R$ perfectoid and the Frobenius of $S/pS$ is surjective. 

\begin{prop}\label{prop:seff-map}
Suppose that $R$ is a $p$-quasisyntomic commutative ring, $S$ is a quasiregular semiperfectoid commutative ring, and $R\to S$ 
is a map of commutative rings. Then 
\[ 
\THR(H\m{R};\bZ_p)\to \THR(H\m{S};\bZ_p)
\]
is seff. 
\end{prop}
\begin{proof}
Let $\bS_{R^{\prime}}=\otimes_{R}\bS[\bar{z}]$ so that the canonical map $\otimes_{R}\bS[\bar{z}]\to \m{R}$ is a $\upi_{*}$-surjection. 
Let $\bS_{S^{\prime}}=\otimes_{R}\bS[z^{1/p^{\infty}}]$. Then $\bS_{R^{\prime}}\to \bS_{S^{\prime}}$ is seff by 
Lemma~\ref{lem:another-seff-map}. 
By Lemma~\ref{lem:seff-map}, we know that $\THR(\mathbb{S}_{R^{\prime}})\to \mathbb{S}_{R^{\prime}}$ is seff. 
Therefore, the composite 
$\THR(\mathbb{S}_{R^{\prime}})_{p}^{\wedge}\to \mathbb{S}_{R^{\prime}}\to  \mathbb{S}_{S^{\prime}}$
is seff and by base change the map $\THR(H\m{R})\to \THR(H\m{S}/\mathbb{S}_{S^{\prime}})$
is seff. 
To see this consider the commutative diagram of pushouts
\[
\begin{tikzcd}
\THR(\bS_{R'})\ar[r] \ar[d] & \bS_{R'}\ar[r] \ar[d] & \bS_{S'}  \ar[d] \\
\THR(R) \ar[r] & \THR(R/\bS_{R'}) \ar[r] &  \THR(S/\bS_{S'})\,.
\end{tikzcd}
\]
By Lemma~\ref{lem:completion}, the map $\THR(H\m{S}/\mathbb{S}_{S^{\prime}};\bZ_p)\simeq  \THR(H\m{S};\bZ_p)$
is an equivalence. 
\end{proof}

\begin{thm}
Let $R$ be $p$-quasisyntomic commutative ring and let $\m{R}$ denote the associated constant Mackey functor.  
Then there are natural equivalences
\[
\filmot^{*}F(\m{R}) \simeq \fil_{\BMSP}^{*}F(\m{R}) 
\]
of $C_{2}$-$\bE_{\infty}$ filtered $C_{2}$-spectra for 
\[ F\in \{\THR(\--;\bZ_p)\,, \TCR^{-}(\--;\bZ_p)\,, \TPR(\--;\bZ_p)\,, \TCR(\--;\bZ_p)\} \,. \]
\end{thm}
\begin{proof}
As in the proof of~\cite[Theorem~5.0.3]{HRW22}, there exists a $p$-quasisyntomic cover $R\to S$ with $S$ quasiregular 
semiperfectoid.  By~\cite[Theorem~8.2]{Par23}, we know that $\THR(H\m{S}; \mathbb{Z}_p)$ is strongly even. 
The proof of~\cite[Theorem~10.1]{Par23} provides an equivalence
\[ 
\fil_{\BMSP}^{s}F(H\m{R})\simeq \lim_{\Delta} P_{2s}F(H\m{S}^{\wedge_{H\m{R}} \bullet+1})
\] 
for $F\in \{\THR(\--;\bZ_p)\,,\TCR^{-}(\--,;\bZ_p)\,,\TPR(\--;\bZ_p)\,,\TCR(\--;\bZ_p)\}$. 
The map 
\[
\THR(H\m{R};\bZ_p)\longrightarrow\THR(H\m{S};\bZ_p)
\]
is seff by Proposition~\ref{prop:seff-map}, so the result follows from Theorems \ref{thm:descent} and~\ref{thm:motivic-descent}, 
which identify 
\[ 
\lim_{\Delta} P_{2s}F(H\m{S}^{\wedge_{H\m{R}} \bullet+1})\simeq \filmot^*F(H\m{R}) 
\]
for $F\in \{\THR(\--;\bZ_p)\,,\TCR^{-}(\--,;\bZ_p)\,,\TPR(\--;\bZ_p)\,,\TCR(\--;\bZ_p)\}$.
\end{proof}

We can also compare to the HKR filtration from Hornbostel--Park \cite{HP23} and Yang~\cite{Yan25} 
under the additional hypotheses that ensure that these filtrations agree.

Given a commutative ring $k$ and an $k$-algebra $R$, we write $\fil_{\HKR}^{\bullet}\mathrm{HR}(\m{R}/\m{k})$ for the 
filtration constructed in~\cite[Theorem~4.30]{HP23} and \cite[\S~7.4]{Yan25} where $\m{R}$ and $\m{k}$ are the associated 
constant Mackey functors.

\begin{prop}
If $k\to R$ is a quasi-lci map of commutative rings, then there is an equivalence 
\[ 
\fil_{\mot}^{\bullet}\THR(\m{R}/\m{k})\simeq \fil_{\HKR}^{\bullet}\mathrm{HR}(\m{R}/\m{k}) \,.
\]
when either $R$ is smooth over $k$ or $1/2\in k$. 
\end{prop}

\begin{proof}
We let $S=k[r\in R]$ denote the polynomial $k$-algebra on all elements $r\in R$, which is equipped with a surjection $S\to R$. 
We showed that $\THR(S/k)\to S$ is seff in Lemma~\ref{lem:seff-map} and consequently the map 
$\THR(\m{R})\to \THR(\m{R})\otimes_{\THR(\m{S}/\m{k})}S$ is seff and 
\[ 
\fil_{\mot}^*\THR(\m{R}/\m{k})\simeq \lim_{\Delta}\fil_{\mot}^*\THR(\m{R}/\m{S}^{\otimes_{\m{k}\bullet+1}})\,.
\]
We claim that 
\[
\fil_{\mot}^*\THR(\m{R}/\m{S}^{\otimes_{\m{k}\bullet+1}})\simeq \fil_{\HKR}^*\mathrm{HR}(\m{R}/\m{S}^{\otimes_{\m{k}\bullet+1}})\,.
\]
To see this, we note that 
\[
\gr_{\HKR}^n\mathrm{HR}(\m{R}/\m{S}^{\otimes_{\m{k}\bullet+1}}) \simeq 
\m{\Lambda_R^n\mathbb{L}^{\textup{alg}}_{R/S^{\otimes_{k}\bullet+1}}}[n\sigma]
\]
by~\cite[Theorem~4.30--4.31]{HP23} (cf.~\cite[\S~7.4]{Yan25}).\footnote{Here we write  $\mathbb{L}^{\textup{alg}}_{A/B}$ instead of $\mathbb{L}_{A/B}$ to be consistent with the 
notation in~\cite{HRW22}.} 
Here the shift $[n\sigma]$ is explained in~\cite[Definition~2.8]{HP23}.
Since $\mathbb{L}_{R/S^{\otimes_k \bullet}}$ has Tor amplitude concentrated in homological degree $1$, 
we conclude that the $C_2$-spectrum $\THR(\m{R}/\m{S}^{\otimes_{\m{k}} \bullet+1})$ is strongly even 
whenever $R$ is smooth over $k$ or $1/2\in k$. 
Consequently, we can identify 
\[ 
\grmot^*\THR(\m{R}/\m{S}^{\otimes_{k}{\bullet+1}})=\Sigma^{\rho *}H\m{\pi_{2*}\THH(R/S^{\otimes_k\bullet+1})}=\m{\Lambda_R^*\mathbb{L}^{\textup{alg}}}_{R/S^{\otimes_{k}\bullet+1}}[*\sigma]
\]
by~\cite{BMS19}. 
Therefore, we can identify the two filtrations after passing to associated graded. 
The result then follows because, under the hypothesis that either $R$ is smooth over $k$ or $1/2\in k$, 
the filtration $\fil_{\textup{HKR}}^{\bullet}$ is complete. It therefore suffices to show that 
\[ 
\m{\Lambda_k^*\mathbb{L}}_{R/k}^{\textup{alg}}[*\sigma] \simeq \lim_{\Delta} \m{\Lambda_R^*\mathbb{L}^{\textup{alg}}}_{R/S^{\otimes_{k}\bullet+1}}[*\sigma]
\]
but since each side is a constant Mackey functor it suffices to check this on the underlying, where it follows as in the proof of~\cite[Theorem~5.0.2]{HRW22}. 
\end{proof}

\begin{remark}
It would be quite interesting to compare the more general construction of the HKR filtration in~\cite{Yan25} 
to our motivic filtration on Real topological Hochschild homology, but this is beyond the scope of the present paper. 
In particular, it would be desirable to remove the assumption that $1/2\in k$. 
See~\cite[\S~7.4]{Yan25} for some discussion of how this assumption affects the computations. 
\end{remark}

\section{Real Hochschild homology}\label{sec:THR}
In this section, we compute the Real topological Hochschild homology of $\BPR\langle n\rangle$ and its motivic filtration.

\subsection{Real topological Hochschild homology}\label{thr}
We will be interested in the cohomology of the $\RO(C_{2})$-graded Mackey functor-valued 
Hopf algebra\footnote{Since \(S^\rho\) is equivariantly connected, this Hopf algebroid is in fact a Hopf algebra.}
\[ 
\left(\upi_{\star}\MUR, \upi_{\star} S^{\rho}\otimes \MUR\right) \,. 
\]
We first recall some results from \cite{DMPR21} and~\cite{HHKWZ20}. 

\begin{prop}\label{prop:thrmur}
There is an isomorphism of $\upi_{\star}\MUR$-algebras 
\[ 
    \upi_{\star}\THR(\MUR)\cong \upi_{\star}\MUR\langle \bar{\lambda}_{i}^{\prime} : i\ge 1 \rangle 
\]
where $| \bar{\lambda}_{i}^{\prime}|=\rho i  +\sigma $. 
Moreover, there is an isomorphism of $\upi_{\star}\MWR$-algebras 
\[ 
\upi_{\star}\THR(\MWR)=\upi_{\star}\MWR\langle \olambda'_i,\olambda''_j : i\ge 1, j\in J \rangle
\]
where $|\olambda'_i|=\rho i+\sigma$ and $|\olambda''_j|=k_j\rho +\sigma$ where $J$ is the same indexing set as in Proposition~\ref{prop:MW}. 
\end{prop}

\begin{proof}
By~\cite[Corollary 7.1.3,~Lemma 7.2.1,~Lemma~7.2.2]{HHKWZ20}, we can identify $\THR(\MUR)$ and $\THR(\MWR)$ with 
\[ \MUR\otimes B^{\sigma}\Omega^{\infty}\MUR_+ 
\text{ and } \MWR\otimes B^{\sigma}\Omega^{\infty}\Sigma^{\rho}\MUR_+
\]
respectively. Note that we can identify 
\[ 
\MUR\otimes B^{\sigma}\Omega^{\infty}\MUR_+\simeq \MUR\otimes B^{\sigma}\BUR_+ \,.
\]
At the underlying level, $\THR(\MUR)^e\simeq \THH(\MU)$, so it suffices to check that the we have the correct 
$C_2$-geometric fixed points. On $C_2$-geometric fixed points, we note that 
\[
\Phi^{C_2}\THR(\MUR)\simeq \MO\otimes \OO_+
\]
and the result then follows because the rings $H_*(\OO;\bF_2)$ is exterior on classes in the relevant degrees 
by~\cite[Theorem~7.1]{Mil74}. This proves the first sentence. 

By~\cite[Theorem~2]{HH18}, we know that $H\mZ\otimes \Omega^{\infty}\Sigma^{\rho}\MUR_+$ is strongly even. 
The underlying was computed in~\cite[Theorem~3.3,~Corollary~3.4]{Wil73}. 
We conclude that $\pi_{\star}^{C_2}H\mZ\otimes \Omega^{\infty}\Sigma^{\rho}\MUR_+$ is polynomial on classes in 
$\rho$-multiple degrees. 

We claim that the computation 
\[
\pi_{\star}^{C_2}(H\mZ\otimes B^{\sigma}\Omega^{\infty}\Sigma^{\rho}\MUR_+)
\cong \pi_{\star}H\mZ\langle \overline{\lambda}'_i,\overline{\lambda}''_j : i\ge 1, j\in J \rangle
\]
follows from the $\RO(C_2)$-graded twisted bar spectral sequence~\cite[\S~6.5]{Pet24}, as we now explain. 
We follow the same strategy as \cite[Example~6.5]{Pet24}, the key point being that $\Omega^\infty\Sigma^{\rho}\MUR$ 
has $H\mZ$-free homology in the sense of~\cite{Hil22}, by~\cite[Theorem~2]{HH18}. 
The fact that the spectral sequence collapses at the $\EE_1$-page can be checked on underlying. 
In particular, on underlying we know that $H\bZ_*\Omega^{\infty}\Sigma^{2}\MU_+$ is polynomial on even degrees 
by~\cite[Theorem~3.3,~Corollary~3.4]{Wil73}. 
The bar spectral sequence computing $H\bZ_*B\Omega^{\infty}\Sigma^{2}\MU_+$ has $\EE_1$-page an exterior algebra on classes 
in degree one more that the generators of $H\bZ_*\Omega^{\infty}\Sigma^{2}\MU_+$.  
The differentials on these algebra generators land in zero groups and therefore the spectral sequence collapses 
at the $\mathrm{E}_1$-page. Any differentials in the twisted bar spectral sequence would affect the known underlying abutment, 
which would lead to a contradiction. 
The fact that there are no multiplicative extensions follows from examining the $\mathrm{RO}(C_2)$-grading. 

We then apply the spectral sequence associated to the filtration 
\[ 
(P_{\bullet}\MWR)\otimes B^{\sigma}\Omega^{\infty}\Sigma^{\rho}\MUR_+
\]
which collapses at the $E_2$-term to produce the result. 
To see that it collapses, note that the slice spectral sequence for $\MWR$ collapses by 
Proposition~\ref{prop:MW} and any differentials would contradict the known underlying computation. 
The fact that there are no hidden multiplicative extensions can be checked on the underlying. 
\end{proof}

The following corollary is immediate by base change. 

\begin{cor}
There is an isomorphism of $\pi_{\star}^{C_2}\mF_{2}$-algebras 
\[
    \pi_{\star}^{C_2}\THR(\MUR;\mF_{2})\cong 
	\pi_{\star}^{C_2}\mF_{2}\langle  \lambda_{i}^{\prime} : i\ge 1 \rangle
\]
and an isomorphism of $\pi_{\star}^{C_2}\mF_2$-algebras 
\[
    \pi_{\star}^{C_2}\THR(\MWR;\mF_{2})\cong 
	\pi_{\star}^{C_2}\mF_{2}\langle  \lambda_{i}^{\prime} ,\lambda''_j: i\ge 1 ,j\in J\rangle \,.
\]
\end{cor}

We also recall a result of \cite{DMPR21}. 

\begin{prop}[{\cite[\S~5.3]{DMPR21}}, {\cite[Corollary~7.2.3]{HHKWZ20}}]
There is an isomorphism of $\pi_{\star}^{C_2}\mF_2$-algebras
\[ 
    \pi_{\star}^{C_2}\THR(\mF_2)=\pi_{\star}^{C_2}\mF_2[\omu_{0} ]
\]
where $|\omu_{0}|=\rho$. 
\end{prop}

\begin{convention}
We fix ($2$-complete) $\mathbb{E}_\infty$-forms of $\BPR\langle n\rangle$ for $-1 \leq n \leq 2$ as follows. Let $\BPR\langle -1\rangle:=\mF_2$, $\BPR\langle 0\rangle :=\mZ_2$, $\BPR\langle 1\rangle :=(\kr)_2$ and $\BPR\langle 2\rangle:=\tmf_1(3)_2$.  
\end{convention}

\begin{prop}\label{prop:thr}
Let $0 \le n\le 2$. 
Then there is a canonical isomorphism 
\[ 
	\upi_{\star}\mF_2[\bar{\mu}^{2^{n+1}}]\langle \olambda_1,\cdots ,\olambda_{n+1} \rangle 
 \cong \upi_{\star}\THR(\BPRn;\mF_2)  \,.
\]
\end{prop}

\begin{proof}
We will construct a map 
\[
\mF_2[\bar{\mu}^{2^{n+1}}]\langle \olambda_1,\cdots ,\olambda_{n+1} \rangle\overset{\simeq}{\longrightarrow} \THR(\BPRn; \mF_2) 
\]
and show that it is an equivalence by checking that it is an equivalence on underlying spectra and geometric fixed points. 
Here we write 
\[
\mF_2[x]:=\mF_2\otimes \mathrm{Free}_{\mathbb{E}_1}(S^{|x|})
\] 
for the free $\mathbb{E}_1$ algebra on a class in degree $|x|$ and define
\[
\mF_2[\omu^{2^{n+1}}]\langle \olambda_1,\cdots ,\olambda_{i}\rangle 
:=\mF_2[\omu^{2^{n+1}}]\otimes_{\mF_2[\ot_1,\cdots, \ot_{n+1}]}\mF_2 
\]
where $|\bar{\mu}^{2^{n+1}}|=2^{n+1}\rho$ and $|\ot_i|=\rho (2^i-1)$.

We consider the multiplicative K{\"u}nneth spectral sequence for
\[
	\Phi^{C_2}\THR(\BPRn; \mF_2) \simeq \Phi^{C_2} \mF_2 \otimes_{\BPn} \Phi^{C_2} \BPRn \,,
\]
which has signature
\[
	E_2^{s,t} = \Tor_{s,t}^{\pi_*\BPn}(\pi_*\Phi^{C_2} \mF_2, \pi_*\Phi^{C_2}\BPRn) 
	\Rightarrow \pi_*\Phi^{C_2}\THR(\BPRn; \mF_2)\,.
\]
We have
\[
	\pi_*\Phi^{C_2}\mF_2 \cong \bF_2[x]
\]
with $|x|=1$ and
\[
	\pi_*\Phi^{C_2}\BPRn \cong \bF_2[x_{2^{n+1}}]
\]
with $|x_{2^{n+1}}|=2^{n+1}$ and so 
\[
	E_2^{**} = \Tor_{**}^{\bZ_{(2)}[v_1,\ldots,v_n]}(\bF_2[x],\bF_2)[x_{2^{n+1}}] 
	\cong \Lambda_{\bF_{2}}(\sigma v_0,\sigma v_1,\ldots, \sigma v_n) \otimes \bF_2[x, x_{2^{n+1}}] \,.
\]
The spectral sequence collapses at the $\EE_2$-page because the indecomposable algebra generators are in 
K\"unneth filtration $\le 1$. Since $n\le 2$, there is no room for  multiplicative extensions. 

We therefore have permanent cycles $\sigma v_i$, $x$, and $x_{2^{n+1}}$. We can therefore produce a map 
\[ 
\mF_2[S^{2^{n+1}\rho}]\otimes_{\mF_2[S^{2\rho}]}\mF_2\otimes_{\mF_2[S^{2^2\rho}]}\cdots 
\otimes_{\mF_2[S^{2^{n+1}\rho}]}\mF_2\to \THR(\BPR\langle n\rangle;\mF_2)
\]
which is an equivalence on geometric fixed points. 
On underlying this is an equivalence by~\cite[Proposition~2.7]{AKCH21}. 
Consequently, the map is an equivalence as desired. 
\end{proof}

\subsection{The motivic filtration on $\THR$}\label{sec:motivic-thr}
By Section~\ref{sec:descent}, seff covers are of fundamental importance for understanding the motivic filtration. 
We therefore begin by determining a seff cover of $\THR(\BPR\langle n\rangle)$. 

\begin{prop}\label{prop:seff-MWR}
The map  $\THR(\MWR)\to \MWR$ is seff. 
\end{prop}

\begin{proof}
Let $C\ne 0$ be a strongly even $C_2$-$\bE_{\infty}$-ring equipped with a map~$\THR(\MWR)\longrightarrow C$ 
of $C_2$-$\bE_{\infty}$-rings and consider the $C_2$-$\bE_{\infty}$-$C$-algebra 
\[
C\otimes_{\THR(\MWR)} \MWR  \,.
\]
By Proposition~\ref{prop:thrmur}, there is an isomorphism of $\pi_{\star}^{C_2}\MWR$-algebras 
\[ 
    \pi_{\star}^{C_2}\THR(\MWR)\cong \pi_{\star}^{C_2}\MWR\langle \bar{\lambda}_{i}^{\prime} : i\ge 1 \rangle 
\]
where $| \bar{\lambda}_{i}^{\prime}|=\rho i  +\sigma $.
Since $C$ is even, the exterior generators map to zero in $\pi_{\star}^{C_2}C$ 
(note that $\rho i + \sigma = \rho (i+1)-1$). 
Consequently, 
\[
\pi_{\star}^{C_2}C\otimes \Gamma (d\bar{\lambda}_i)
\]
is the associated graded of a filtration on 
\[ 
\upi_{\star}(C\otimes_{\THR(\MWR)}\MWR)
\]
where $|d\bar{\lambda}_i|$ is divisible by $\rho$. 
We know that 
\[
C\to C\otimes_{\THR(\MWR)}\MWR 
\]
is faithfully flat on underlying and as a $C$-module by~\cite[Proposition~3.2.3,~Example~4.2.3]{HRW22}. 
This implies that 
\[
C\otimes_{\THR(\MWR)}\MWR
\]
is a wedge of $\rho$-multiple suspensions of $C$, proving the  that 
\[ 
\THR(\MWR)\longrightarrow \MWR
\] 
is seff. 
\end{proof}

\begin{convention}
In the remainder of Section~\ref{sec:motivic-thr}, we implicitly work in the $2$-complete setting and omit $2$-completion from our notation. We also write $\fil_{\sev}$ and $\gr_{\sev}$ in place of $\fil_{\sev,2}$ and $\gr_{\sev,2}$ respectively. 
\end{convention}

The following result is immediate by base change along the map 
\[ 
\THR(\MWR)\longrightarrow \THR(\BPR\langle n\rangle)
\]
induced by the map from Theorem~\ref{thm: examples of Real cyclotomic bases}. 

\begin{cor}
The map 
\[ 
\THR(\BPR\langle n\rangle)\longrightarrow \THR(\BPR\langle n\rangle/\MWR)
\]
is $2$-completely seff for $-1\le n\le 2$. 
\end{cor}

\begin{lem}\label{stronglyevenlemma}
Suppose $E$ is a strongly even $C_2$-$\bE_{\infty}$-$\mF_2$-algebra such that $\pi_{*}^eE$ is a 
polynomial algebra with set of indecomposable algebra generators $\{x_i : i\in I\}$ in degree $|x_i|=2k_i$. Then 
\[
\pi_{\star}^{C_2}E\cong \pi_{\star}^{C_2}\mF_2[\ol{x}_i : i\in I]
\]
where $|\overline{x}_i|=\rho k_i$. If $F$ is a strongly even $C_2$-$\bE_{\infty}$-$\mF_2$-algebra and 
\[ 
\pi_*^eF\cong \bF_2[x_i : i\in I][u,u^{-1}]
\]
where $|x_i|=2k_i$ and $|u|=2m$, then 
\[
\pi_{\star}^{C_2}F\cong \pi_{\star}^{C_2}\mF_2[\ol{x}_i : i\in I][\ol{u},\ol{u}^{-1}]
\]
where $|\ol{x}_i|=\rho  k_i$ and $|\ol{u}|=\rho m$.
\end{lem}

\begin{proof}
Let $E$ be a strongly even $C_2$-$\bE_{\infty}$-$\mF_2$-algebra such that $\pi_*^eE=\bF_2[x_i : i\in I]$. 
Then since $x_i\in \pi_{2k_i}E \cong \pi_{\rho k_i}^{C_2}E$ there exists an associated $C_2$-equivariant 
map of $\bE_1$-$\mF_2$-algebras
\[ 
\mF_2\otimes \mathrm{Free}_{\bE_1}(S^{\rho k_i})\longrightarrow E 
\]
for each $x_i$ and consequently a map 
\[
\mF_2\otimes (\bigotimes_{i\in I}\mathrm{Free}_{\bE_1}(S^{\rho k_i}))\longrightarrow E 
\]
inducing an isomorphism
\[
P_*^*(\mF_2\otimes (\bigotimes_{i\in I}\mathrm{Free}_{\bE_1}(S^{\rho k_i})))\longrightarrow P_*^*E .
\]
Consequently, there is an equivalence 
\[
\mF_2\otimes (\bigotimes_{i\in I}\mathrm{Free}_{\bE_1}(S^{\rho k_i}))\simeq E  
\]
and since the source slice spectral sequence collapses, the result follows. 
\end{proof}

\begin{thm}\label{thm:s-even}
Let $-1\le n\le 2$. 
Then there is an isomorphism of $\pi_{\star}^{C_2}H\mF_2$-algebras  
\[ 
	\pi_{\star}^{C_2}\THR(\BPR\langle n\rangle/\MWR;\mF_2)\cong \pi_{\star}^{C_2}\mF_2[\bar{\mu}^{2^{n+1}}]\otimes P
\] 
where $P$ is generated by underlying elements in degrees divisible by $\rho$. 
Moreover, we have that the $C_2$-spectrum $\THR(\BPR\langle n\rangle/\MWR)$ is strongly even. 
\end{thm}

\begin{proof}
It suffices to compute the K\"unneth spectral sequence~\cite{LM07}:
\[ 
\Tor_{*,*}^{\pi_{\star}^{C_2}\THR(\MWR;\mF_2)}(\pi_{\star}^{C_2}\THR(\BPR\langle n\rangle;\mF_{2}),\pi_{\star}^{C_2}\mF_{2})
\implies \pi_{\star}^{C_2}\THR(\BPR\langle n\rangle/\MWR;\mF_2) \,.
\]
We know that 
\[
\pi_{\star}^{C_2}\THR(\MWR;\mF_{2})\cong \pi_{\star}^{C_2}\mF_{2}\langle \lambda'_i , \lambda''_j :i\ge 1 ,j\in J\rangle . 
\]
By Theorem~\ref{prop:thr}, we also know that 
\[
\pi_{\star}^{C_2}\THR (\BPR\langle n\rangle;\mF_{2})
\cong \pi_{\star}^{C_2}\mF_{2}[\bar{\mu}^{2^{n+1}}]\langle \olambda_1,\cdots ,\olambda_{n+1}\rangle \,,
\]
and the unit map 
\[
\pi_{\star}^{C_2}\THR(\MWR;\mF_{2})\to \pi_{\star}^{C_2}\THR(\BPR\langle n\rangle;\mF_{2})
\]
sends $\ol{\lambda}'_{2^{i}-1}$ to $\ol{\lambda}_{i}$ for $1\le i\le n+1$ and the remaining exterior 
generators are sent to zero. The input of the spectral sequence is therefore 
\[ 
	\Gamma (d \ol{\lambda}'_{i},d\ol{\lambda}''_j : i\ne 2^k-1, 1\le k\le n+1,j\in J)
	\otimes \pi_{\star}^{C_2}\mF_{2}[\overline{\mu}^{2^{n+1}}]\,.
\] 
Here $|d\ol{\lambda}'_i|=|\ol{\lambda}_i|+1$ and $|d\ol{\lambda}''_i|=|\ol{\lambda}_j|+1$ 
are in $\rho$-multiple degrees. Already, this implies that the abutment is strongly even for bidegree reasons.  
We can further determine that the spectral sequence collapses at the $\EE_2$-page using Mackey functor structure. 
The restriction map $\mathrm{res}_e^{C_2}$ produces a map to the K\"unneth spectral sequence of underlying spectra
\[ 
\Tor_{\ast}^{\pi_{\ast}\THH(\MW;\bF_2)}(\pi_{\ast}\THH(\BP\langle n\rangle;\bF_{2}),\bF_{2})
\implies \pi_{\ast}\THH(\BP\langle n\rangle/\MW;\bF_2)
\]
whose $E_2$-term can be identified with 
\[ 
	\Gamma (d \lambda'_{i},d\lambda''_j : i\ne 2^k-1, 1\le k\le n+1,j\in J)\otimes \bF_{2}[\mu^{2^{n+1}}]\,,
\] 
which is concentrated in even degrees (cf.~\cite[Theorem~2.5.4]{HW22}). 
Therefore, the latter K\"unneth spectral sequence collapses at the $\mathrm{E}_2$-page. 
The restriction map is an isomorphism between degrees $\rho k$ in the source and degrees $2k$ in the target so any 
differentials on algebra generators would contradict the known underlying computation. 

Since we already deduced that the spectrum is strongly even, we can also resolve  multiplicative extensions 
using this underlying computation and Lemma~\ref{stronglyevenlemma} (cf.~\cite[Theorem~2.5.4]{HW22}). 
This gives the desired isomorphism
\[  
\pi_{\star}^{C_2}\THR(\BPR\langle n\rangle/\MWR;\mF_2)\cong \pi_{\star}^{C_2}\mF_{2}[\overline{\mu}^{2^{n+1}}]\otimes P\,.
\]

The last claim follows from the fact that $\BPR\langle n\rangle$ is strongly even using the Bockstein spectral sequences 
for $\bar{v}_0, \ \bar{v}_1, \ \ldots , \ \bar{v}_n$. 
\end{proof}

Fix $-1\le n\le 2$ and set $A=\pi_{\star}^{C_2}\THR(\BPR\langle n\rangle ;\mF_2)$ and 
$B=\pi_{\star}^{C_2}\THR(\BPR\langle n\rangle/\MWR;\mF_{2})$. 
Then there is a K\"unneth spectral sequence
\[ 
\Tor_{*}^{A}(B,B)\implies \pi_{\star}^{C_2}\THR(\BPR\langle n\rangle/\MWR^{\otimes 2};\mF_{2})
\]
that collapses at the $\EE_{2}$-page. 
To see this note that the map 
\[\pi_{\star}^{C_2}\THH(\BPR\langle n\rangle ;\mF_2)\longrightarrow \pi_{\star}^{C_2}\THR(\BPR\langle n\rangle/\MUR;\mF_2)\]
sends $\ol{\mu}^{2^n}$ to $\ol{\mu}^{2^n}$ and the remaining algebra generators besides those in 
$\pi_{\star}^{C_2}\mF_2$ to zero. 
We can therefore identify the $\EE_2$-page of the K\"unneth spectral sequence with 
\[
\Gamma(d\ol{\lambda}_1,\cdots ,d\ol{\lambda}_{n+1} )\otimes \pi_{\star}^{C_2}\mF_2[\ol{\mu}^{2^{n+1}}]\otimes \Gamma (d \ol{\lambda}'_i ,d \ol{\lambda}''_j : i\ge 1 ,j \in J )^{\otimes 2}\otimes P^{\otimes 2} \,.
\]
and since $\THH(\BPR\langle n\rangle/\MUR^{\otimes 2};\mF_2)$ is strongly even, 
it suffices note that the underlying spectral sequence collapses at the $\EE_2$-page (cf.~\cite[Proposition~6.1.6]{HW22}) 
and therefore any differentials in the K\"unneth spectral sequence would contradict the underlying computation. 

Let $\ol{\Sigma}= \Tor_{\star}^{A}(B,B)$ and $\Sigma =\pi_{\star}^{C_2}\THR(\BPR\langle n\rangle/\MWR^{\otimes 2};\mF_{2})$. 
We can therefore identify the cobar complex for the Hopf algebroid $(B,\ol{\Sigma})$ with the associated graded of a filtration 
on the cobar complex for the Hopf algebroid $(B,\Sigma)$. 
This produces a spectral sequence 
\[ 
\Cotor_{(B,\ol{\Sigma})}^{\ast,\ast,\star}(B,B)\implies \Cotor_{(B,\Sigma)}^{\ast,\star}(B,B)
\]
that we call the \emph{May--Ravenel spectral sequence} after~\cite{May66} and~\cite{Rav86}. 

We also note that there is a motivic spectral sequence 
\[
	\pi_{\star}^{C_2}\gr_{\mot}^*\THR(\BPR\langle n\rangle ;\mF_{2})\implies \pi_{\star}^{C_2}\THR(\BPR\langle n\rangle ;\mF_{2}) 
\]
and we can identify 
\[ 
\pi_{\star}^{C_2}\gr_{\mot}^*\THR(\BPR\langle n\rangle ;\mF_{2})=\Cotor_{(B,\Sigma)}^{\ast,\star}(B,B)\,.
\]

\begin{thm}\label{thm:thr}
Let $-1\le n\le 2$. The May--Ravenel spectral sequence and the motivic spectral sequence each collapse at the $\mathrm{E}_2$-page. 
Consequently, we have an isomorphism
\[ 
\pi_{\star}^{C_2}\grmot^{*}\THR(\BPR\langle n\rangle ;\mF_{2})
= \pi_{\star}^{C_2}\mF_{2}[\omu^{2^{n+1}}]\langle \olambda_1,\cdots,\olambda_{n+1}\rangle 
\]
where the generators $\bar{\lambda}_{1},\ldots ,\bar{\lambda}_{n+1}$ have 
Adams weight $1$ and $\bar{\mu}^{2^{n+1}}$ has Adams weight $0$.  
Here $|\bar{\lambda}_{i}|=2^i\rho-1$ and $|\bar{\mu}^{2^{n+1}}|=2^{n+1}\rho$.
\end{thm}

\begin{proof}
Note that $(B,\ol{\Sigma})$ is a tensor product of the Hopf algebroids $(P,P\otimes P)$,
\[
(\Gamma(d \ol{\lambda}_i ,d \ol{\lambda}_j : i\ge 1,j\in J),\Gamma(d \ol{\lambda}_i ,d \ol{\lambda}_j : i\ge 1,j\in J)
\otimes \Gamma(d \ol{\lambda}_i ,d \ol{\lambda}_j : i\ge 1,j\in J))\,, 
\]
$(\pi_{\star}^{C_2}\bF_2,\Gamma (d\ol{\lambda}_1,\cdots , d\ol{\lambda}_{n+1}))$ 
and $(\pi_{\star}^{C_2}\bF_2[\ol{\mu}^{2^{n+1}}],\pi_{\star}^{C_2}\bF_2[\ol{\mu}^{2^{n+1}}])$. 
The first two Hopf algebroids have trivial cohomology, so we compute directly that
\[
\Cotor_{(B,\overline{\Sigma})}^{\ast,\ast,\star}(B,B)\cong \pi_{\star}^{C_2}\mF_2[\bar{\mu}^{2^{n+1}}]\langle \bar{\lambda}_1,\cdots,\bar{\lambda}_{n+1}\rangle  
\]
and therefore both the May--Ravenel spectral sequence and the motivic spectral sequence collapse without extensions 
since we know the abutment by Proposition~\ref{prop:thr}. 
\end{proof}

In order to preserve equivariance in addition to multiplicative structure, 
we make a further identification.

\begin{lem}
Let $-1\le n\le 2$. There is an equivalence of $\bE_{\infty}$-$\gr_{\sev}^*\bS$-algebras 
\[ 
\grmot^*\THR(\BPR\langle n\rangle ;\mF_2)
\simeq \left ( \grmot^*\THR(\BPR\langle n\rangle ) \right /(\ol{v}_0,\cdots ,\ol{v}_n)
\]
in the $\infty$-category $\mathrm{Gr}({\Sp}^{C_2})$. 
\end{lem}
\begin{proof}
Since $\THR(\MWR)\to \MWR$ is seff by Proposition~\ref{prop:seff-MWR}, it suffices to note that  
\[ \upi_{\star}\THR(\BPR\langle n\rangle/\MWR)
\]
is $(2,\bar{v}_1,\cdots ,\bar{v}_n)$ -torsion free (cf.~\cite[Theorem 6.18(4)]{BHS22}), which follows from 
Theorem~\ref{thm:s-even}. 
\end{proof}

\begin{defin}\label{def:varepsilon}
Fix $-1\le n\le 2$. Let 
\[ 
\bar{v}_{n+1} : \Sigma^{\rho(2^{n+1}-1),0}\grsev^*\bS/(\ol{v}_0,\cdots ,\ol{v}_n)
\to \grsev^*\bS/(\ol{v}_0,\cdots ,\ol{v}_n)
\]
denote the $\grsev^*\mathbb{S}$-module map corresponding to the $\upi_{\star}\MUR\otimes \MUR$-comodule map 
\[
\Sigma^{\rho(2^{n+1}-1)}\upi_{\star}\MUR\to \upi_{\star}\MUR
\] 
so that there is a cofiber sequence 
\[
\Sigma^{\rho (2^{n+1}-1),0}\grsev^*\bS/(\ol{v}_0,\cdots ,\ol{v}_n)\overset{\bar{v}_{n+1}}{\longrightarrow} \grsev^*\bS/(\ol{v}_0,\cdots ,\ol{v}_n) \overset{i_{n+1}}{\longrightarrow} \grsev^*\bS/(\ol{v}_0,\cdots ,\ol{v}_{n+1})
\]
of $\grsev^*\mathbb{S}$-modules with boundary map
\[
\grsev^*\bS/(\ol{v}_0,\cdots ,\ol{v}_{n+1})\overset{j_{n+1}}{\longrightarrow} 
\Sigma^{\rho (2^n-1)+1,-1}\grsev^*\bS/(\ol{v}_0,\cdots ,\ol{v}_n)\,,
\] 
where the suspension is defined as in Notation~\ref{not:deg-Aw-suspension}. 
The induced map 
\[
\ol{v}_{n+1}\otimes \id :\Sigma^{\rho (2^{n+1}-1),0}\grsev^*\BPR\langle n\rangle/(\ol{v}_0,\cdots ,\ol{v}_n) 
\to \grsev^*\BPR\langle n\rangle/(\ol{v}_0,\cdots ,\ol{v}_n)
\]
is nullhomotopic. We define a class 
\[
\overline{\varepsilon}_{n+1}\in \pi_{\star}^{C_2}\grsev^*\BPR\langle n\rangle/(\ol{v}_0,\cdots ,\ol{v}_{n+1}) 
\]
to be the unique class in bidegree 
\[
\| \bar{\varepsilon}_{n+1} \|=(\rho (2^{n+1}-1)+1,-1)
\]
satisfying 
\[
j_{n+1}(\bar{\varepsilon}_{n+1})=1\,.
\]
In particular, the first $\ol{v}_{n+1}$-Bockstein $\ol{\beta}_{1,n+1}$ satisfies 
\begin{align}\label{Bocsktein}
\ol{\beta}_{1,n+1}(\ovarepsilon_{n+1})=1\,.
\end{align}
We have $\ovarepsilon_{n+1}^2=0$ since this bidegree of 
\[
\pi_{\star}^{C_2}\grsev^*\BPR\langle n\rangle/(\ol{v}_0,\cdots ,\ol{v}_{n+1})
\] 
is trivial. Consequently, for any $\grsev^*\BPR\langle n\rangle$-module $M$ we have 
\[
\pi_{\star}^{C_2}M/(\ol{v}_0,\cdots ,\ol{v}_{n+1})
\cong \pi_{\star}^{C_2}M/(\ol{v}_0,\cdots ,\ol{v}_n)\langle \varepsilon_{n+1}\rangle\,.
\]
\end{defin}

In light of this, we have the following corollary of the work in this section. 

\begin{cor}\label{cor:vn-thr}
Let $-1 \le n\le 2$. 
Then there is a preferred isomorphism
\[ 
\pi_{\star}^{C_2}\grmot^*\THR(\BPR\langle n\rangle)/(\ol{v}_0,\cdots ,\ol{v}_{n+1})\cong \pi_{\star}^{C_2}\mF_2[\omu^{2^{n+1}}]\langle \olambda_1,\cdots ,\olambda_{n+1}, \ovarepsilon_{n+1} \rangle 
\]
of $\pi_{\star}^{C_2}\mF_2\langle \overline{\lambda}_1,\overline{\lambda}_2,\cdots \rangle$-algebras
where $\bar{\varepsilon}_{n+1}$ has stem $\rho (2^{n+1}-1)+1$ and Adams weight $-1$. 
\end{cor}

\section{A toolkit}\label{sec:tool}
In this section, we collect various tools that will be useful for our computations of Real syntomic cohomology. 

\subsection{The zero-slice of Real topological cyclic homology}
The following lemma will be useful in several of our computations. 
\begin{lem}\label{lem:pi0tcr}
Suppose $R$ is a connective $C_2$-$\bE_{\infty \rho}$-ring such that $\upi_0R \in \{\mZ_2,\mF_2\}$. Then 
\[
    P_0^0\TCR(R)=\mZ_2\,.
\]
\end{lem}

\begin{proof}
By~\cite[Theorem~7.12]{AKGH25} for example, we know that 
\[
\upi_0\TCR(R)\cong \upi_0\TCR(\upi_{0}R)\,. 
\]
Since $\mZ_2$ and $\mF_2$ are cohomological, we know from ~\cite[Theorem~3.7]{DMP22} that  
\[
\upi_0\TCR(\mZ_2)=\ker \left ((1-F):\underline{W(\mathbb{Z})}\to \underline{W(\mathbb{Z}_2)} \right )=\mZ_2
\] 
(cf.~\cite[Corollary~8.26]{AKCH21}) and 
\[\upi_0\TCR(\mF_2)=\mZ_2=\ker \left ((1-F):\underline{W(\mathbb{F}_2)}\to \underline{W(\mathbb{F}_2)} \right )\,.
\]
The result follows by combining these results with the general fact that 
\[ 
P_0^0\TCR(R)\simeq \underline{\pi}_0\TCR(R)
\]
from~\cite[Corollary~8.9]{Ull-thesis}. 
\end{proof}

\subsection{Power operations}\label{sec:toolkit-power}
Given a $\bE_{\rho}$ $\mF_{2}$-algebra $A$~\cite{GM17}, then we have a map 
\[ 
(S(\rho)_+\otimes \mF_{2})\otimes_{\mF_{2}}A^{\otimes_{\mF_{2}}2}\to A\,,
\]
where $S(\rho)$ denotes the unit sphere in $\rho$. 
There is a $\Sigma_2\times C_2$-equivariant map $(\Sigma_2)_+\to S(\rho)_+$ given by 
inclusion of $\{-1,1\}$ into $S(\bC)=S(\rho)$; i.e. the inclusion of the $C_2$-fixed points. 
The following result appears in~\cite{Wil16}, but we include a proof for completeness. 

\begin{lem}[{\cite[Lemma~4.8]{Wil16}}]\label{lem:squaring}
Let $A$ be an $\bE_{\rho}$ $\mF_2$-algebra and let $x\in \pi_{k\rho}A$. 
Then there is a squaring operation $Q^{k\rho}$ satisfying $Q^{k\rho}(x)=x^2$. 
\end{lem}

\begin{proof}
As in~\cite[pp.~6-8]{BW17}, we can construct a power operation 
\[ 
\cP(x) :\pi_{2k\rho}^{C_2}\mF_2\otimes \left( S(\rho)_+\otimes_{\Sigma_2} S^{2k\rho}\right )\to \pi_{2k\rho}A\,.
\]
In more detail, $\cP(x)$ is obtained by applying $\pi_{2k\rho}(-)$ to the composite
\[ 
\mF_2 \otimes S(\rho)_+ \otimes_{\Sigma_2} S^{2k\rho} \xrightarrow{1 \otimes 1 \otimes x \otimes x} \mF_2 
\otimes S(\rho)_+ \otimes_{\Sigma_2} A^{\otimes 2} \to \mF_2 \otimes A \to A \,,
\]
where the second map is the $\bE_\rho$-algebra structure map and the last map is the action of $\mF_2$ on $A$. 

We then consider the inclusion $(\Sigma_2)_+\to S(\rho)_+$ and the induced diagram 
\[
\begin{tikzcd}
\pi_{2k\rho}^{C_2} \mF_2\otimes S^{2k\rho}\arrow{r}\arrow{d}{x\otimes x} & \pi_{2k\rho}^{C_2}\mF_2\otimes \left( S(\rho)_+\otimes_{\Sigma_2} S^{2k\rho}\right ) \arrow{d}\\
\pi_{2k\rho}^{C_2}A\otimes_{\mF_2} A \arrow{r} & \pi_{2k\rho}^{C_2}A
  \end{tikzcd}
\]
and observe that the composite along the top and right is our power operation $Q^{k\rho}(x)$ 
whereas the composition along the left and bottom is $x^2$ as desired.
\end{proof}

We also recall a result of Behrens--Wilson~\cite{BW17}.

\begin{lem}\label{lem:BW-power}
Given an $\bE_{\rho}$-$\underline{\bF}_2$-algebra $R$, then there are power operations 
\begin{align*} 
Q^{k\rho} : &\pi_{k\rho -\sigma}R \longrightarrow \pi_{2k\rho -\sigma}R \,, \\ 
Q^{k\rho-1} : &\pi_{k\rho -\sigma}R \longrightarrow \pi_{2k\rho -\sigma-1}R,.
\end{align*}
\end{lem}

\begin{proof}
This follows from the discussion in \cite[pp.5-6]{BW17}, even though it is not stated at this level of generality there. 
See also~\cite{Wil16}.  
\end{proof}

\subsection{Hopf algebroid formulas}\label{sec:toolkit-Hopf}
The following result will be a key tool for computing differentials in the periodic $t$-Bockstein spectral sequence. 

\begin{lem}\label{lem:right-unit}
The right unit $\eta_R^{C_2}$ for the Hopf algebroid 
\[ 
(\pi_{\star}^{C_2}\BP_{\bR}^{h_{C_{2}}S^{1}}, \pi_{\star}^{C_2}((\BP_{\bR}\otimes \BP_{\bR})^{h_{C_{2}}S^{1}})) 
\]
satisfies 
\[ 
\eta_{R}^{C_2}(\overline{t})=\sum_{i\ge 0}\chi(\bar{t}_{i})\overline{t}^{p^{i}}
\]
where $\chi(\bar{t}_{i})$ is the conjugate of $\bar{t}_{i}$ in the Hopf algebroid
\[ 
(\pi_{\star}^{C_2}\BP_{\bR}, \pi_{\star}^{C_2}(\BP_{\bR}\otimes \BP_{\bR}))\,.  
\]
\end{lem}

\begin{proof}
Since $\BPR$ and $\BPR\otimes \BPR$ are strongly even, we know that $\BPR$ and $\BPR\otimes \BPR$ are 
Real oriented and thus  
\[
(\BP_{\mathbb{R}}\otimes \BP_{\mathbb{R}})^{h_{C_2}S^1}\,, \text{ and }\BP_{\mathbb{R}}^{h_{C_2}S^1}
\]
are also strongly even. 
The right unit formula is therefore compatible with the right unit formula on the underlying in the sense that the diagram
\[
\begin{tikzcd}
\pi_{-\rho}^{C_2}\BP_{\bR}^{h_{C_{2}}S^{1}}\ar[r,"\eta_R^{C_2}"]\ar[d,"\cong"] & \pi_{-\rho}^{C_2}(\BP_{\bR}\otimes \BP_{\bR})^{h_{C_{2}}S^{1}} \ar[d,"\cong"]\\
\pi_{-2}\BP^{hS^{1}} \ar[r,"\eta_R"] & \pi_{-2}(\BP\otimes \BP)^{hS^{1}} 
\end{tikzcd}
\]
commutes, where the bottom map sends $t$ to $\sum_{i\ge 0}\chi(t_{i})t^{p^{i}}$ by~\cite{RW77} 
(cf.~\cite[Theorem~3.13,~Lemma~3.14]{Wil82} and~\cite[Remark 6.3.4]{HRW22}). Since the isomorphism sends 
\[ 
\chi(\overline{t}_i)\overline{t}^{p^i}\in \pi_{-\rho}^{C_2}\BP_{\bR}^{h_{C_{2}}S^{1}}
\] 
to $\chi(\overline{t}_i)\overline{t}^{p^i}\in \pi_{-2}\BP^{hS^{1}}$, the result follows. 
\end{proof}

\begin{lem}\label{lem:p-series}
The $p$-series of the formal group law associated to $\MUR$ satisfies 
\[ 
[p](\ot)=\ov_n\ot^{p^n}+\mathcal{O}(\ot^{p^{n}+2}) \mod (\ov_0,\cdots ,\ov_n) \,.
\]
\end{lem}
\begin{proof}
Since $\MUR$ is strongly even, $\MUR^{\mathbb{C}P_{\bR}^{\infty}}$ is strongly even and 
$\ot\in \pi_{-\rho} \MUR^{\mathbb{C}P_{\bR}^{\infty}}$. 
The computation then follows from the computation on the underlying. 
More precisely, this follows from  the commutative diagram 
\[
\begin{tikzcd}
\pi_{-\rho}^{C_2} \MUR^{\mathbb{C}P_{\bR}^{\infty}}/(\ov_0,\cdots ,\ov_n)\arrow{d}{\res}[swap]{\cong}  \arrow{r}{[p](t)} &\arrow{d}{\res}[swap]{\cong} \pi_{-\rho}^{C_2}\MUR^{\mathbb{C}P_{\bR}^{\infty}}/(\ov_0,\cdots ,\ov_n) \\ 
\pi_{-2}^{C_2} \MU^{\mathbb{C}P^{\infty}}/(p,v_1,\cdots ,v_n)  \arrow{r}{[p](t)} & \pi_{-2}^{C_2}\MU^{\mathbb{C}P^{\infty}}/(p,v_1,\cdots ,v_n) \,.
\end{tikzcd}
\]
\end{proof}

\subsection{Useful lemmas}\label{toolkit-lemmas}
 We first need a variant of~\cite[Lemma~3.14]{QS21a}. 

\begin{lem}\label{lem:limit-colimit-commuting}
Let $X\in \Sp^{h_{C_2}S^1}$. Then there are equivalences
\begin{align} 
\label{l1:lem:limit-colimit-commuting} X_{h_{C_2}S^1}\simeq  & \lim_n (\tau_{\le n}X)_{h_{C_2}S^1} \,, \\
\label{12:lem:limit-colimit-commuting} \colim_n (\tau_{\ge -n}X)^{h_{C_2}S^1}\simeq  & X^{h_{C_2}S^1} \,, \\
\label{13:lem:limit-colimit-commuting} X^{t_{C_2}S^1}\simeq & \lim_n (\tau_{\le n}X)^{t_{C_2}S^1} \,, \\ 
\label{14:lem:limit-colimit-commuting} \colim_n (\tau_{\ge -n} X)^{t_{C_2}S^1}\simeq & X^{t_{C_2}S^1} \,,
\end{align}
where $\tau_{\ge n}$ and $\tau_{\ge -n}$ are the truncation functors associated to the homotopy $t$-structure; see~Section~\ref{sec:equivariant-Mackey}. 
\end{lem}
\begin{proof}
The equivalences \eqref{l1:lem:limit-colimit-commuting} and \eqref{12:lem:limit-colimit-commuting} follow 
from~\cite[Lemma~3.14]{QS21a}. Then \eqref{13:lem:limit-colimit-commuting} and \eqref{14:lem:limit-colimit-commuting} 
follow from \eqref{l1:lem:limit-colimit-commuting} and \eqref{12:lem:limit-colimit-commuting} 
respectively using the fact that $C_2$-limits (resp. $C_2$-colimits) commute with $C_2$-limits (resp. $C_2$-colimits) 
together with the fiber sequence 
\[ \Sigma^{\sigma}X_{h_{C_2}S^1}\longrightarrow X^{h_{C_2}S^1}\longrightarrow X^{t_{C_2}S^1}
\]
from~\cite[Theorem~A]{QS21b}. 
\end{proof}

We also need an analogue of~\cite[Lemma~IV.4.12]{NS18}. 

\begin{lem}\label{MUR-lemma}
Let $M$ be a $\MUR$-module with twisted $S^1$-action. Then the map 
\[ 
M^{t_{C_2}S^1}\otimes_{\MUR^{t_{C_2}S^1}}\MUR^{t_{C_2}\mu_p}\longrightarrow M^{t_{C_2}\mu_p}
\]
is an equivalence. 
\end{lem}

\begin{proof}
By~\cite{HK01} we know $\pi_*\MUR^{t_{C_2}S^1}=(\MUR)_*((t))$ and by~\cite{LLQ22} wd know that $\pi_*\MUR^{t_{C_2}\mu_p}=\pi_*\MUR((\ot)/([p](\ot)$ so $\MUR^{t_{C_2}S^1}/([p](\ot))$ is a perfect $\MUR^{t_{C_2}S^1}$-module. Therefore, the basechange functor 
\[ 
\-- \otimes_{\MUR^{t_{C_2}S^1}}\MUR^{t_{C_2}\mu_p} 
\]
commutes with all limits and colimits. Using the filtrations $\colim (\tau_{\ge n}M)^{t_{C_2}G}$ and 
$\lim (\tau_{\le n}M)^{t_{C_2}G}$ for $G\in \{\mu_p,S^1\}$, 
where $\tau_{\ge n}$ and $\tau_{\le n}$ come from the homotopy t-structure (cf. Proposition~\ref{prop:homotopy-t-structure}), 
we are reduced to the case when $M$ is the Eilenberg--MacLane spectrum of a Mackey functor. 
In this case, the twisted $S^1$-action and consequently the twisted $\mu_p$-actions are trivial. 
The result then follows because for Eilenberg-MacLane spectra of Mackey functors that are Real oriented, we know that 
\[
HM^{t_{C_2}S^1}/([p](\ot)) \simeq HM((\ot))/([p](\ot))\simeq 
 (HM)^{t_{C_2}\mu_p}
\]
as desired. 
\end{proof}

\begin{notation}
Given a monoidal $\infty$-category $\cC$, let $\mathrm{Free}_{\mathbb{E}_1}$ denote the left adjoint to the forgetful 
functor	from $\mathbb{E}_1$-algebras in $\cC$ to $\cC$. 
\end{notation}

\section{Real syntomic cohomology}\label{sec:syntomic}
In this section, we compute the  Real syntomic cohomology of $\mF_2$, $\mZ_2$, $(\kr)_2$, and $\tmf_1(3)_2$. We begin by proving detection results in Section~\ref{sec:syntomic-detection}. We then compute the Real syntomic cohomology of $\mF_2$ in Section~\ref{sec:syntomic-F2} and use it to prove the 
Real Segal conjecture for Real truncated Brown--Peterson spectra in Section~\ref{sec:syntomic-Segal}. 
We then compute the Real syntomic cohomology of $\mZ_2$ in Section~\ref{sec:syntomic-Z}. 
Finally, we compute the Real syntomic cohomology of $(\kr)_2$ and $\tmf_1(3)_2$ in Section~\ref{sec:syntomic-kr}.
In Section~\ref{sec:syntomic-F2} we also recover a computation of Real topological cyclic homology of $\mF_2$ and in 
Section~\ref{sec:syntomic-Z} we recover a computation of mod $2$ Real topological cyclic homology of $\mZ_2$.

\begin{convention}
In this section, we implicitly work in the $2$-complete setting throughout and omit $2$-completion from our notation. We also write $\filsev$ and $\grsev$ in place of $\fil_{\sev,2}$ and $\gr_{\sev,2}$ respectively.
\end{convention}

\subsection{Detection}\label{sec:syntomic-detection}

Here we demonstrate that certain classes in the $C_2$-equivariant stable stems are detected in 
$\TCR(\mF_2)$ and $\TCR(\mZ)$. 
This is mainly used to bootstrap in the computation of Real syntomic cohomology of $\kr$. 

\begin{defin}\label{def:vi-classes}
Given a choice of classes $v_i\in \MU_{2^{i+1}-2}$ for each $i\ge 0$, we know that mod $(2,v_1,\cdots ,v_{i-1})$, 
the class $v_i$ is a cocycle in the cobar complex 
\[ 
\MU^{\otimes \bullet+1}/(2,v_1,\cdots ,v_{i-1})
\] 
and therefore produces a class 
\[ 
v_i \in \gr_{\ev}^*\bS/(2,v_1,\cdots ,v_{i-1}) \,.
\]
We define classes $\ol{v}_i\in (\MUR)_{\rho (2^i-1)}$ such that $\ol{v}_i:=\mathrm{res}^{-1}(v_i)$ (note that $\mathrm{res}^{-1}(v_i)$ is a single element since $\MUR$ is strongly even). 
Consequently, there are well-defined classes 
\[ 
\ol{v}_i\in \pi_{\rho (2^i-1)}\gr_{\sev}^*\bS/(\ov_0,\cdots ,\ol{v}_{i-1})\,.
\]
\end{defin}

\begin{notation}
We fix a $C_2$-CW complex structure on $\mathbb{C} P_{\bR}^\infty$ with $2k$-skeleton given by $\mathbb{C} P_{\bR}^k$ with 
$C_2$ acting by complex conjugation. 
Note that there is an associated multiplicative conditionally convergent 
\emph{approximate $\overline{t}$-Bockstein spectral sequence} 
\[ 
\pi_\star^{C_2}\gr_{\mot}^*\THR(R)[\bar{t}]/(\bar{t}^{k+1})\implies \pi_\star^{C_2}\gr_{\mot}^*\lim_{\mathbb{C} P_{\bR}^k}\THR(R)\,.
\]
Here we assume 
$\pi_{\rho k-1}^{C_2}\grmot^w\THR(R)=0$ for all integers $k$ and each integer $w$ in order to identify the 
$\EE_1$-page as described.  
\end{notation}

\begin{prop}\label{prop:detection-Fp}
The class $\ol{v}_0\in \pi_{0}^{C_2}\gr_{\mot}^*\mathbb{S}$ maps to 
$\ol{t}\ol{\mu}\in \pi_0^{C_2}\lim_{\mathbb{C} P_{\bR}^1}\THR(\mF_2)$.
\end{prop}

\begin{proof}
By~\cite[Addendum~5.3,~Proposition~5.4]{HM97} (cf.~\cite[Proposition IV.4.6]{NS18}), 
we know that $2\in \pi_0\bS$ maps to $t\mu\in \pi_0\lim_{\bC P^1}\THH(\bF_2)$. 
Since the motivic spectral sequence for $\lim_{\bC P^{1}}\THH(\bF_2)$ collapses and the 
motivic spectral sequence for $\bS$ collapses in degree zero, we also know that the map
\[ 
\pi_0\gr_{\ev}^*\bS\longrightarrow \pi_0\gr_{\mot}^*\lim_{{\bC}P^1}\THH(\bF_2)
\]
sends $2$ to $t\mu$. Since $\lim_{{\bC}P_{\bR}^{1}}\THR(\mF_2)$ is strongly even, $\res^{-1}(t\mu)=\ol{t}\ol{\mu}$, 
the motivic spectral sequence collapses and $\ol{v}_0=\res^{-1}(2)$ by Definition~\ref{def:vi-classes}, we know $\ol{v}_0$ maps to $\ol{t}\ol{\mu}$. 
\end{proof}

We now prove an analogous detection result for $\mZ$, $\kr$, and $\tmf_1(3)$. 

\begin{thm}\label{thm:detection}
The following statements hold
\begin{enumerate}
    \item   The unit map 
            \begin{equation*}
            \pi_*^{C_2}\gr_{\sev}^*\bS/(\ol{v}_0) \to \pi_*^{C_2}\gr_{\mot}^{\ast}\TCR^{-}(\mZ)/(\ol{v}_0)
            \end{equation*}
            sends a class detecting $\eta_{C_2}\in \pi_{\sigma}^{C_2}\bS/2$ to $\ot\olambda_1$ and a class $\ov_1$ to a class detected by $\ot\omu^2$ in the $\ot$-Bockstein spectral sequence.
    \item   The unit map 
            \begin{equation*}
            \pi_*^{C_2}\gr_{\sev}^*\bS/(\ol{v}_0,\ol{v}_1) \to \pi_*^{C_2}\gr_{\mot}^{\ast}\TCR^{-}(\kr)/(\ol{v}_0,\ol{v}_1)
            \end{equation*}
            sends the class $\bar{v}_2$ to a class detected by $\ot\bar{\mu}^4$ for in the $\ot$-Bockstein spectral sequence.
    \item   The unit map 
            \begin{equation*}
            \pi_*^{C_2}\gr_{\sev}^*\bS/(\ol{v}_0,\ol{v}_1,\ol{v}_2) \to 
            \pi_*^{C_2}\gr_{\mot}^{\ast}\TCR^{-}(\tmf_1(3))/(\ol{v}_0,\ol{v}_1,\ol{v}_2)
            \end{equation*}
            sends the class $\bar{v}_3$ to a class detected by $\ot\bar{\mu}^8$ in the $\ot$-Bockstein spectral sequence. 
\end{enumerate}
\end{thm}

\begin{proof}
By~\cite[Theorem~1.5(1)]{Rog99}, we know that $\eta\in \pi_1\bS$ maps to a nontrivial class in 
$\pi_1\lim_{\bC P^1}\THH(\bZ_{2})/2$ detected by $t\lambda_1$. 
This can be viewed as the observation that $\lambda_1=\sigma^2\eta$ using the reduced suspension map~\cite[Example~A.2.4]{HW22}. 
Note that $\eta$ is detected by $[t_1]$ in the Novikov spectral sequence and this class maps to $[t_1]$ in the 
descent spectral sequence associated to the cosimplical spectrum $\lim_{\bC P^1}\THH(\bZ_{2}/\MU_2^{\otimes_{\bS_2}\bullet+1})/2$. 
This spectral sequence collapses by~\cite[Proposition~6.1.1]{HW22} and we can determine that the image of $[t_1]$ is an 
explicit class $t\sigma^{2}[t_1]\in \lim_{\bC P^1}\THH(\bZ_{2}/\MU_2^{\otimes_{\bS_2} 2})/2$ by~\cite[Lemma~A.4.1]{HW22}. 
We now observe that both $\MUR^{\otimes 2}$ and $\lim_{\mathbb{C} P_{\bR}^1}\THR(\mZ_{2}/(\MUR)_2^{\otimes_{\bS_2} 2})/(\ol{v}_0)$ 
are strongly even so we can lift the classes $[t_1]$, $t$, and $t\sigma^2[t_1]$, respectively, 
to classes $[\ot_1]\in \pi_{\rho}^{C_2}(\MUR)_2^{\otimes_{\bS_2} 2}$, $\ol{t}$ and 
\[
\ot \sigma^{\rho}[\ot_1]\in \pi_{\rho}^{C_2}\lim_{\mathbb{C} P_{\bR}^1}\THR(\mZ_{2}/(\MUR)_2^{\otimes_{\bS_2} 2})/(\ol{v}_0),
\] 
respectively, in the sense of Definition~\ref{def:rho-suspension}, by Remark~\ref{rem:res-sigma}. 
We note that $\olambda_1$ in degree $2\rho-1$ can be chosen to be detected by 
\[
\sigma^{\rho}[\ot_1] \in \pi_{2\rho}^{C_2}\THR(\mZ_{2}/(\MUR)_2^{\otimes_{\bS_2} 2})/(\ol{v}_0)
\]
and ${{\eta}}_{C_2}$ is detected by $[\ot_1]$ to conclude the proof of the first part of the first statement using 
Lemma~\ref{HWlemma} and Remark~\ref{rem:compatibility}.   

For the second part of the first statement, we begin with the $d_1$-cycle $v_1\in \pi_{2}\lim_{\bC P^1}\MU/2$ in the 
cobar complex representing $v_1\in \pi_{2}\gr_{\mot}^*\lim_{\bC P^1}\bS/2$ on the $\EE_2$-page of the 
motivic spectral sequence; see~\cite[Corollary 2.2.21(2)]{HRW22}. By~\cite[Theorem~5.0.1]{HW22}, this maps to 
\[
t\mu^2\in \pi_2\grmot^*\lim_{\bC P^1}\THH(\bZ)/2 \,.
\] 
Again, since both $\MUR$  and 
$\THR(\mZ/\MUR)$ are strongly even, the class $v_1$ and its image $t\sigma^{2}v_1$ 
admit unique lifts $\ol{v}_1\in \pi_{\rho}^{C_2}\MUR^{h_{C_2}S^1}/(\ol{v}_0)$ and 
$\ot\sigma^{\rho}\ol{v}_1\in \pi_{\rho}\TCR^{-}(\mZ/\MUR)/(\ol{v}_0)$, respectively. 
We can choose our class $\omu^2:=\sigma^{\rho}v_1$ compatibly with the choice of class $\mu^2=\sigma^2v_1$ in 
Definition~\ref{def:rho-suspension} by Remark~\ref{rem:compatibility}, so this proves the claim. 

Similarly, for $n=1,2$ we define 
\[ 
\ol{v}_{n+1}\in \pi_{\rho (2^{n+1}-1)}^{C_2}(\lim_{\bC P_{\bR}^1}\MUR)
\]
to be the unique lift of 
\[ 
v_{n+1}\in \pi_{2p^{n+1}-2}(\lim_{\bC P^1}\MU)
\]
along 
\[
\mathrm{res}_e^{C_2} : \pi_{\rho (2^{n+1}-1)}^{C_2}(\lim_{\bC P_{\bR}^1}\MUR)
\longrightarrow \pi_{2p^{n+1}-2}(\lim_{\bC P^1}\MU)\,.
\]
This is clearly compatible with Definition~\ref{def:vi-classes}. 
Then by~\cite[Theorem~5.0.1]{HW22}, we know that the class $\ol{v}_2$-maps to $\ol{t}\ol{\mu}^{4}$ and $\ol{v}_3$ 
maps to $\ol{t}\ol{\mu}^8$. 
\end{proof}

\subsection{Real syntomic cohomology of $\bF_2$}\label{sec:syntomic-F2}
The starting point is the computation 
\[ 
\pi_{\star}^{C_2}\gr_{\mot}^*\THR(\mF_2)/(\ov_0)\cong \pi_{\star}^{C_2}\mF_2[\omu]\otimes \Lambda (\bar{\varepsilon}_0)
\]
from Corollary~\ref{cor:vn-thr}. 

\begin{prop}\label{prop:tpfp}
\
\begin{enumerate}
    \item   The periodic $\ot$-Bockstein spectral sequence 
            \[ 
            \pi_{\star}^{C_2}\gr_{\mot}^*\THR(\mF_2)/(\ov_0)[\ot,{\ol{t}}^{-1}]\implies \pi_{\star}^{C_2}\gr_{\mot}^*\TPR(\mF_2)/(\ov_0)
            \]
            has a differential 
            \[ 
            d_1(\bar{\varepsilon}_0)=t\omu
            \]
            and no further differentials besides those generated by the Leibniz rule. The spectral sequence degenerates 
            at the $\mathrm{E}_2$-page with $\mathrm{E}_{2}=\mathrm{E}_{\infty}$-page 
            \[ 
            \pi_{\star}^{C_2}\mF_2[\ot,\ol{t}^{-1}]
            \]
            and therefore has no room for extensions. 

    \item   The $\ot$-Bockstein spectral sequence 
            \[ 
            \pi_{\star}^{C_2}\gr_{\mot}^*\THR(\mF_2)/(\ov_0)[\ot]\implies \pi_{\star}^{C_2}\gr_{\mot}^*\TCR^{-}(\mF_2)/(\ov_0)
            \]
            has a $d_1$-differential $d_1(\bar{\varepsilon}_0)=\ot\omu$ and collapses at the $\mathrm{E}_2$-page with 
            abutment the $\pi_{\star}^{C_2}\mF_2$-algebra
            \[
            \pi_{\star}^{C_2}\mF_2[\ot,\omu]/(\ol{t}\omu) 
            \]
            which we can describe additively as  
            \[
            \pi_{\star}^{C_2}\mF_2[\ot]\oplus T 
            \]
            where $T=\pi_{\star}^{C_2}\mF_2[\omu]\{\omu\}$ consists of simple $\ot$-torsion elements. 
    \item   The canonical map is given by the canonical inclusion 
            \[
            \pi_{\star}^{C_2}\mF_2[\ot]\subset\pi_{\star}^{C_2}\mF_2[\ot,\ot^{-1}]
            \]
            direct sum with the zero map $T\to 0$.
\end{enumerate}
\end{prop}

\begin{proof}
Note that since $\THR(\mF_2)$ is concentrated in Adams weight zero and $\ol{t}$ is in Adams weight zero,
 we know that $\TCR^{-}(\mF_2)$ is concentrated in Adams weight $0$ and $\TCR(\mF_2)$ is concentrated in Adams 
weights~$[0,1]$. This implies that the motivic spectral sequence collapses for each of 
$\THR(\mF_2)$, $\TCR^{-}(\mF_2)$, and $\TCR(\mF_2)$. 

To determine that there aren't differentials on classes in $\pi_0^{C_2}\THR(\mF_2)/(\ov_0)$ in the 
(periodic) $\ol{t}$-Bockstein spectral sequence, we note that by Lemma~\ref{lem:pi0tcr}, the map 
\[ 
P_0^0\KR(\mF_2)\to P_0^0\TCR(\mF_2)
\]
induced by the trace map is an equivalence and since for $\TCR(\mF_2)$, $\TCR^{-}(\mF_2)$, 
and $\THR(\mF_2)$ the motivic spectral sequence collapses, 
we know that the classes  
\[
x\in \pi_{\star}^{C_2}\mF_2\subset  \pi_{\star}^{C_2}\gr_{\mot}^*\THR(\mF_2)/(\ov_0)[t]\subset 
\pi_{\star}^{C_2}\gr_{\mot}^*\THR(\mF_2)/(\ov_0)[t,t^{-1}]
\]
are permanent cycles. 

We determine the differential $d_1(\bar{\varepsilon}_0)=\ot\omu$ using the fact that $\ov_0$ is detected by $\ot\omu$, 
proven in Proposition~\ref{prop:detection-Fp} and the definition of $\bar{\varepsilon}_0$ from 
Definition~\ref{def:varepsilon}. Explicitly, this follows from applying the snake lemma to the commutative diagram
\[ 
\begin{tikzcd}
\ker \arrow[d] \arrow[r] &\ker \arrow[r]  \arrow[d]& \ker  \arrow[d,"i_{0}"]  \\
\pi_{\star}^{C_2}\Sigma^{-1,-1}\lim_{\bC P_{\bR}^1}\THR(\mF_2)/(\ov_0)
\arrow{r} \arrow[d] &  \pi_{\star}^{C_2}\Sigma^{-1,-1}\THR(\mF_2)/(\ov_0) \arrow[twoheadrightarrow]{r} \arrow[d,"j_{0}"] & \im (d_1)\arrow[d]    \\
 \ker (d_1)\arrow[hookrightarrow]{r} \arrow[d,"\ov_0"] &  \pi_{\star}^{C_2}\THR(\mF_2) \arrow[r]  \arrow[d] & \pi_{\star}^{C_2}\Sigma^{-\rho+1,1}\THR(\mF_2) \arrow[d]  \\ 
 \mathrm{coker} \arrow[r] &  \mathrm{coker}\arrow[r] & \mathrm{coker} 
\end{tikzcd}
\]
once we note that $j_{0}(\ovarepsilon_0)=1$ and $\ov_0\cdot 1=\ol{t}\omu$.

After running the $d_1$-differential, the considerations at the beginning of the proof imply that the 
spectral sequences collapse at the $\mathrm{E}_2$-page. 
\end{proof}

\begin{figure}
\resizebox{.5\textwidth}{!}{
\begin{tikzpicture}[radius=.08,scale=1]
\foreach \n in {-5,-4,-3,-2,-1,0,...,5} \node [below] at (\n+.1,-.1) {$\n$};
\foreach \s in {-5,-4,-3,-2,-1,1,2,...,5} \node [left] at (-.1,\s) {$\s$};
\draw [thin,color=lightgray] (-5,-5) grid (5,5);
\clip (-5.1,-5.1) -- (5.1,-5.1) -- (5.1,5.1) -- (-5.1,5.1) -- cycle;
\node [right] at (0,0) {1};
\node [right] at (0,-1) {$a_{\sigma}$};
\node [right] at (1,-1) {$u_{\sigma}$};
\node [right] at (-2,2) {$\theta$};


\foreach \s in {-5,-4,-3,-2,-1,0,1,...,5}
    \foreach \m in {-5,-4,-3,-2,-1,0}
        \foreach \n in {-5,...,\m} 
            {\draw [fill] (-\m+\s,\n+\s) circle;
            }

\foreach \s in {-5,-4,-3,-2,-1,0,1,...,5}
    \foreach \n in {0,1,2,3,4,5}
        {
        \draw (\n+\s,-\n+\s) -- (\n+\s,-5+\s);
        }
\foreach \s in {-5,-4,-3,-2,-1,0,1,...,5}
    \foreach \n in {0,1,2,3,4,5}
        {
        \draw (\s,-\n+\s) -- (5-\n+\s,-5+\s);
        }


\foreach \s in {-3,-2,-1,0,1,2,3}
    \foreach \n in {2,3,4,5} 
        {\draw [fill] (-2+\s,\n+\s) circle;}
\foreach \s in {-3,-2,-1,0,1,2,3}
    \foreach \n in {3,4,5} 
        {\draw [fill] (-3+\s,\n+\s) circle;}
\foreach \s in {-3,-2,-1,0,1,2,3}
    \foreach \n in {4,5} 
        {\draw [fill] (-4+\s,\n+\s) circle;}
\foreach \s in {-3,-2,-1,0,1,2,3}
    \foreach \n in {5} 
        {\draw [fill] (-5+\s,\n+\s) circle;}

\foreach \s in {-3,-2,-1,0,1,2,3}
    \foreach \n in {2,...,5}
        {\draw (-2+\s,\n+\s) -- (-7+\n+\s,5+\s);}
\foreach \s in {-3,-2,-1,0,1,2,3}
    \foreach \n in {2,3,4} 
        {\draw (-\n+\s,\n+\s) -- (-\n+\s,5+\s);}

\node [right] at (-1,-1) {$\ol{t}$};
\foreach \s in {-4,-3,-2,-1,2,3,4}
    \node [right] at (-\s,-\s) {$\ol{t}^{\s}$};

\end{tikzpicture}
}
\caption{
The $\mathrm{E}_2=\mathrm{E}_{\infty}$-page of the periodic $\ol{t}$-Bockstein spectral sequence 
computing $\pi_{\star}^{C_2}\grmot^*\TPR(\mF_2)$. The group $E_2^{m+n\sigma,a,b}$ appears in 
bidegree~$(m,n)$. Each bullet $\bullet$ represents a copy of $\mathbb{F}_2$. 
\label{fig2}
}
\end{figure}

\begin{prop}
The Frobenius map 
\[ 
\varphi : \pi_{\star}^{C_2}\gr_{\mot}^*\TCR^{-}(\mF_2)/(\ov_0) \to \pi_{\star}^{C_2}\gr_{\mot}^*\TPR(\mF_2)/(\ov_0)
\]
is given by $\varphi (x)=x$ if $x\in \pi_{\star}^{C_2}\mF_2$, $\varphi(y)=0$ 
if $y\in (\ot)\subset \pi_{\star}^{C_2}\mF_2[\ot]$, and $\varphi |_{T}$ is a monomorphism, where $T=\pi_\star^{C_2}\underline{\mathbb F}_2[\omu]\{\omu\}$.
\end{prop}
\begin{proof}
The strategy is similar to~\cite[Proposition IV.4.9]{NS18}. We consider the diagram 
\[
\begin{tikzcd}
\pi_{-\rho}^{C_2}\TCR^{-}(\mF_2)/(\ov_0)\ar[r] \ar[d] & \pi_{-\rho}^{C_2}\TPR(\mF_2)/(\ov_0) \ar[r] \ar[d] & \pi_{-\rho}^{C_2}\mF_2^{t_{C_2}S^1}/(\ov_0) \ar[d] \\
\pi_{-\rho}^{C_2}\THR(\mF_2)/(\ov_0) \ar[r] & \pi_{-\rho}^{C_2}\THR(\mF_2)^{t_{C_2}\mu_2}/(\ov_0)\ar[r] & \pi_{-\rho}^{C_2} \mF_2^{t_{C_2}\mu_2}/(\ov_0)
\end{tikzcd}
\]
and note that $\pi_{-\rho}^{C_2}\THR(\mF_2)/(\ov_0)=0$ and the composite map 
\[
\mathbb{F}_2=\pi_{-\rho}^{C_2}\TPR(\mF_2)/(\ov_0)  \to \pi_{-\rho}^{C_2}\mF_2^{t_{C_2}S^1}/(\ov_0) \to  \pi_{-\rho}^{C_2} \mF_2^{t_{C_2}\mu_2}/(\ov_0)=\bF_2 
\]
is nontrivial. 

Consequently, the class $\ol{t}\in \pi_{-\rho}^{C_2}\TCR^{-}(\mF_2)/(\ov_0)$ 
maps to zero in $\pi_{-\rho}\TPR(\mF_2)/(\ov_0)$.  
Since $\upi_{\star}\TCR^{-}(\mF_2)/(\ov_0)=\upi_{\star}\mF_2[\omu,\ot]/(\ot\omu)$ as a $\upi_{\star}\mF_2$-algebra, 
we note that
\[ 
\varphi : \pi_{\star}^{C_2}\gr_{\mot}^*\TCR^{-}(\mF_2)/(\ov_0)\to \pi_{\star}^{C_2}\gr_{\mot}^*\TPR(\mF_2)/(\ov_0)
\]
is zero on $y\in (\ol{t})$. 

Applying a similar argument using the diagram
\[
\begin{tikzcd}
\pi_{\rho}^{C_2}\TCR^{-}(\mF_2)/(\ov_0)\ar[r] \ar[d] & \pi_{\rho}^{C_2}\TPR(\mF_2)/(\ov_0) \ar[r] \ar[d] & \pi_{\rho}^{C_2}\mF_2^{t_{C_2}S^1}/(\ov_0) \ar[d] \\
\pi_{\rho}^{C_2}\THR(\mF_2)/(\ov_0) \ar[r] & \pi_{\rho}^{C_2}\THR(\mF_2)^{t_{C_2}\mu_2}/(\ov_0)\ar[r] & \pi_{\rho}^{C_2} \mF_2^{t_{C_2}\mu_2}/(\ov_0)
\end{tikzcd}
\]
where instead we know that 
\[
\pi_{\rho}^{C_2}\TCR^{-}(\mF_2)/(\ov_0)\longrightarrow \pi_{\rho}^{C_2}\THR(\mF_2)/(\ov_0)
\]
sends $\omu$ to $\omu$ and that the map 
\[
\mathbb{F}_2=\pi_{\rho}^{C_2}\TPR(\mF_2)/(\ov_0)  \to \pi_{\rho}^{C_2}\mF_2^{t_{C_2}S^1}/(\ov_0) 
\to  \pi_{\rho}^{C_2} \mF_2^{t_{C_2}\mu_2}/(\ov_0)=\bF_2 
\]
is nontrivial, we also determine from the ring structure that $\varphi |_{T}$ is a monomorphism.

 Since $x\in \upi_{\star}\mF_2$ is in the kernel of $\can-\varphi$, we know that $\varphi(x)=x$ for $x\in \upi_{\star}\mF_2$.
\end{proof}

\begin{cor}
The mod $\ov_0$ syntomic cohomology of $\mathbb{F}_2$ can be identified as
\[ 
\pi_{\star}^{C_2}\gr_{\mot}^*\TCR(\mF_2)/
(\ov_0)\simeq \pi_{\star}^{C_2}\mF_2\langle \partial \rangle
\]
where $|\partial|=(-1,1)$.
\end{cor}

\begin{proof}
We computed $\varphi$ and $\mathrm{can}$, so it suffices to observe that these are the only $x\in \pi_{\star}^{C_2}\mF_2$ 
which are in the kernel or nonzero in the cokernel of $\varphi-\mathrm{can}$. 
\end{proof}

We therefore recover a special case of a result of Dotto--Moi--Patchkoria~\cite[Corollary~4.12]{DMP24} and 
Quigley--Shah~\cite[Theorem~6.6]{QS21a}. 
\begin{thm}

There is an equivalence of $C_2$-spectra
\[ 
\TCR(\mF_2)\simeq  H\mZ\oplus \Sigma^{-1}H\mZ \,.
\]
\end{thm}

\begin{proof}
The motivic spectral sequence
\[ 
\pi_\star^{C_2}\gr_{\mot}^*\TCR(\mF_2)/(\ov_0)\implies \pi_{\star}^{C_2}\TCR(\mF_2)/(2)
\]
collapses for bidegree reasons since we know that the classes $x\in \upi_{\star}\mF_2$ are permanent cycles; 
see Figure~\ref{fig:motivic-F2}. Consequently, we can also conclude that $P^0\TCR(\mF_2)\simeq \TCR(\mF_2)$. 

Observe that $\upi_0 \KR(\mF_2) = \mZ$ and $\KR(\mF_2)$ is slice bounded below, 
so $P^0 \KR(\mF_2) \simeq \mZ$. We then consider the map
\[ 
\mZ\simeq P^0\KR(\mF_2) \to P^0\TCR(\mF_2)\simeq \TCR(\mF_2)
\]
and the associated map of Bockstein spectral sequences
\[ 
\begin{tikzcd}
\upi_{\star}\mF_2[\bar{v}_0] \ar[r] \arrow[Rightarrow]{d} & \upi_{\star}\mF_2\otimes \Lambda (\partial) [\ov_0] \arrow[Rightarrow]{d} \\
\upi_{\star}H\mZ \ar[r] & \upi_{\star}\TCR(\mF_2)
\end{tikzcd}
\]
produces the result since $\partial$ is a permanent cycle for bidegree and motivic filtration reasons and 
$\TCR(\mF_2)$ is a $\mZ$-module. 
\end{proof}

\begin{figure}
\resizebox{.3\textwidth}{!}{
\begin{tikzpicture}[radius=.08,scale=1]
\foreach \n in {-1,0,1} \node [below] at (\n+.1,-.1) {$\n$};
\node [below] at (2+.1,-.1) {$\sigma$};
\node [below] at (-2+.1,-.1) {$-\sigma$};
\foreach \s in {-1,1,2} \node [left] at (-.1,\s) {$\s$};

\draw [thin,color=lightgray] (-2,-1) grid (2,2);

\clip (-2,-1) -- (-2,2) -- (2,2) -- (2,-1) -- cycle;

\draw [fill] (0,0) circle;
\node [above] at (0.1,.1) {1};
\draw [fill] (-1,1) circle;
\node [above] at (-.9,1.1) {$\partial$};
\draw (0,0) -- (-1,1);
\end{tikzpicture}
}
\caption{The motivic spectral sequence computing $\pi_{\star}^{C_2}\TCR(\mF_2)/2$. 
Each bullet $\bullet$ represents a copy of $\pi_{\star}^{C_2}\mF_2$. 
The horizontal axis represents the stem (a $C_2$-representation) and the vertical axis represents the Adams weight.
\label{fig:motivic-F2}
}
\end{figure}

The following corollary is immediate since the functor $(\--)^{\textup{triv}}$ is left adjoint to $\mathrm{TCR}$ by definition; 
see~\cite[Definition~2.30]{QS21a}.
\begin{cor}
There is a map of $C_2$-$\bE_\infty$-rings in Real $2$-cyclotomic spectra 
\[ 
\mZ^{\textup{triv}} \longrightarrow \mathrm{THR}(\mF_2)
\]
where $\mZ^{\textup{triv}}$ denotes $\mZ$ equipped with trivial Real cyclotomic structure.
\end{cor}

\subsection{The Real Segal conjecture}\label{sec:syntomic-Segal}
 The goal of this section is to prove an analogue of the 
 Segal conjecture for Real topological Hochschild homology of Real truncated Brown--Peterson spectra, 
 i.e. the cyclotomic Frobenius map 
\[ 
\THR(\BPR\langle n\rangle )\otimes_{\BPR\langle n\rangle}H\mF_2 \longrightarrow  
\THR(\BPR\langle n\rangle )^{t_{C_2}\mu_2} \otimes_{\BPR\langle n\rangle}H\mF_2 
\]
has bounded above fiber. 

\begin{prop}
There is an equivalence 
\[ 
\TPR(\mF_2)/(2)\simeq \THR(\mF_2)^{t_{C_2}\mu_2}\,.
\]
The cyclotomic Frobenius
\[ 
\upi_{\star}\THR(\mF_2) \longrightarrow \upi_{\star}\THR(\mF_2)^{t_{C_2}\mu_2}
\]
is given by inverting $\bar{\mu}$.
\end{prop}

\begin{proof}
The first claim is a direct consequence of Lemma~\ref{MUR-lemma} 
in the case where $X=\THR(\mF_2)$. To see this, note that $H\mZ$ is even so it is Real oriented and we have  
\[\mZ^{t_{C_2}\mu_2}=\mZ((t))/([2](t))\simeq \mZ^{t_{C_2}S^1}/(2)\,,
\]
see~\cite[Theorem~2.10]{HK01}. This also implies that $\THR(\mF_2)^{t_{C_2}\mu_2}$ is strongly even,
since $\TPR(\mF_2)/2$ is strongly even by Proposition~\ref{prop:tpfp}. 
Thus, it suffices to check the second statement on underlying, where it is known by~\cite{HM97}. 
\end{proof}

\begin{thm}\label{thm:Segal}
For $-1\le n\le 2$, the map 
\[ 
\pi_{\star}^{C_2}\grmot^*\THR(\BPR\langle n\rangle )/(\ov_0,\cdots ,\ov_n) 
\longrightarrow \pi_{\star}^{C_2}\grmot^*\THR(\BPR\langle n\rangle )^{t_{C_2}\mu_2}/(\ov_0,\cdots ,\ov_n)
\]
is given by inverting $\overline{\mu}^{2^{n+1}}$\,.
\end{thm}

\begin{proof}
Since $\THR(\BPR\langle n\rangle)\to \THR(\BPR\langle n\rangle/\MWR)$ is seff 
and $\THR(\BPR\langle n\rangle/\MWR)$ is strongly even, we can consider the map of $\EE_1$-pages of motivic spectral sequences 
\[ 
\pi_{\star}^{C_2}\THR(\BPR\langle n\rangle/\MWR^{\otimes \bullet+1})
\longrightarrow \pi_{\star}^{C_2}\THR(\BPR\langle n\rangle/\MWR^{\otimes \bullet+1})^{t_{C_2}\mu_2}\,.
\]
Each $C_2$-spectrum 
\[ 
\THR(\BPR\langle n\rangle/\MWR^{\otimes q+1})
\]
is strongly even and consequently each $C_2$-spectrum
\[
\THR(\BPR\langle n\rangle/\MWR^{\otimes q+1})^{t_{C_2}\mu_2}
\]
is strongly even for $q\ge 0$.
Moreover, the graded rings 
\[
\pi_{*}\THH(\BP\langle n\rangle/\MW^{\otimes q+1})
\]
and 
\[
\pi_{*}\THH(\BP\langle n\rangle/\MW^{\otimes q+1})^{t\mu_2}
\]
are polynomial and polynomial tensor Laurent polynomial, respectively, on classes in even degrees. 
We can therefore apply Lemma~\ref{stronglyevenlemma}. Specifically, on underlying homotopy groups, the map 
\[
\pi_{*}(\THH(\BP\langle n\rangle/\MW^{\otimes q+1})\otimes_{\BP\langle n\rangle}\bF_2)
\to \pi_{*}(\THH(\BP\langle n\rangle/\MW^{\otimes q+1})^{t\mu_2}\otimes_{\BP\langle n\rangle}\bF_2)
\]
is given by 
\[
\bF_2[\mu^{2^{n+1}}]\otimes P\longrightarrow \bF_2[\mu^{\pm 2^{n+1}}]\otimes P
\]
when $q=0$, and by
\[
\bF_2[\mu^{2^{n+1}},t_1,\cdots ,t_{n+1}]\otimes P\otimes P 
\longrightarrow \bF_2[\mu^{\pm 2^{n+1}},t_1,\cdots ,t_{n+1}]\otimes P\otimes P
\]
when $q=1$, and for $q\ge 2$ it is determined by degeneracies. 
This follows by the same argument as~\cite[Proposition~6.1.6]{HW22}. We conclude that the map 
\[ 
\pi_{\star}^{C_2}\THR(\BPR\langle n\rangle/\MWR^{\otimes q+1})\otimes_{\BPR\langle n\rangle } \mF_2
\longrightarrow \pi_{\star}^{C_2}\THR(\BPR\langle n\rangle/\MWR^{\otimes q+1})^{t_{C_2}\mu_2}\otimes_{\BPR\langle n\rangle } 
\mF_2
\]
is given by 
\[
\pi_\star^{C_2} \mF_2[\omu^{2^{n+1}}]\otimes \ol{P}\longrightarrow \pi_\star^{C_2} \mF_2[\omu^{\pm 2^{n+1}}]\otimes \ol{P}
\]
when $q=0$, and by  
\[
\pi_\star^{C_2} \mF_2[\omu^{2^{n+1}},\ot_1,\cdots, \ot_{n+1}]\otimes \ol{P}\otimes \ol{P} 
\longrightarrow \pi_\star^{C_2} \mF_2[\omu^{\pm 2^{n+1}},\ot_1,\cdots ,\ot_{n+1}]\otimes \ol{P}\otimes \ol{P}
\]
when $q=1$, and by degeneracies for $q\ge 2$ where $\ol{P}$ is a polynomial algebra on generators 
in degrees divisible by $\rho$. We conclude as in Theorem~\ref{thm:thr} that the map 
\[ 
\pi_{\star}^{C_2}\grmot^*\THR(\BPR\langle n\rangle)/(\ov_0,\cdots, \ov_n)\longrightarrow 
\pi_{\star}^{C_2}\grmot^*\THR(\BPR\langle n\rangle)^{t_{C_2}\mu_2}/(\ov_0,\cdots ,\ov_n)
\]
is given by 
\[ 
\pi_\star^{C_2} \mF_2[\omu^{2^{n+1}} ]\langle \olambda_1,\cdots , \olambda_{n+1}\rangle 
\longrightarrow \pi_\star^{C_2}\mF_2[\omu^{\pm 2^{n+1}} ]\langle \olambda_1,\cdots ,\olambda_{n+1}\rangle
\]
as desired. 
\end{proof}

\subsection{The useful isomorphism}\label{sec:syntomic-useful}
We now prove an equivariant analogue of the useful isomorphism from~\cite[\S~6.4]{HRW22}. 
When $R=\BPR\langle n\rangle$ with $-1 \leq n \leq 2$, 
note that there are $C_2$-$\bE_\infty$-algebra maps 
\[
\begin{tikzcd}
   \filmot^*\BPR\langle n\rangle /(\ov_0,\cdots ,\ov_n) \ar[r] &  \filmot^*\THR(\BPR\langle n\rangle )/(\ov_0,\cdots ,\ov_n) \ar[d] \\  
   & \filmot^*\THR(\BPR\langle n\rangle )^{t_{C_2}\mu_p} /(\ov_0,\cdots ,\ov_n) 
\end{tikzcd}
\]
and since $\ol{v}_{n+1}$ acts trivially on $\filmot^*\BPR\langle n\rangle /(\ov_0,\cdots ,\ov_n),$
the natural map 
\[ 
\grmot^*\TPR(\BPR\langle n\rangle )/(\ov_0,\cdots ,\ov_n)\longrightarrow 
 \grmot^*\THR(\BPR\langle n\rangle )^{t_{C_2}\mu_p}/(\ov_0,\cdots ,\ov_n)
\]
factors over a map 
\[ 
g : \grmot^*\TPR(\BPR\langle n\rangle )/(\ov_0,\cdots ,\ov_n,\ov_{n+1})\longrightarrow 
 \grmot^*\THR(\BPR\langle n\rangle )^{t_{C_2}\mu_p}/(\ov_0,\cdots ,\ov_n) \,.
\]
\begin{proposition}\label{prop:useful-iso}
Let $0\le n\le 2$. The map 
\[  
g: \grmot^*\TPR(\BPR\langle n\rangle )/(\ov_0,\cdots ,\ov_n,\ov_{n+1})\longrightarrow  
\grmot^*\THR(\BPR\langle n\rangle )^{t_{C_2}\mu_p}/(\ov_0,\cdots ,\ov_n) \,.
\]
is an equivalence. 
\end{proposition}
\begin{proof}
Consider the cobar complex with $q$-cosimplices
\[ 
\pi_*\TPR(\BPR\langle n\rangle /\MWR^{\otimes q+1})/(\ol{v}_0,\cdots ,\ol{v}_{n+1})
\]
and since there is a homotopy commutative ring map from $\MUR$ into this it suffices to consider the case $q=0$. 
It then suffices to note the following:
\begin{enumerate}
   \item    There is a map of homotopy commutative $\MUR$-algebras 
            \[ 
            \BPR\langle n\rangle  \longrightarrow \THR(\BPR\langle n\rangle/\MWR)\longrightarrow 
            \THR(\BPR\langle n\rangle/\MWR)^{t_{C_2}\mu_p}
            \]
            and therefore modulo $(\ov_0,\cdots ,\ov_n)$ the element $\ov_{n+1}$ is well-defined and acts trivially on 
            \[
            \THR(\BPR\langle n\rangle/\MWR)^{t_{C_2}\mu_2}/(\ov_0,\cdots ,\ov_n) \,.
            \]
            Applying Lemma~\ref{MUR-lemma}, this implies that $\ov_{n+1}=[p](\ot)f$
            for some 
            \[
            f\in \pi_{\star}^{C_2}\TPR(\BPR\langle n\rangle/\MWR)/(\ov_0,\cdots ,\ov_n)\,.
            \] 
   \item    By Lemma~\ref{lem:p-series}, we can write 
            \[ 
            [2](\ot)=\ov_{n+1}t^{2^n}+\mathcal{O}(\ot^{2^n+2})
            \]
            in $\MUR^{h_{C_2}S^1}/(\ov_0,\cdots ,\ov_n)$ and $\ov_{n+1}=\ot \sigma^{\rho}\ov_{n+1}$ and therefore 
            \[ 
            \ot\sigma^{\rho}\ov_{n+1}=(\ot^{2^{n+1}+1}\sigma^{\rho}\ov_{n+1}+\mathcal{O}(\ot^{2^{n+1}+2}))f
            \]
            so since $\sigma^{\rho}\ov_{n+1}$ is not a zero divisor, we note that the element $f$ is detected by $\ot^{-2^{n+1}}$ in the parametrized Tate spectral sequence, which is a unit. 
\end{enumerate}
\end{proof}

\subsection{Real syntomic cohomology of the $2$-adic integers}\label{sec:syntomic-Z}

The starting point is the computation 
\[ 
\pi_{\star}^{C_2}\grmot^*\THR(\bZ)/(\ov_0,\ov_1)= 
\pi_{\star}^{C_2}\mF_2[\bar{\mu}^2]\langle \olambda_1 , \ol{\varepsilon}_1 \rangle 
\]
from Corollary~\ref{cor:vn-thr}. 

\begin{prop}\label{prop:TPRZ}
In the spectral sequence 
\[
\pi_{\star}^{C_2}\mF_2[\omu^2,\ot,\ot^{-1}]\langle \olambda_1 ,\ovarepsilon_1 \rangle 
\implies 
\pi_{\star}^{C_2}\gr_{\mot} ^*\TPR(\mZ_2)/(\ov_0,\ov_1) \,,
\]
there is a $d_1$-differential
\[ 
d_1(\bar{\varepsilon}_1)=\bar{t}\bar{\mu}^2\,,
\]
which, together with those differentials generated by the Leibniz rule, 
leaves $\EE_2=\pi_{\star}^{C_2}\mF_2[\ot,\ot^{-1}]\langle \olambda_1 \rangle$\,. 
There is a $d_2$-differential 
\[ 
d_2(\bar{t}^{-1})=\bar{t}\bar{\lambda}_1\,,
\]
which, together with those differentials generated by the Leibniz rule, leaves 
\[
\EE_3=\pi_{\star}^{C_2}\mF_2[\ot^2,\ot^{-2}]\langle \olambda_1 \rangle\,.
\]
The spectral sequence collapses at the $\EE_3$-page. 
Moreover, there is an isomorphism of $\pi_{\star}^{C_2}\mF_2\langle \olambda_1\rangle$-modules 
\[ 
\mathrm{E}_{\infty}=\pi_{\star}^{C_2}\mF_2[\ot^2]\langle \olambda_1\rangle 
\oplus \pi_{\star}^{C_2}\mF_2\langle \olambda_1\rangle [t^{-1}]\{t^{-1}\}
\]
\end{prop}

\begin{proof}
By the same argument as in the proof of Proposition~\ref{prop:tpfp}, there is a differential $d_1(\ovarepsilon_1)=\ot\omu^2$ using Theorem~\ref{thm:detection}~(1). Similarly, 
there is a differential $d_2(\ot^{-1})=\ot\olambda_1$. Alternatively, we can note that $\olambda_1=\sigma^\rho\ot_1$ 
and there is a differential in the motivic cobar complex for the sphere 
\[ 
\partial(\ot^{-1})=\eta_R(\ot^{-1})-\eta_L(t^{-1})=-\ot^2\ot_1 + \mathcal{O}(\ot^{3}) \mod(\ov_0,\ov_1) 
\]
which agrees with the total differential on $\ot^{-1}$ in the case of $\TPR(\underline{\bZ})$, which also implies $d_2(\ot^{-1})=\ot\olambda_1$. 

From this, we can also conclude from the multiplicative structure that the spectral sequence collapses 
at the $\EE_3$-page, cf.~Figure~\ref{fig3}. 
Alternatively, we can observe that we know the abutment by Proposition~\ref{prop:useful-iso} and 
Theorem~\ref{thm:Segal}, and any further differentials would lead to a contradiction.
\end{proof}

\begin{figure}
\resizebox{.6\textwidth}{!}{
\begin{tikzpicture}[radius=.08,scale=.05cm]
\foreach \n in {-5,-4,-3,-2,-1,0,...,5} \node [below] at (\n+.1,-.1) {$\n$};
\foreach \s in {-5,-4,-3,-2,-1,1,2,...,5} \node [left] at (-.1,\s) {$\s$};
\draw [thin,color=lightgray] (-5,-5) grid (5,5);
\clip (-5.1,-5.1) -- (5.1,-5.1) -- (5.1,5.1) -- (-5.1,5.1) -- cycle;
\node [right] at (0,0) {1};
\node [right] at (0,-1) {$a_{\sigma}$};
\node [right] at (1,-1) {$u_{\sigma}$};
\node [right] at (-2,2) {$\theta$};


\foreach \s in {-4,-2,0,2,4}
    \foreach \m in {-5,-4,-3,-2,-1,0}
        \foreach \n in {-5,...,\m} 
            {\draw [fill] (-\m+\s,\n+\s) circle;
            }

\foreach \s in {-4,-2,0,2,4}
    \foreach \n in {0,1,2,3,4,5}
        {
        \draw (\n+\s,-\n+\s) -- (\n+\s,-5+\s);
        }
\foreach \s in {-4,-2,0,2,4}
    \foreach \n in {0,1,2,3,4,5}
        {
        \draw (\s,-\n+\s) -- (5-\n+\s,-5+\s);
        }


\foreach \s in {-2,0,2}
    \foreach \n in {2,3,4,5} 
        {\draw [fill] (-2+\s,\n+\s) circle;}
\foreach \s in {-2,1,0,1,2}
    \foreach \n in {3,4,5,6} 
        {\draw [fill] (-3+\s,\n+\s) circle;}
\foreach \s in {-2,1,0,1,2}
    \foreach \n in {4,5} 
        {\draw [fill] (-4+\s,\n+\s) circle;}
\foreach \s in {-2,1,0,1,2}
    \foreach \n in {5} 
        {\draw [fill] (-5+\s,\n+\s) circle;}

\foreach \s in {-2,0,2}
    \foreach \n in {2,...,5}
        {\draw (-2+\s,\n+\s) -- (-8+\n+\s,6+\s);}
\foreach \s in {-2,0,2}
    \foreach \n in {2,3,4} 
        {\draw (-\n+\s,\n+\s) -- (-\n+\s,7+\s);}

\foreach \s in {-4,-2,2,4}
    \node [right] at (-\s,-\s) {$\overline{t}^{\s}$};

\node [right] at (1,2) {$\overline{\lambda}_1$};

\begin{scope}[blue,,dashed,xshift=-1.15cm,yshift=.15cm]

\foreach \s in {-4,-2,0,2,4}
    \foreach \m in {-5,-4,-3,-2,-1,0}
        \foreach \n in {-5,...,\m} 
            {\draw [fill] (-\m+\s,\n+\s) circle;
            }

\foreach \s in {-4,-2,0,2,4}
    \foreach \n in {0,1,2,3,4,5}
        {
        \draw (\n+\s,-\n+\s) -- (\n+\s,-5+\s);
        }
\foreach \s in {-4,-2,0,2,4}
    \foreach \n in {0,1,2,3,4,5}
        {
        \draw (\s,-\n+\s) -- (5-\n+\s,-5+\s);
        }


\foreach \s in {-2,0,2}
    \foreach \n in {2,3,4,5} 
        {\draw [fill] (-2+\s,\n+\s) circle;}
\foreach \s in {-2,1,0,1,2}
    \foreach \n in {3,4,5,6} 
        {\draw [fill] (-3+\s,\n+\s) circle;}
\foreach \s in {-2,1,0,1,2}
    \foreach \n in {4,5} 
        {\draw [fill] (-4+\s,\n+\s) circle;}
\foreach \s in {-2,1,0,1,2}
    \foreach \n in {5} 
        {\draw [fill] (-5+\s,\n+\s) circle;}

\foreach \s in {-2,0,2}
    \foreach \n in {2,...,5}
        {\draw (-2+\s,\n+\s) -- (-8+\n+\s,6+\s);}
\foreach \s in {-2,0,2}
    \foreach \n in {2,3,4} 
        {\draw (-\n+\s,\n+\s) -- (-\n+\s,7+\s);}
\end{scope}
\end{tikzpicture}
}
\caption{
The $\mathrm{E}_3=\mathrm{E}_{\infty}$-page of periodic $\overline{t}$-Bockstein spectral sequence computing 
$\pi_{\star}^{C_2}\grmot^*\TPR(\mZ_2)/(\ov_0,\ov_1)$. The group $E_3^{m+n\sigma,a,b}$ appears in 
bidegree $(m,n)$. Each bullet $\bullet$ represents a copy of $\bF_2$. 
Black bullets indicate classes in motivic filtration $0$ and black lines indicate multiplications in 
motivic filtration $0$. Blue bullets indicate classes in motivic filtration $1$ and blue dashed lines 
indicate multiplications in motivic filtration $1$. 
\label{fig3}
}
\end{figure}

\begin{defin}\label{def:Aij}
Following Section~\ref{sec:trace-Tate}, we have a commutative diagram 
\[ 
\begin{tikzcd}
0\ar[r] & A_{11} \ar[r]\ar[d]  & T\ar[r] \ar[d] &  A_{10}\ar[r] \ar[d] & 0  \\
0\ar[r] & \mathrm{Nyg}^{\ge 1}\ar[r]\ar[d]  & \pi_*\grmot^*\TCR^{-}(\BPR\langle n\rangle)/(\ov_0,\cdots ,\ov_{n+1})\ar[r] \ar[d]  & \mathrm{Nyg}^{=0}\ar[r]\ar[d]  & 0  \\
0\ar[r] & A_{01}\ar[r] & F \ar[r] & A_{00}\ar[r] & 0 
\end{tikzcd}
\]
with exact rows and columns, where 
\begin{align*}
F&:=\mathrm{im} (\can : \pi_*\grmot^*\TCR^{-}(\BPR\langle n\rangle)/(\ov_0,\cdots ,\ov_{n+1})\longrightarrow \pi_*\grmot^*\TPR(\BPR\langle n\rangle)/(\ov_0,\cdots ,\ov_{n+1})) \,, \\ 
T&:=\ker (\pi_*\grmot^*\TCR^{-}(\BPR\langle n\rangle)/(\ov_0,\cdots ,\ov_{n+1})\to F)
\end{align*}
and we define 
\begin{align*}
A_{00}&:=\mathrm{Nyg}^{\ge 1}\cap T\,, \quad & \quad A_{10}&:=T/A_{00} \,, \\
A_{01}&:=\mathrm{Nyg}^{\ge 1}/A_{00}\,, \quad & \quad A_{11}&:=F/A_{01}\,.
\end{align*}
\end{defin}
\begin{prop}\label{prop:TCR-Z}
We can identify the $\EE_\infty$-page of the $\ol{t}$-Bockstein spectral sequence converging to 
$\pi_{\star}^{C_2}\grmot^*\TCR^{-}(\mZ )/(\ov_0,\ov_1)$ with 
\[ 
\EE_{\infty}=\pi_{\star}^{C_2}\mF_2[\ot^2, \omu^2]/(\ot^2\omu^2)\langle \olambda_1\rangle 
\oplus \pi_{\star}^{C_2}\mF_2\{\ot\olambda_1 \}
\]
and there is an isomorphism of $\pi_{\star}^{C_2}\mF_2\langle \olambda_1\rangle$-modules
\[ 
\pi_{\star}^{C_2}\grmot^*\TCR^{-}(\mZ)\cong A_{00}\oplus A_{01}\oplus A_{10}\oplus A_{11}
\]
where $A_{ij}$ (Definition~\ref{def:Aij}) can be identified with 
\begin{align*} A_{00}&=\pi_{\star}^{C_2}\mF_2\langle \olambda_1\rangle \,, \quad & \quad 
A_{10}&=\pi_{\star}^{C_2}\mF_2\langle \olambda_1\rangle[\ot^{2}]\{\ot^2\} \,, \\ 
A_{01}&=\mathbb{F}_p\langle \olambda_1\rangle [\omu^2]\{\omu^2\} \,, \quad & \quad 
A_{11}&=\pi_{\star}^{C_2}\mF_2\{\Xi_1\}
\end{align*}
with $\Xi_1$ detected by $\ot\olambda_1$. 
\end{prop}

\begin{proof}
This is determined from the periodic $\ot$-Bockstein spectral sequence. 
It suffices to see what classes are in the image of the differentials that cross from negative Nygaard filtration to 
positive Nygaard filtration, which is exactly $\pi_{\star}^{C_2}\mF_2\{\ot\olambda_1 \}$ as a 
$\pi_{\star}^{C_2}\mF_2\langle \olambda_1\rangle$-module. To see that 
\[ 
\pi_{\star}^{C_2}\grmot^*\TCR^{-}(\mZ)\cong A_{00}\oplus A_{10}\oplus A_{01}\oplus A_{11}
\]
as $\pi_{\star}^{C_2}\mF_2\langle \olambda_1\rangle$-modules, we observe that
\begin{align*}
A_{00} &=\pi_{\star}^{C_2}\mF_2\langle \olambda_1\rangle \,,\\
A_{10} &=\pi_{\star}^{C_2}\mF_2\langle \olambda_1\rangle [\ot^2]\{\ot^2\} \,,\\
A_{01} &=\pi_{\star}^{C_2}\mF_2\langle \olambda_1\rangle[\omu^2]\{\omu^2\}
\end{align*}
as $\pi_{\star}^{C_2}\mF_2\langle \olambda_1\rangle$-modules, and consequently, each of the short exact sequences in Definition~\ref{def:Aij} split. 
Each of these identifications follow from the fact that the canonical map is given by inverting $\ot^2$, along with  
examination of Nygaard filtrations. 
\end{proof}

\begin{proposition}
There is a class 
\[
^{\TCR}\olambda_1\in \pi_{2\rho-1}^{C_2}\gr_{\mot}^*\TCR(\mZ)/(\ov_0)
\]
mapping to the class 
$\olambda_1\in \pi_{2\rho-1}^{C_2}\grmot^*\THR(\mZ)/(\ov_0)$. 
\end{proposition}

\begin{proof}

Note that $\lambda_1\in \pi_{3}\grmot^*\THH(\bZ)/(2)$ is detected by a cocycle 
\[ 
\sigma^2t_1\in \pi_{4} \TC^{-}(\bZ/\MW^{\otimes 2})/(2)
\]
and $\TCR^{-}(\mZ/\MWR^{\otimes 2})/(2)$ is strongly even so we produce a cocycle
\[
\sigma^{\rho}\ot_1\in \pi_{2\rho}^{C_2}\TCR^{-}(\mZ/\MWR^{\otimes 2})/(2)
\] 
mapping to 
$\bar{\lambda}_1\in \pi_{2\rho}^{C_2}\THR(\mZ/\MWR^{\otimes 2})/(2)$. 
A similar argument can be made for $\TPR(\mZ)$, so considering commutativity of the diagrams
\[
\begin{tikzcd}
\pi_{2\rho}^{C_2}\TCR^{-}(\mZ/(\MWR^{\otimes 2})/2 \arrow{r}{f} \arrow{d}{\res}& \pi_{2\rho}^{C_2}\TPR(\mZ/\MWR^{\otimes 2})/2 \arrow{d}{\res}\\ 
\pi_{4}^{C_2}\TC^{-}(\bZ/\MW^{\otimes 2})/2  
\arrow{r}{f} & \pi_{4}^{C_2}\TP(\bZ/\MW^{\otimes 2})/2
\end{tikzcd}
\]
for $f\in \{\can,\varphi\}$ we determine that $\overline{\lambda}_1$ lifts to a 
class in $\pi_{2\rho -1}\grmot^*\TCR(\mZ)/(\ov_0)$. 
\end{proof}

\begin{cor}\label{cor:lambda1}
The class 
$\olambda_1\in \pi_{2\rho-1}\grmot^*\THR(\mZ)/(\ov_0)$ 
lifts to a class 
\[ 
^{\TCR}\olambda_1\in \pi_{2\rho-1}\grmot^*\TCR(\tmf_1(3))/(\ov_0)
\]
and consequently it also lifts to a class  
\[ 
^{\TCR}\olambda_1 \in \pi_{2\rho-1}\grmot^*\TCR(\kr)/(\ov_0) \,.
\]
\end{cor}
\begin{proof}
Since there are maps of $C_2$-$\bE_{\infty}$-rings $\tmf_1(3)\to \kr\to H\mZ$, 
it suffices to lift $\overline{\lambda}_1$ to $\pi_{2\rho-1}\grmot^*\TCR(\tmf_1(2))$. 
The starting point is the fact that we can lift $\overline{\lambda}_1\in \pi_{2\rho-1}\grmot^*\THR(\mZ)$ 
to a class in $\pi_{2\rho-1}\grmot^*\TCR(\mZ)$.
We write $\BP\langle 2\rangle$ for the underlying $\bE_{\infty}$-algebra in spectra associated to $\tmf_1(3)$. 
By~\cite{BM94,Wal82}, the composite 
\[ 
\TC_{3}(\BP\langle 2\rangle)/(2)\longrightarrow \TC_3( \bZ)/(2) 
 \longrightarrow \TC^-_3(\bZ)/(2) \longrightarrow \THH_{3}(\bZ)/(2)
\]
is a surjection so we can produce a cocycle $\lambda_1\in \pi_{3}\gr_{\mot}^*\TC^-(\BP\langle 2 \rangle)/(2)$ 
that detects $\lambda_1$. We then note that $\lambda_1$ is detected by a cocycle
\[ 
\sigma^2 t_1 \in \pi_{4} \TC^{-}(\BP\langle 2\rangle/\MW^{\otimes 2})/(2)
\]
and $\TCR^{-}(\tmf_1(3)/\MWR^{\otimes 2})/(2)$ is strongly even so we produce a cocycle
\[
\sigma^{\rho}\ol{t}_1\in \pi_{2\rho}^{C_2}\TCR^{-}(\tmf_1(3)/\MWR^{\otimes 2})/(2)
\] 
mapping to 
$\bar{\lambda}_1\in \pi_{2\rho}^{C_2}\TCR^{-}(\mZ/\MWR^{\otimes 2})/(2)$. 

The same argument provides a class $\bar{\lambda}_1\in \pi_{2\rho}^{C_2}\TPR(\tmf_1(3)/\MWR^{\otimes 2})/(2)$ 
and proves that 
\[
\text{can}(\bar{\lambda}_1)=\bar{\lambda}_1=\varphi(\bar{\lambda}_1)
\] 
so we produce a class 
\[
\sigma^{\rho}\ol{t}_1\in \pi_{2\rho}^{C_2}\TCR(\tmf_1(3)/\MWR^{\otimes 2})/(2)\,.
\]
Since $\sigma^2t_1$ is a cocycle, we can again use strong evenness to prove that 
$\sigma^{\rho}\ol{t}_1$ is a cocycle representative for $\ol{\lambda}_1$. 
\end{proof}

\begin{notation}
From now on we simply write $\olambda_1$ for $^{\TCR}\olambda_1$ 
\end{notation}

\begin{defin}\label{def:lambda-classes}
We define 
\[ 
\bar{\lambda}_{2}=Q^{2\rho}\olambda_1 \in \pi_{3\rho+\sigma} \grmot^*\TCR(\kr)
\]
using the power operation from Lemma~\ref{lem:BW-power}. We define 
\begin{align*}
\bar{\lambda}_{2}&=Q^{2\rho}\bar{\lambda}_{1}\in \pi_{3\rho+\sigma}\grmot^*\TCR(\tmf_1(3)) \,, \\
\bar{\lambda}_{3}&=Q^{4\rho}\bar{\lambda}_{2}\in \pi_{7\rho+\sigma}\grmot^*\TCR(\tmf_1(3))
\end{align*}
using the power operation from Lemma~\ref{lem:BW-power}. 
\end{defin}

\begin{remark}
By construction, the class $\bar{\lambda}_2$ is nonzero by examining the commutative diagram 
\[
\begin{tikzcd}
\pi_{\rho+\sigma}^{C_2}\grmot^*\TCR(\kr)\ar[r,"Q^{2\rho}"] \ar[d]  &  \pi_{3\rho+\sigma}^{C_2}\gr_{\mot}^*\TCR(\kr)  \ar[d] \\
\pi_{\rho+\sigma}^{C_2}\grmot^*\THR(\kr)\ar[r,"Q^{2\rho}"] &  \pi_{3\rho+\sigma}^{C_2}\grmot^*\THR(\kr) 
\end{tikzcd}
\]
and the two potential meanings of the notation $\bar{\lambda}_2$ are compatible. 
Similarly, $\overline{\lambda}_2$ and $\overline{\lambda}_3$ are non-zero and map to classes 
in $\pi_{\star}^{C_2}\grmot^*\THR(\tmf_1(3))$ with the same name. 
\end{remark}


We now conclude with a computation of Real syntomic cohomology of the $2$-adic integers. 

\begin{thm}\label{thm:synz}
We can identify the mod $(\ov_0,\ov_1)$ Real syntomic cohomology of the $2$-adic integers with 
\[
\pi_{\star}^{C_2}\mF_2\langle \olambda_1,\partial \rangle \oplus \pi_{\star}^{C_2}\mF_2\{\ol{\Xi}_1\}   \,.
\]
Moreover, the $\ov_1$-Bockstein spectral sequence collapses and we can identify the mod $(\ov_0)$ 
Real syntomic cohomology of the $2$-adic integers with 
\[
\pi_{\star}^{C_2}\mF_2[\ov_1]\langle \olambda_1,\partial \rangle \oplus \pi_{\star}^{C_2}\mF_2[\ov_1]\{\ol{\Xi}_1\}   \,.
\]
\end{thm}

\begin{proof}
We first prove the claim in the first sentence of the theorem statement.
Note that 
\[ \pi_*\grmot^*\TCR^{-}(\mZ)/(\ov_0,\ov_1)\cong A_{00}\oplus A_{10}\oplus A_{01}\oplus A_{11}%
\]
by Proposition~\ref{prop:TCR-Z} and we identified
\[
\pi_*\grmot^*\TPR(\mZ)/(\ov_0,\ov_1)\cong A_{00}\oplus A_{10}\oplus A_{01}
\]
in Proposition~\ref{prop:TPRZ}.  
We determine that 
\[ 
\can(A_{00}\oplus A_{10})=A_{00}\oplus A_{10}\,,
\]
\[ 
\can(A_{01}\oplus A_{11})=0
\]
by construction. 
We also determine that 
\[ \varphi(A_{10}\oplus A_{11})=0
\]
by the map of spectral sequences 
\[
\begin{tikzcd}
\pi_*\grmot^*\THR(\mZ)/(\ov_0,\ov_1)[t]\ar[r] \ar[d] & 
\pi_*\grmot^*\THR(\mZ)^{t_{C_2}\mu_2}/(\ov_0,\ov_1)[t]  \ar[d] \\
\pi_*\grmot^*\TCR^{-}(\mZ)/(\ov_0,\ov_1)\ar[r]  & 
\pi_*\grmot^*\TPR(\mZ)/(\ov_0,\ov_1)   
\end{tikzcd}
\]
which is zero in positive Nygaard filtration on $\mathrm{E}_2$-pages and converges to the 
Frobenius map $\varphi$ up to isomorphism. 
From this we also determine that 
\[
\varphi(A_{00}\oplus A_{10})=A_{00}\oplus A_{10} \,.
\]
We also know that  $\olambda_1$ is equalized by $\can$ and $\varphi$ by 
Corollary~\ref{cor:lambda1}. Therefore, 
\[ 
(\can-\varphi)(A_{00}\oplus A_{11})=0\,,
\]
\[
(\can-\varphi)(A_{10}\oplus A_{01})=A_{10}\oplus A_{01}\,.
\]
Consequently, we also determine that 
\[\ker(\can-\varphi)=A_{00}\oplus A_{11}\,,\]  
\[ 
\mathrm{coker}(\can-\varphi)=A_{00} \,.
\]
Since $A_{00}$ is a a free $\pi_{\star}^{C_2}\mF_2\otimes \Lambda( \olambda_1)$-module, the short exact sequence 
\[ 0 \to A_{00} \to \pi_{\star}^{C_2}\grmot^*\TCR(\mZ) \to A_{00}\oplus A_{11} \to 0 \]
is split. Altogether, this proves the first sentence of the theorem statement.

To prove the second sentence of the theorem statement, we use the $\bar{v}_1$-Bockstein spectral sequence 
\[
\EE_1 = \gr^*_{\mot}\TCR(\mZ)/(\ov_0,\ov_1)[v_1] \Rightarrow \gr^*_{\mot}\TCR(\mZ)/(\ov_0).
\]
Since we know that the classes in $\pi_{\star}^{C_2}\mF_2$ are permanent cycles in the 
subsequent motivic spectral sequence
\[
\pi_*\gr^*_{\mot}\TCR(\mZ)/(\ov_0)\Longrightarrow \pi_*\TCR(\mZ)/(2) 
\]
they must also be permanent cycles in the $\ov_1$-Bockstein spectral sequence and they cannot be boundaries. 
We observe that there are no other possible differentials for bidegree reasons (note that the differentials 
in the $\ov_1$-Bockstein spectral sequence are $\ov_1$-linear).
\end{proof}

\begin{defin}\label{filtration-of-TCR}
Let 
\[ 
\mathrm{Fil}_{\mot}^q \pi_\star^{C_2} \TCR(\mZ)/(2):=\im ( \pi_{\star}^{C_2}\fil_{\mot}^q\TCR(\mZ)/(2)\to 
\pi_{\star}^{C_2} \TCR(\mZ)/(2) ),
\]
\[ 
\mathrm{Gr}_{\mot}^q\pi_\star^{C_2} \TCR(\mZ)/(2):= \left( \mathrm{Fil}_{\mot}^q\pi_\star^{C_2} \TCR(\mZ)/(2) \right) / \left( \mathrm{Fil}_{\mot}^{q+1}\pi_\star^{C_2} \TCR(\mZ)/(2) \right) \,.
\]
\end{defin}

\begin{thm}\label{thm:tcrz}
There is an isomorphism of $\pi_{\star}^{C_2}\mF_2$-modules 
\[
\mathrm{Gr}_{\mot}^*\left (\pi_{\star}^{C_2}\TCR(\mZ)/(2)\right )\cong 
\pi_{\star}^{C_2}\mF_2[\ov_1]\langle \olambda_1,\partial \rangle \oplus \pi_{\star}^{C_2}\mF_2[\ov_1]\{\ol{\Xi}_1\}\
\]
with $\eta_{C_2}$-multiplications given in Figure~\ref{fig:motivic-Z2}. 
\end{thm}

\begin{proof}
The motivic spectral sequence collapses since the classes in $\pi_{\star}^{C_2}\mF_2$ are permanent cycles 
in the motivic spectral sequence, using multiplicative structure, and using the fact that 
the $\pi_{\star}^{C_2}\mF_2$-algebra generators are in Adams weight $[0,2]$ together with parity considerations; 
see Figure~\ref{fig:motivic-Z2}. 

The $\eta_{C_2}$-action is determined from the fact that the Hurewicz image of $\eta_{C_2}$ is $\Xi_1$ by 
Theorem~\ref{thm:detection} and the fact that on underlying the Hurewicz image of $\eta$ and $\eta^2$ is $\Xi_1$ 
and $\partial\lambda_1$ by~\cite{Rog99,Rog99b} for example (in~\cite[pp.~229]{Rog99b} Rognes attributes this to 
Dennis--Stein~\cite{DS75}). 
The $\partial$ multiplications come from the fact that $\{\partial, \partial \overline{\lambda}_1\}$ are in the image 
of $\{1,\overline{\lambda}_1\}$ via the connecting homomorphism 
\[ 
\partial : \pi_*\grmot^*\TP(\bZ)/(2)\longrightarrow \pi_{*}\grmot^{*}\TC(\bZ)/(2)  
\]
where we suppress grading shifts.
\end{proof}

\begin{remark}
At the underlying level, Theorem~\ref{thm:tcrz} coincides with~\cite[Theorem~0.4]{Rog99}, 
but we should emphasize that op.~cit. is  used as input in our computation. 
To see the relationship between the computations, note that on underlying the group 
$\pi_{i}^e\mathrm{TCR}(\mathbb{Z}_2)/2$ is $\mathbb{Z}/2$ for $i=-1,0$, 
the group has order $4$ for $i=2n$ where $n\ge 2$ and $i=1$, and the group has order $8$ for $i=2n+1$ where $n\ge 1$. 
\end{remark}

\begin{remark}
Note that the fact that the Hurewicz image of $\eta_{C_2}$ is $\Xi_1$ and the 
Hurewicz image of $\eta_{C_2}^2$ is $\partial \overline{\lambda}_1$ is consistent with the computation 
\[ 
\pi_0(\TCR(\mathbb{Z}_2)^{\phi C_2})=\mathbb{Z}/8
\]
from~\cite[Proposition~3.3]{HLN21} and~\cite[Theorem~C]{DMP24} since $\eta_{C_2}^{\Phi C_2}=2$. 
Letting $(\overline{v}_1^8)^{\Phi C_2}=x$, $\overline{\lambda}_1^{\Phi C_2}=e$, and $\partial^{\Phi C_2}=f$ 
we also have a consistent answer, though we have decided to refrain from carrying out the $\overline{v}_0$-Bockstein 
spectral sequence. 
\end{remark}

\begin{figure}
\resizebox{.5\textwidth}{!}{
\begin{tikzpicture}[radius=.08,scale=1.4]
\node [below] at (-.9,-.1) {-1};
\node [below] at (.1,-.1) {0};
\node [below] at (1.1,-.1) {$\rho-1$};
\node [below] at (2.1,-.1) {$2\rho-2$};
\node [below] at (3.1,-.1) {$2\rho-1$};

\foreach \s in {-1,1,2,3} \node [left] at (-.1,\s) {$\s$};

\draw [thin,color=lightgray] (-2,-1) grid (4,3);

\clip (-2,-1) -- (-2,3) -- (4,3) -- (4,-2) -- cycle;


\draw [fill] (0,0) circle;
\node [above] at (.1,.1) {1};
\draw [fill] (-1,1) circle;
\node [above] at (-.9,1.1) {$\partial$};
\draw [fill] (1,1) circle;
\node [above] at (1.1,1.1) {$\ol{\Xi}$};
\draw [fill] (3,1) circle;
\node [above] at (3.1,1.1) {$\olambda_1$};
\draw [fill] (2,2) circle;
\node [above] at (2.1,2.1) {$\partial\olambda_1$};

\draw (0,0)  -- (2,2);
\draw (0,0) -- (-1,1);
\draw (3,1) -- (2,2);
\end{tikzpicture}
}
\caption{The motivic spectral sequence computing $\pi_{\star}^{C_2}\TCR(\mZ_2)/(2)$. 
Each bullet $\bullet$ represents a copy of $\pi_{\star}^{C_2}\mF_2[\ov_1]$. 
The horizontal axis represents the stem (a $C_2$-representation) and the vertical axis represents the Adams weight. 
Lines of slope $1$ indicate multiplication by $\eta_{C_2}$. Lines of slope $-1$ indicate multiplication by $\partial$.
\label{fig:motivic-Z2}
}
\end{figure}

\subsection{Real syntomic cohomology of $\kr$ and $\tmf_1(3)$}\label{sec:syntomic-kr}

As our main application, we compute the Real syntomic cohomology of $\kr$ mod $(\ov_0,\ov_{1},\ov_2)$ and $\tmf_1(3)$
 mod $(\ov_0,\ov_{1},\ov_2,\ov_3)$, proving a Lichtenbaum--Quillen type property for Real syntomic cohomology of $\kr$. 
 Recall that in this section we work in the $2$-complete setting throughout. 

We first recall that by Corollary~\ref{cor:vn-thr}, we have isomorphisms of $\pi_{\star}^{C_2}\mF_2$-algebras
\begin{align}
\pi_{\star}^{C_2}\grmot^{*}\THR(\kr)/(\ov_0,\ov_1,\ov_2) & \cong  \pi_{\star}^{C_2}\mF_2 [\omu^{4} ]\langle \olambda_1,\olambda_2,\ovarepsilon_2\rangle  \,, \\ 
\pi_{\star}^{C_2}\grmot^{*}\THR(\tmf_1(3))/(\ov_0,\ov_1,\ov_2,\ov_3) & \cong  \pi_{\star}^{C_2}\mF_2 [\omu^{8} ]\langle \olambda_1,\olambda_2,\olambda_3,\ovarepsilon_3\rangle  \,.
\end{align}
We will show that $\olambda_1$ and $\olambda_2$ (resp. $\olambda_1$, $\olambda_2$, and $\olambda_3$) lift to classes in 
$\pi_{\star}^{C_2}\grmot^*\TCR(\kr)$ (resp.  $\pi_{\star}^{C_2}\grmot^*\TCR(\tmf_1(3))$). 
We have already accomplished this in the case of $\olambda_1$. 

\begin{prop}\label{prop:differentials}
Let $-1\le n\le 2$. In the spectral sequences \eqref{prismaticSS} and \eqref{TC-SS} with $R = \BPR\langle n \rangle$, 
there are differentials generated by 
\[ 
d_{2^{i}}(\ot^{2^{i-1}})\dot{=}\ot^{2^{i}+2^{i-1}}\bar{\lambda}_{i}
\]
for $1\le i\le n+1$, and the class $\ot^{2^{n+1}}$ is a permanent cycle. 
\end{prop}

\begin{proof}
In the cobar complex computing $\grmot^*\TCR^{-}(\bS)$, there is a differential 
\[
\partial(\ot)=-\chi (\ot_1)\ot^{2} \mod (\ov_0,\ov_1,\ot^{4})
\]
by Lemma~\ref{lem:right-unit}. 
We also have differentials 
\begin{equation}\label{total-diff}
\partial(\ot^{2^j})=-\chi (\ot_1)\ot^{2^{j+1}} \mod (\ov_0,\ov_1,\ot^{2^{j}+2})\,.
\end{equation}
Since the spectral sequences \eqref{prismaticSS} and~\eqref{TC-SS} can be viewed as spectral sequences associated to a 
filtered chain complex, the $d_{2^j}$ differential is therefore determined by the formula \eqref{total-diff} producing the 
differential 
\[
d_{2^j}(\ot^{2^{j-1}})\dot{=}\ot^{2^{j+1}+2^j}\olambda_j
\]
for $1\le j\le n+1$ because $[\ot_1]$ maps to $\ot\olambda_1$ and thus, by compatibility with power operations (Section~\ref{sec:toolkit-power}), $[\ot_i]$ maps to $\ot\olambda_i$ for $2\le i\le n+1$. 

This also implies that $d_{r}(\ot^{2^{j-1}})=0$ for $r<j$, which suffices to prove that $\ot^{2^{n+1}}$ is a permanent cycle.
This follows by considering Nygaard filtration, motivic filtration, and $\mathrm{RO}(C_2)$-grading; see Figure~\ref{fig6}. 
\end{proof}

\begin{prop}
The Real prismatic cohomology of $\BPR\langle n\rangle$ mod $(\ov_0,\ov_1,\cdots,\ov_{n+1})$ is 
\[  
\pi_{\star}^{C_2}\grmot^{*}\TPR(\BPR\langle n\rangle)/(\ov_0,\cdots ,\ov_{n+1})\cong 
\pi_{\star}^{C_2}\bF_{2}[\ol{t}^{\pm 2^{n+1}}]\langle \olambda_1,\cdots ,\olambda_{n+1}\rangle \,. 
\]
\end{prop}

\begin{proof}
We first observe that there is a differential $d_1(\ovarepsilon_{n+1})=\ot\omu^{2^{n+1}}$ because $\ov_{n+1}$ is detected by $\ot\omu^{2^{n+1}}$ and $\ovarepsilon_{n+1}$ is 
chosen as a nullhomotopy of $\ov_{n+1}$. The argument is similar to the one given in the proof of Proposition~\ref{prop:tpfp}. After running this differential and the differentials determined in Proposition~\ref{prop:differentials} there is no room for 
further differentials in light of Lemma~\ref{lem:pi0tcr} as well as Theorem~\ref{thm:Segal} and 
Proposition~\ref{prop:useful-iso}. See Figure~\ref{fig6} in the case of $\kr$. 
The case of $\tmf_1(3)$ is similar.
\end{proof}

\begin{figure}[h!]
\resizebox{.8\textwidth}{!}{
\begin{tikzpicture}[radius=.08,scale=.05cm]
\foreach \n in {-7,-4,-3,-2,-1,0,...,7} \node [below] at (\n+.1,-.1) {$\n$};
\foreach \s in {-7,-4,-3,-2,-1,1,2,...,7} \node [left] at (-.1,\s) {$\s$};
\draw [thin,color=lightgray] (-7,-7) grid (7,7);
\clip (-7.1,-7.1) -- (7.1,-7.1) -- (7.1,7.1) -- (-7.1,7.1) -- cycle;

\node [right] at (0,0) {1};
\node [right] at (0,-1) {$a_{\sigma}$};
\node [right] at (1,-1) {$u_{\sigma}$};
\node [right] at (-2,2) {$\theta$};

\foreach \s in {-4,4}
    \node [right] at (-\s,-\s) {$\ol{t}^{\s}$};

\begin{scope}[black]

\foreach \s in {-8,-4,0,4,8}
    \foreach \m in {-10,-9,-8,-7,-6,-5,-4,-3,-2,-1,0}
        \foreach \n in {-10,...,\m} 
            {\draw [fill] (-\m+\s,\n+\s) circle;
            }

\foreach \s in {-8,-4,0,4,8}
    \foreach \n in {0,1,...,10}
        {
        \draw (\n+\s,-\n+\s) -- (\n+\s,-5+\s);
        }
\foreach \s in {-8,-4,0,4,8}
    \foreach \n in {0,1,...,10}
        {
        \draw (\s,-\n+\s) -- (5-\n+\s,-5+\s);
        }


\foreach \s in {-4,0,4}
    \foreach \n in {2,...,14} 
        {\draw [fill] (-2+\s,\n+\s) circle;}
\foreach \s in {-4,0,4}
    \foreach \n in {3,...,14} 
        {\draw [fill] (-3+\s,\n+\s) circle;}
\foreach \s in {-4,0,4}
    \foreach \n in {4,...,14} 
        {\draw [fill] (-4+\s,\n+\s) circle;}
\foreach \s in {-4,0,4}
    \foreach \n in {5,...,14} 
        {\draw [fill] (-5+\s,\n+\s) circle;}
\foreach \s in {-4,0,4}
    \foreach \n in {6,...,14} 
        {\draw [fill] (-14+\s,\n+\s) circle;}
 \foreach \s in {-4,0,4}
    \foreach \n in {7,...,14} 
        {\draw [fill] (-14+\s,\n+\s) circle;}
 \foreach \s in {-4,0,4}
    \foreach \n in {8,...,14} 
        {\draw [fill] (-14+\s,\n+\s) circle;}
 \foreach \s in {-4,0,4}
    \foreach \n in {9,...,14} 
        {\draw [fill] (-14+\s,\n+\s) circle;}
  \foreach \s in {-4,0,4}
    \foreach \n in {10,...,14} 
        {\draw [fill] (-14+\s,\n+\s) circle;}
\foreach \s in {-4,0,4}
    \foreach \n in {11,...,14} 
        {\draw [fill] (-14+\s,\n+\s) circle;}
\foreach \s in {-4,0,4}
    \foreach \n in {12,13,14} 
        {\draw [fill] (-14+\s,\n+\s) circle;}
\foreach \s in {-4,0,4}
    \foreach \n in {13,14} 
        {\draw [fill] (-14+\s,\n+\s) circle;}
\foreach \s in {-4,0,4}
    \foreach \n in {14} 
        {\draw [fill] (-14+\s,\n+\s) circle;}

\foreach \s in {-4,0,4}
    \foreach \n in {2,...,10}
        {\draw (-2+\s,\n+\s) -- (-14+\n+\s,12+\s);}
\foreach \s in {-4,0,4}
    \foreach \n in {2,...,10} 
        {\draw (-\n+\s,\n+\s) -- (-\n+\s,12+\s);}
\end{scope}

\node [left] at (2.85,4.15) {$\olambda_2$};

\begin{scope}[blue,dotted,xshift=2.85cm,yshift=4.15cm]

\foreach \s in {-8,-4,0,4,8}
    \foreach \m in {-10,-9,-8,-7,-6,-5,-4,-3,-2,-1,0}
        \foreach \n in {-10,...,\m} 
            {\draw [fill] (-\m+\s,\n+\s) circle;
            }

\foreach \s in {-8,-4,0,4,8}
    \foreach \n in {0,1,...,10}
        {
        \draw (\n+\s,-\n+\s) -- (\n+\s,-10+\s);
        }
\foreach \s in {-8,-4,0,4,8}
    \foreach \n in {0,1,...,10}
        {
        \draw (\s,-\n+\s) -- (10-\n+\s,-10+\s);
        }


\foreach \s in {-4,0,4}
    \foreach \n in {2,...,14} 
        {\draw [fill] (-2+\s,\n+\s) circle;}
\foreach \s in {-4,0,4}
    \foreach \n in {3,...,14} 
        {\draw [fill] (-3+\s,\n+\s) circle;}
\foreach \s in {-4,0,4}
    \foreach \n in {4,...,14} 
        {\draw [fill] (-4+\s,\n+\s) circle;}
\foreach \s in {-4,0,4}
    \foreach \n in {5,...,14} 
        {\draw [fill] (-5+\s,\n+\s) circle;}
\foreach \s in {-4,0,4}
    \foreach \n in {6,...,14} 
        {\draw [fill] (-14+\s,\n+\s) circle;}
 \foreach \s in {-4,0,4}
    \foreach \n in {7,...,14} 
        {\draw [fill] (-14+\s,\n+\s) circle;}
 \foreach \s in {-4,0,4}
    \foreach \n in {8,...,14} 
        {\draw [fill] (-14+\s,\n+\s) circle;}
 \foreach \s in {-4,0,4}
    \foreach \n in {9,...,14} 
        {\draw [fill] (-14+\s,\n+\s) circle;}
  \foreach \s in {-4,0,4}
    \foreach \n in {10,...,14} 
        {\draw [fill] (-14+\s,\n+\s) circle;}
\foreach \s in {-4,0,4}
    \foreach \n in {11,...,14} 
        {\draw [fill] (-14+\s,\n+\s) circle;}
\foreach \s in {-4,0,4}
    \foreach \n in {12,13,14} 
        {\draw [fill] (-14+\s,\n+\s) circle;}
\foreach \s in {-4,0,4}
    \foreach \n in {13,14} 
        {\draw [fill] (-14+\s,\n+\s) circle;}
\foreach \s in {-4,0,4}
    \foreach \n in {14} 
        {\draw [fill] (-14+\s,\n+\s) circle;}

\foreach \s in {-4,0,4}
    \foreach \n in {2,...,10}
        {\draw (-2+\s,\n+\s) -- (-14+\n+\s,12+\s);}
\foreach \s in {-4,0,4}
    \foreach \n in {2,...,10} 
        {\draw (-\n+\s,\n+\s) -- (-\n+\s,12+\s);}
\end{scope}

\node [left] at (.85,2.15) {$\olambda_1$};

\begin{scope}[blue,dotted, xshift=.85cm,yshift=2.15cm]

\foreach \s in {-8,-4,0,4,8}
    \foreach \m in {-10,-9,-8,-7,-6,-5,-4,-3,-2,-1,0}
        \foreach \n in {-10,...,\m} 
            {\draw [fill] (-\m+\s,\n+\s) circle;
            }

\foreach \s in {-8,-4,0,4,8}
    \foreach \n in {0,1,...,10}
        {
        \draw [dotted] (\n+\s,-\n+\s) -- (\n+\s,-10+\s);
        }
\foreach \s in {-8,-4,0,4,8}
    \foreach \n in {0,1,...,10}
        {
        \draw [dotted] (\s,-\n+\s) -- (10-\n+\s,-10+\s);
        }


\foreach \s in {-4,0,4}
    \foreach \n in {2,...,14} 
        {\draw [fill] (-2+\s,\n+\s) circle;}
\foreach \s in {-4,0,4}
    \foreach \n in {3,...,14} 
        {\draw [fill] (-3+\s,\n+\s) circle;}
\foreach \s in {-4,0,4}
    \foreach \n in {4,...,14} 
        {\draw [fill] (-4+\s,\n+\s) circle;}
\foreach \s in {-4,0,4}
    \foreach \n in {5,...,14} 
        {\draw [fill] (-5+\s,\n+\s) circle;}
\foreach \s in {-4,0,4}
    \foreach \n in {6,...,14} 
        {\draw [fill] (-14+\s,\n+\s) circle;}
 \foreach \s in {-4,0,4}
    \foreach \n in {7,...,14} 
        {\draw [fill] (-14+\s,\n+\s) circle;}
 \foreach \s in {-4,0,4}
    \foreach \n in {8,...,14} 
        {\draw [fill] (-14+\s,\n+\s) circle;}
 \foreach \s in {-4,0,4}
    \foreach \n in {9,...,14} 
        {\draw [fill] (-14+\s,\n+\s) circle;}
  \foreach \s in {-4,0,4}
    \foreach \n in {10,...,14} 
        {\draw [fill] (-14+\s,\n+\s) circle;}
\foreach \s in {-4,0,4}
    \foreach \n in {11,...,14} 
        {\draw [fill] (-14+\s,\n+\s) circle;}
\foreach \s in {-4,0,4}
    \foreach \n in {12,13,14} 
        {\draw [fill] (-14+\s,\n+\s) circle;}
\foreach \s in {-4,0,4}
    \foreach \n in {13,14} 
        {\draw [fill] (-14+\s,\n+\s) circle;}
\foreach \s in {-4,0,4}
    \foreach \n in {14} 
        {\draw [fill] (-14+\s,\n+\s) circle;}

\foreach \s in {-4,0,4}
    \foreach \n in {2,...,10}
        {\draw [dotted] (-2+\s,\n+\s) -- (-14+\n+\s,12+\s);}
\foreach \s in {-4,0,4}
    \foreach \n in {2,...,10} 
        {\draw [dotted] (-\n+\s,\n+\s) -- (-\n+\s,12+\s);}
\end{scope}

\node [left] at (2.7,6.3) {$\olambda_1\olambda_{2}$};

\begin{scope}[orange,dashed,xshift=2.7cm,yshift=6.3cm]

\foreach \s in {-8,-4,0,4,8}
    \foreach \m in {-10,-9,-8,-7,-6,-5,-4,-3,-2,-1,0}
        \foreach \n in {-10,...,\m} 
            {\draw [fill] (-\m+\s,\n+\s) circle;
            }

\foreach \s in {-8,-4,0,4,8}
    \foreach \n in {0,1,...,10}
        {
        \draw (\n+\s,-\n+\s) -- (\n+\s,-10+\s);
        }
\foreach \s in {-8,-4,0,4,8}
    \foreach \n in {0,1,...,10}
        {
        \draw  (\s,-\n+\s) -- (10-\n+\s,-10+\s);
        }


\foreach \s in {-4,0,4}
    \foreach \n in {2,...,14} 
        {\draw [fill] (-2+\s,\n+\s) circle;}
\foreach \s in {-4,0,4}
    \foreach \n in {3,...,14} 
        {\draw [fill] (-3+\s,\n+\s) circle;}
\foreach \s in {-4,0,4}
    \foreach \n in {4,...,14} 
        {\draw [fill] (-4+\s,\n+\s) circle;}
\foreach \s in {-4,0,4}
    \foreach \n in {5,...,14} 
        {\draw [fill] (-5+\s,\n+\s) circle;}
\foreach \s in {-4,0,4}
    \foreach \n in {6,...,14} 
        {\draw [fill] (-14+\s,\n+\s) circle;}
 \foreach \s in {-4,0,4}
    \foreach \n in {7,...,14} 
        {\draw [fill] (-14+\s,\n+\s) circle;}
 \foreach \s in {-4,0,4}
    \foreach \n in {8,...,14} 
        {\draw [fill] (-14+\s,\n+\s) circle;}
 \foreach \s in {-4,0,4}
    \foreach \n in {9,...,14} 
        {\draw [fill] (-14+\s,\n+\s) circle;}
  \foreach \s in {-4,0,4}
    \foreach \n in {10,...,14} 
        {\draw [fill] (-14+\s,\n+\s) circle;}
\foreach \s in {-4,0,4}
    \foreach \n in {11,...,14} 
        {\draw [fill] (-14+\s,\n+\s) circle;}
\foreach \s in {-4,0,4}
    \foreach \n in {12,13,14} 
        {\draw [fill] (-14+\s,\n+\s) circle;}
\foreach \s in {-4,0,4}
    \foreach \n in {13,14} 
        {\draw [fill] (-14+\s,\n+\s) circle;}
\foreach \s in {-4,0,4}
    \foreach \n in {14} 
        {\draw [fill] (-14+\s,\n+\s) circle;}

\foreach \s in {-4,0,4}
    \foreach \n in {2,...,10}
        {\draw  (-2+\s,\n+\s) -- (-14+\n+\s,12+\s);}
\foreach \s in {-4,0,4}
    \foreach \n in {2,...,10} 
        {\draw  (-\n+\s,\n+\s) -- (-\n+\s,12+\s);}
\end{scope}

\end{tikzpicture}
}
\caption{
The $\mathrm{E}_5=\mathrm{E}_{\infty}$-page of periodic $\ol{t}$-Bockstein spectral sequence computing 
$\pi_{\star}^{C_2}\grmot^*\TPR(\kr)$.
The group $E_2^{m+n\sigma,a,b}$ appears in bidegree $(m,n)$. Each bullet $\bullet$ represents a copy of $\bF_2$. 
Black bullets indicate classes in motivic filtration $0$ and black lines indicate multiplications in filtration $0$, 
blue bullets indicate classes in motivic filtration $1$ and dotted lines indicate multiplications in filtration $1$, 
and orange bullets indicate classes in motivic filtration $2$ and dashed lines indicate multiplications in filtration $2$. 
\label{fig6}
}
\end{figure}

\begin{prop}\label{prop:TCR-kr}
Let $0\le n\le 2$ and let $[i]_{n+1}=i\mod n+1$.
There is an isomorphism 
\[ 
\pi_{\star}^{C_2}\grmot^*\TCR^{-}(\BPR\langle n\rangle)/(\ov_0,\cdots,\ov_{n+1})=
A_{00}\oplus A_{10}\oplus A_{01}\oplus A_{11}
\]
where $A_{ij}$, defined as in Definition~\ref{def:Aij}, can be identified with 
\begin{align*}
A_{00}&=\pi_{\star}^{C_2}\mF_2\langle \olambda_1,\olambda_2\rangle \,, \\
A_{10}&= \pi_{\star}^{C_2}\mF_2\langle \olambda_1,\olambda_2\rangle [\ot^{2^{n+1}}]\{\ot^{2^{n+1}}\} \,, \\
A_{01}&= \pi_{\star}^{C_2}\mF_2\langle \olambda_1,\cdots ,\olambda_{n+1}\rangle [\omu^{2^{n+1}}]\{\omu^{2^{n+1}}\} \,, \\ 
A_{11}&=\bigoplus_{i=1}^{n+1}\pi_{\star}^{C_2}\mF_2\langle\olambda_{s} :s\in \{1,\cdots ,n+1\}\setminus \{i\}\rangle \{\ot^{p^{i-1}}\olambda_i \} \,.
\end{align*}
\end{prop}

\begin{proof}
We see that the differentials crossing from Real Nygaard filtration $<0$ to 
Real Nygaard filtration $>0$ hit the classes 
\[ 
\bigoplus_{i=1}^{n+1}\pi_{\star}^{C_2}\mF_2\langle\olambda_{s} :s\in \{1,\cdots ,n+1\}\setminus\{i\}\rangle 
\{\ot^{2^{i-1}}\olambda_i \}\,.
\]
The rest is determined by examination of the map of spectral sequences 
\[ 
\begin{tikzcd}
\pi_{\star}^{C_2}\grmot^*\THR(\BPR\langle n\rangle)/(\ov_0,\cdots ,\ov_{n+1})[\ot]\arrow{r} \arrow[Rightarrow]{d}& \pi_{\star}^{C_2}\grmot^*\THR(\BPR\langle n\rangle)/(\ov_0,\cdots ,\ov_{n+1})[\ot,\ot^{-1}] \arrow[Rightarrow]{d} \\ 
\pi_{\star}^{C_2}\TCR^{-}(\BPR\langle n\rangle)\arrow{r}& \pi_{\star}^{C_2}\TPR(\BPR\langle n\rangle)   
\end{tikzcd}
\]
from \eqref{eq:diagram-of-ss} which is given by inverting $\ot$ on $\mathrm{E}_1$-pages and by inverting $\ot^{2^{n+1}}$ on $\mathrm{E}_{\infty}$-pages. 
\end{proof}

\begin{prop}\label{prop:collapse-fixedpoint}
The spectral sequence 
\[ 
\pi_{\star}^{C_2} \grmot^*\THR(\BPR\langle n\rangle)^{t_{C_2}\mu_2}/(\ov_0,\ov_1,\ov_2)[\ot]
\Longrightarrow \pi_{\star}^{C_2} \gr_{\mot}^*\TPR(\BPR\langle n\rangle)/(\ov_0,\cdots ,\ov_{n+1})
\]
collapses at the $\EE_2$-term with $d_1$-differentials generated by the differential 
\[ 
d_1(\ovarepsilon_{n+1})=\ot\omu^{2^{n+1}}\,.
\]
We can therefore identify 
\[ \pi_{\star}^{C_2}\grmot^*\TPR(\BPR\langle n\rangle)/(\ov_0,\cdots ,\ov_{n+1})\cong A_{00}\oplus A_{10}\oplus A_{01}
\]
and
$\varphi (A_{01}\oplus A_{11})=0$
and $\varphi(A_{10}\oplus A_{00})=A_{10}\oplus A_{00}$. 
\end{prop}

\begin{proof}
The differential $d_1(\ovarepsilon_{n+1})=\ot\omu^{2^{n+1}}$ holds because $\ov_{n+1}$ is detected by $\ot\omu^{2^{n+1}}$ and $\ovarepsilon_{n+1}$ is 
chosen as a nullhomotopy of $\ov_{n+1}$. The argument is similar to the one given in the proof of Proposition~\ref{prop:tpfp}. The spectral sequence then collapses because the spectral sequence is concentrated in Nygaard filtration zero.
\end{proof}

We conclude with the following result: 
\begin{thm}\label{thm:main1}
    Let $1\le n\le 2$. The Real syntomic cohomology of $\BPR\langle n\rangle$ mod $(\ov_0,\cdots ,\ov_{n+1})$ is
\begin{equation}\label{eq:syntomic-bprn}
 \pi_{\star}^{C_2}\mF_2\langle \delta ,\bar{\lambda}_1, \cdots ,\bar{\lambda}_{n+1}\rangle\oplus
  \bigoplus_{j=1}^{n+1}  \pi_{\star}^{C_2}\mF_2\langle \bar{\lambda}_{s}:s\in \{1,\cdots ,n+1\}\setminus\{j\}
\rangle \otimes \bF_2\{ \bar{\Xi}_{j} \}.
\end{equation}
where the bidegrees are given by $|\bar{\Xi}_{j}|=((2^j-2^{j-1})\rho-1,1)$, $|\olambda_s|=(2^j\rho-1,1)$,  
$|\partial |=(-1,1)$. 

Moreover, the $\ov_{n+1}$-Bockstein spectral sequence collapses so the mod $(\ov_0,\cdots ,\ov_{n})$ 
Real syntomic cohomology of $\BPR\langle n\rangle$ is
\begin{equation}\label{eq:syntomic-bprn-2}
 \pi_{\star}^{C_2}\mF_2[\ov_{n+1}]\langle \delta ,\bar{\lambda}_1, \cdots ,\bar{\lambda}_{n+1}\rangle\oplus 
 \bigoplus_{j=1}^{n+1}  \pi_{\star}^{C_2}\mF_2[\ov_{n+1}]\langle \bar{\lambda}_{s}:s\in \{1,\cdots ,n+1\}\setminus\{j\}
\rangle \otimes \bF_2\{ \bar{\Xi}_{j} \}.
\end{equation}
\end{thm}

\begin{proof}
By Corollary~\ref{cor:lambda1} and Definition~\ref{def:lambda-classes}, the classes~$\olambda_i$ 
for $1\le i\le n+1$ are classes in mod $2$ Real syntomic cohomology of $\BPR\langle n\rangle$, 
and must therefore be equalized by $\can$ and $\varphi$. 
The map $\can$ is given by inverting $\ol{t}^{2^{n+1}}$ and it is described explicitly in Proposition~\ref{prop:TCR-kr}. 
The map $\varphi$ is given by inverting $\omu^{2^{n+1}}$ (up to multiplication by a unit) by Theorem~\ref{prop:TCR-kr} and 
Theorem~\ref{prop:collapse-fixedpoint}. From this, the computation of Real syntomic cohomology is the same as in the proof of Theorem~\ref{thm:synz}. 


To determine that the final Bockstein spectral sequence collapses it suffices to examine bidegrees of algebra generators carefully and 
note that the classes in $\pi_{\star}^{C_2}\mF_2$ must be permanent cycles because we know they are permanent 
cycles by Lemma~\ref{lem:pi0tcr}. 
\end{proof}

\begin{remark}
On underlying, the computation of syntomic cohomology of $\BP\langle 2\rangle $ has not yet appeared in the literature at the 
prime $2$. The computation at primes $p\ge 7$ can be determined from~\cite{AKACHR25}. 
\end{remark}

\begin{remark}\label{rem:finite-spectra}
Let $v_{(n,m)}$ denote a self-map of a finite $C_2$-spectrum of degree $\rho (2^{n}-1)j$ such that $v_{n,m}^{\Phi e}$ 
is a $v_n$-power self-map and $v_{(n,m)}^{\Phi C_2}$ is a $v_{m}$-power self-map as in~\cite[Definition~1.2]{BGL22}. 
For example, $2$ is a $v_{(0,0)}$-self map. If a finite $C_2$-spectrum 
$\mathbb{S}/(v_{(0,0}^{i_0},v_{(1,1)}^{i_1},\cdots ,v_{(s,s)}^{i_s})$ exists for  each $s\le 3$, then 
Theorem~\ref{thm:main1} proves that 
\[  
\TCR(\BPR\langle n\rangle)/(v_{(0,0}^{i_0},v_{(1,1)}^{i_1},\cdots ,v_{(n+1,n+1)}^{i_s})
\]
is in the thick subcategory generated by $H\mF_2$ for $n=1,2$. 
Work in progress of Burklund--Hausmann--Levy--Meier should shed light on the 
question of whether such finite spectra exist. 
\end{remark}

\begin{remark}\label{rem:less-structure}
We expect that similar methods can be used to prove an analogue of Theorem \ref{thm:main1} for 
arbitrary $\bE_{\sigma}$-$\MUR$-algebra forms of $\BPR\langle n\rangle$, along the lines of~\cite{AKHW24}. 
Once one has an $\mathbb{E}_\sigma$-algebra structure on $A$, one can apply~\cite[Corollary~3.31]{Ste25} to define $\THR(A)$ 
as a Real cyclotomic spectrum in the sense of~\cite{QS21a}. 
However, it is currently not known whether $\BPR\langle n\rangle$ has an $\bE_{\sigma}$-$\MUR$-algebra structure 
for $n\ge 3$ even though it is expected to be the case; see~\cite[Remark~1.0.14]{HW22}. 
As in Remark~\ref{rem:finite-spectra}, if the 
finite spectrum $\bS/(v_{(0,0}^{i_0},v_{(1,1)}^{i_1},\cdots ,v_{(n+1,n+1)}^{i_s})$ exists, 
then we expect that these methods would show that 
\[ 
\TCR(\BPR\langle n\rangle)/(v_{(0,0}^{i_0},v_{(1,1)}^{i_1},\cdots ,v_{(n+1,n+1)}^{i_s})
\]
is in the thick subcategory generated by $H\mF_2$ for arbitrary $n$. 
\end{remark}

\begin{figure}[h!]
\resizebox{.8\textwidth}{!}{
\begin{tikzpicture}[radius=.05,scale=1.2]
\node [below] at (-.9,-.1) {-1};
\node [below] at (.1,-.1) {0};
\node [below] at (1.1,-.1) {$\rho-1$};
\node [below] at (2.1,-.1) {$2\rho-2$};
\node [below] at (3.1,-.1) {$2\rho-1$};
\node [below] at (6.1,-.1) {$4\rho-2$};
\node [below] at (7.1,-.1) {$4\rho-1$};
\node [below] at (8.1,-.1) {$5\rho-2$};
\node [below] at (9.1,-.1) {$6\rho-3$};
\node [below] at (10.1,-.1) {$6\rho-2$};

\foreach \s in {-1,1,2,3,4} \node [left] at (-.1,\s) {$\s$};

\draw [thin,color=lightgray] (-2,-1) grid (11,4);

\clip (-2,-1) -- (-2,4) -- (11,4) -- (11,-1) -- cycle;

\draw [fill] (0,0) circle;
\node [above] at (.1,.1) {1};
\draw [fill] (-1,1) circle;
\node [above] at (-.9,1.1) {$\partial$};
\draw [fill] (1,1) circle;
\node [above] at (1.1,1.1) {$\ol{\Xi}_1$};
\draw [fill] (2,2) circle;
\node [above] at (2.1,2.1) {$\partial\olambda_1$};
\draw [fill] (3,1.1) circle;
\node [above] at (3.1,1.1) {$\ol{\Xi}_2$};
\draw [fill] (3,.9) circle;
\node [below] at (3.1,.9) {$\olambda_1$};
\draw [fill] (6,1.9) circle;
\node [below] at (6.1,1.9) {$\partial\olambda_2$};
\draw [fill] (6,2.1) circle;
\node [above] at (6.1,2.1) {$\olambda_1\ol{\Xi}_2$};
\draw [fill] (7,1) circle;
\node [below] at (7.1,.9) {$\olambda_2$};
\draw [fill] (8,2) circle;
\node [above] at (8.1,2.1) {$\ol{\Xi}_1\olambda_2$};
\draw [fill] (10,2) circle;
\node [above] at (10.1,2.1) {$\olambda_1\olambda_2$};
\draw [fill] (9,3) circle;
\node [above] at (9.1,3.1) {$\partial\olambda_1\olambda_2$};
\end{tikzpicture}
}
\caption{The mod $(\ov_0,\ov_1)$ Real syntomic cohomology of $\kr$. Each bullet $\bullet$ represents a copy of $\pi_{\star}^{C_2}\mF_2[\ov_2]$. The horizontal axis represents the stem (a $C_2$-representation) and the vertical axis represents the Adams weight.
\label{fig:motivic-kr}
}
\end{figure}

\begin{figure}[h!]
\resizebox{\textwidth}{!}{
\begin{tikzpicture}[radius=.05,scale=1.4]
\node [below] at (-.9,-.1) {-1};
\node [below] at (.1,-.1) {0};
\node [below] at (1.1,-.1) {$\rho-1$};
\node [below] at (2.1,-.1) {$2\rho-2$};
\node [below] at (3.1,-.1) {$2\rho-1$};
\node [below] at (6.1,-.1) {$4\rho-2$};
\node [below] at (7.1,-.1) {$4\rho-1$};
\node [below] at (8.1,-.1) {$5\rho-2$};
\node [below] at (9.1,-.1) {$6\rho-3$};
\node [below] at (10.1,-.1) {$6\rho-2$};
\node [below] at (14.1,-.1) {$8\rho-2$};
\node [below] at (15.1,-.1) {$8\rho-1$};
\node [below] at (16.1,-.1) {$9\rho-1$};
\node [below] at (17.1,-.1) {$10\rho-3$};
\node [below] at (18.1,-.1) {$10\rho-2$};
\node [below] at (21.1,-.1) {$12\rho-3$};
\node [below] at (22.1,-.1) {$12\rho-2$};
\node [below] at (23.1,-.1) {$13\rho-3$};
\node [below] at (24.1,-.1) {$14\rho-4$};
\node [below] at (25.1,-.1) {$14\rho-3$};

\foreach \s in {-2,-1,1,2,3,4,5} \node [left] at (-.1,\s) {$\s$};

\draw [thin,color=lightgray] (-2,-2) grid (26,5);

\clip (-2,-3) -- (-2,5) -- (26,5) -- (26,-3) -- cycle;

\node [above] at (.1,.1) {1};
\draw [fill] (0,0) circle;
\node [above] at (-.9,1.1) {$\partial$};
\draw [fill] (-1,1) circle;
\node [above] at (1.1,1.1) {$\ol{\Xi}_1$};
\draw [fill] (1,1) circle;
\node [above] at (2.1,2.1) {$\partial\olambda_1$};
\draw [fill] (2,2) circle;
\node [below] at (3.1,.9) {$\olambda_1$};
\draw [fill] (3,.9) circle;
\node [above] at (3.1,1.1) {$\ol{\Xi}_2$};
\draw [fill] (3,1.1) circle;
\node [below] at (6.1,1.9) {$\partial\olambda_2$};
\draw [fill] (6,1.9) circle;
\node [above] at (6.1,2.1) {$\olambda_1\ol{\Xi}_2$};
\draw [fill] (6,2.1) circle;
\node [below] at (7.1,.9) {$\olambda_2$};
\draw [fill] (7,.9) circle;
\node [above] at (7.1,1.1) {$\ol{\Xi}_3$};
\draw [fill] (7,1.1) circle;
\node [above] at (8.1,2.1) {$\ol{\Xi}_1\olambda_2$};
\draw [fill] (8,2) circle;
\node [above] at (9.1,3.1) {$\partial\olambda_1\olambda_2$};
\draw [fill] (9,3) circle;
\node [above] at (10.1,2.1) {$\olambda_1\ol{\Xi}_3$};
\draw [fill] (10,2.1) circle;
\node [below] at (10.1,1.9) {$\olambda_1\olambda_2$};
\draw [fill] (10,1.9) circle;
\node [below] at (14.1,1.9) {$\partial\olambda_3$};
\draw [fill] (14,1.9) circle;
\node [above] at (14.1,2.1) {$\olambda_2\ol{\Xi}_3$};
\draw [fill] (14,2.1) circle;
\node [above] at (15.1,1.1) {$\olambda_3$};
\draw [fill] (15,1) circle;
\node [above] at (16.1,2.1) {$\ol{\Xi}_1\lambda_3$};
\draw [fill] (16,2) circle;
\node [below] at (17.1,2.9) {$\partial\olambda_1\olambda_3$};
\draw [fill] (17,2.9) circle;
\node [below] at (18.1,1.9) {$\olambda_1\olambda_3$};
\draw [fill] (18,1.9) circle;
\node [above] at (18.1,2.1) {$\ol{\Xi}_2\olambda_3$};
\draw [fill] (18,2.1) circle;
\node [above] at (17.1,3.1) {$\ol{\Xi}_3\olambda_1\olambda_2$};
\draw [fill] (17,3.1) circle;
\node [above] at (21.1,3.1) {$\olambda_1\ol{\Xi}_2\olambda_3$};
\draw [fill] (21,3.1) circle;
\node [below] at (21.1,2.9) {$\partial\olambda_2\olambda_3$};
\draw [fill] (21,2.9) circle;
\node [above] at (22.1,2.1) {$\olambda_2\olambda_3$};
\draw [fill] (22,2) circle;
\node [above] at (23.1,3.1) {$\ol{\Xi}_1\olambda_2\olambda_3$};
\draw [fill] (23,3) circle;
\node [above] at (24.1,4.1) {$\partial\olambda_1\olambda_2\olambda_3$};
\draw [fill] (24,4) circle;
\node [above] at (25.1,3.1) {$\olambda_1\olambda_2\olambda_3$};
\draw [fill] (25,3) circle;

\end{tikzpicture}
}
\caption{The Real syntomic cohomology of  $\tmf_1(3)$ mod $(\ov_0,\ov_1,\ov_2)$. 
Each bullet $\bullet$ represents a copy of $\pi_{\star}^{C_2}\mF_2[\ov_3]$. 
The horizontal axis represents the stem (a $C_2$-representation) and the vertical axis represents the Adams weight.
}
\label{fig:motivic-tmf13}
\end{figure}

\appendix 

\section{Equivariant homotopy theory}\label{sec:equivariant}
In this section, we briefly provide our foundations and notation in the setting of equivariant homotopy theory. We begin by defining equivariant spaces and spectra in Section~\ref{sec:equivariant-spectra}. We then describe our notations for Mackey functors and equivariant homotopy groups in Section~\ref{sec:equivariant-Mackey}. We then describe our foundations for multiplicative equivariant homotopy theory in Section~\ref{sec:equivariant-multiplicative}.

\subsection{Equivariant spaces and spectra}\label{sec:equivariant-spectra}
Let $G$ be a finite group, let $\mathcal{O}_{G}$ denote the orbit $\infty$-category of $G$, and let $\rho_{G}$ denote the 
regular representation of $G$ (we simply write $\rho$ when $G=C_2$ is the cyclic group of order $2$). 
We rely on the framework from \cite{QS21b} and~\cite{BH21}  at various points in this paper. 
We therefore introduce the necessary terminology here.

\begin{defin}
A \emph{$G$-$\infty$-category~$\cC$} (resp. \emph{$G$-space~$X$}) is a cocartesian fibration~$\cC\rightarrow \cO_{G}^{\op}$ 
(resp. left fibration~$X\rightarrow \cO_{G}^{\op}$). 
A $G$-functor~$\cC\rightarrow \cD$ is a morphism over $\cO_{G}^{\op}$ that preserves cocartesian edges. 
We write $\Fun_{G}(\cC,\cD)$ for the $\infty$-category of $G$-functors. 
\end{defin}

By straightening--unstraightening~\cite{Lur09}, a $G$-space corresponds to a functor $X : \cO_{G}^{\op} \longrightarrow \Top$. 
By Elmendorf's theorem~\cite{Elm83}, the $\infty$-category of such functors is equivalent to the $\infty$-category of 
$G$-spaces. This was the original motivation for this definition of $G$-$\infty$-category.  	

\begin{defin} 
We define 
$$\Top^{G}:=\Fun(\cO_{G}^{\op},\Top)$$
to be the \emph{$\infty$-category of $G$-spaces} and let $\Top^{G}_{*} := \Fun(\cO^{\op}_G, \Top_*)$ denote the 
\emph{$\infty$-category of pointed $G$-spaces}. 
The restriction functors $\Top^H_* \to \Top^K_*$ for $K \leq H \leq G$ give rise to the $G$-$\infty$-category of $G$-spaces 
$\underline{\Top}^{G}_*$, where
\[
(\m{\Top}^{G}_*)_{G/H} =\Top_*^{H}
\,.
\] 
\end{defin}

To define $G$-spectra, recall that in~\cite[Section~9]{BH21}, Bachmann and Hoyois define a functor $\mathcal{SH}$ which 
associates a presentable, stable, symmetric monoidal $\infty$-category $\mathcal{SH}(X)$ to each profinite groupoid $X$. 

\begin{defin}
We define 
\[
\Sp^{G}:= \mathcal{SH}(BG)
\]
to be the \emph{$\infty$-category of $G$-spectra}, where $BG$ is the classifying space of $G$.
The restriction functors $\Sp^H \to \Sp^K$ for $K \leq H \leq G$ give rise to the $G$-$\infty$-category of $G$-spectra 
$\m{\Sp}^{G}$, where 
\[
(\m{\Sp}^{G})_{G/H} =\Sp^{H}\,.
\]
\end{defin}

\begin{remark}
The Guillou--May theorem~\cite{GM11} (cf. \cite[Example 9.12]{BH21} and~\cite[Example B]{Bar17}) implies that $\Sp^{G}$ 
is equivalent to (the underlying $\infty$-category of) orthogonal $G$-spectra. 
Moreover, by the remark before~\cite[Lemma 9.5]{BH21}, this is also equivalent to the notion of $G$-spectra 
as in~\cite[Definition II.2.3]{NS18}. Moreover, as in~\cite[Remark 2.2]{QS21b}, 
there is an equivalence between $\Sp^{G}$ as defined here and the colimit
\[ 
\Top_{*}^{G}\overset{\Sigma^{\rho}}{\longrightarrow} \Top_{*}^{G} 
\overset{\Sigma^{\rho}}{\longrightarrow} \Top_{*}^{G} \overset{\Sigma^{\rho}}{\longrightarrow} \dots  
\]
in $\mathrm{Pr}^{L}$, the $\infty$-category of presentable $\infty$-categories and left adjoints. 
\end{remark}

\subsection{Equivariant homotopy groups and Mackey functors}\label{sec:equivariant-Mackey}
In this section, we recall the essential structure and facts concerning the homotopy groups and homotopy 
Mackey functors of $G$-spectra. We restrict to the case $G = C_2$ throughout. 

We begin with notation for representations and their associated spheres. 

\begin{notation}
We fix the following notations. 
\begin{itemize}
    \item Let $\sigma$ denote the sign representation of $C_2$. 
    \item Let $\rho = \sigma \oplus 1$ denote the regular representation of $C_2$. 
    \item Let $RO(C_2) \cong \mathbb{Z}\{1,\sigma\}$ denote the group completion (with respect to direct sum) 
    of the ring of isomorphism classes of real orthogonal representations of $C_2$. Given $V \in RO(C_2)$, 
    let $|V|$ denote its virtual dimension. 
    \item Let $n \in RO(C_2)$ denote a trivial representation of virtual dimension $n$. 
    \item Given a real orthogonal representation $V$ of $C_2$, we write $S^V \in \Top^G_*$ for its one-point 
    compactification and $S(V) \in \Top^G$ for the unit sphere in $V$. 
\end{itemize}
\end{notation}

Using these notations, we can define equivariant homotopy groups: 

\begin{defin}
Let $X \in \Sp^{C_2}$ and let $T$ be a finite $C_2$-set. \
\begin{itemize}
    \item Let $\m{\pi}_V(X)(T) := [S^V \wedge T_+, X]^C_2$ be the abelian group of $C_2$-equivariant maps 
    from $S^V \wedge T_+$ to $X$. 
    \item Let $\pi^{C_2}_V(X) := \m{\pi}_V(X)(C_2/C_2)$ and let $\pi^e_{|V|}(X) := \m{\pi}_V(X)(C_2/e)$. 
    \item Let $\tr_e^{C_2}: \pi^e_{|V|}(X) \to \pi^{C_2}_V(X)$ and $\res^{C_2}_e: \pi^{C_2}_V(X) \to \pi^e_{|V|}(X)$ 
    denote the \emph{transfer} and \emph{restriction} defined in~\cite[Definition~2.2]{HHR14}. 
\end{itemize}
\end{defin}

The transfer and restriction maps above equip $\m{\pi}_\star(X)$ with the structure of an 
\emph{$RO(C_2)$-graded Mackey functor.} Recall that a \emph{$C_2$-semi-Mackey functor} $\m{M}$ 
consists of a pair of commutative monoids $\m{M}(C_2/C_2)$ and $\m{M}(C_2/e)$, a restriction map $\m{M}(C_2/C_2) \to \m{M}(C_2/e)$, 
a transfer map $\m{M}(C_2/e) \to \m{M}(C_2/C_2)$, and a conjugation map $\m{M}(C_2/e) \to \m{M}(C_2/e)$ satisfying 
several axioms (cf.~\cite[Def. 2.1]{Nak12}). A $C_2$-semi-Mackey functor is a \emph{Mackey functor} if $\m{M}(C_2/C_2)$ 
and $\m{M}(C_2/e)$ are abelian groups and the restriction and transfer are group homomorphisms. 
An \emph{$RO(C_2)$-graded Mackey functor} is a certain graded analogue of a 
Mackey functor; we refer the reader to \cite[Section~3.1]{HHR14} for a precise definition and additional background. 

\begin{defin}
    \
    \begin{itemize}
    \item The \emph{Burnside Mackey functor} $\m{A}$ is defined by letting $\m{A}(C_2/H) := A(H)$ 
    be the Burnside ring of $H$. 
    \item If $B$ is an abelian group, then we write $\m{B}$ for the \emph{constant Mackey functor on $B$} 
    with $\m{B}(C_2/H) = B$ for all $H \subseteq C_2$ and restrictions given by identity homomorphisms. A Mackey functor is 
    \emph{constant} if it is isomorphic to a constant Mackey functor. 
    \end{itemize}
\end{defin}

Let $\Sp_{\ge 0}$ and $\Sp_{\le 0}$ denote the connective and coconnective parts of the 
standard $t$-structure on spectra. 

\begin{prop}[{\cite[Proposition~6.1]{BGS20}}]\label{prop:homotopy-t-structure}
There is a $t$-structure on $\Sp^{G}$, which we call the \emph{homotopy $t$-structure}, with the following properties: 
\begin{enumerate}
\item $X\in \Sp^{G}_{\ge 0}$ if and only if $X^{H}\in \Sp_{\ge 0}$ for all subgroups $H$ of $G$
\item $X\in \Sp^{G}_{\le 0}$ if and only if $X^{H}\in \Sp_{\ge 0}$. 
\item $X\in \Sp^{G}_{\le 0}\cap \Sp^{G}_{\ge 0}$ if and only if $X=\upi_{0}X$. 
\end{enumerate}
\end{prop}

\subsection{Multiplicative equivariant homotopy theory}\label{sec:equivariant-multiplicative}
Let $G$ be a finite group. 
Let $\mathcal{O}_{G}^{\op}$ be the corresponding orbit category and let $\mathbb{F}_{G}$ be the category of finite $G$-sets; 
i.e. the finite coproduct completion of $\mathcal{O}_{G}^{\op}$. 
We let 
\[ 
\bF_{G}^{\slash S}:=\mathrm{Ar}(\bF_{G})\times_{\bF_{G}}\{S\}
\]
where the pullback is taken along the target map $t:\mathrm{Ar}(\bF_{G})\to \bF_{G}$ and the inclusion $\{S\}\to \bF_{G}$ 
of the discrete category with unique object $S$. 
The association  
\[ 
G/H\mapsto \mathbb{F}_{G}^{\slash (G/H)}
\]
defines a $G$-$\infty$-category $\m{\bF}_{G}$. 

\begin{defin}
We write 
\[ 
\CAlg^{G} \subset \Fun_{/\m{\bF}_{G},\mathcal{O}_G^{\op}}(\m{\bF}_{G},\m{\Sp}^{G,\otimes}) 
\]
for the \emph{$\infty$-category of $G$-$\bE_\infty$-algebras} in the sense of~\cite[ Definition~2.2.1]{NS22}. More generally, if $\mathcal{C}$ is a $G$-symmetric monoidal $G$-$\infty$-category with value $\cC_{G/H}$ at the object $G/H\in \cO_{G}^{\op}$, we write 
\[ 
\CAlg^{G}(\cC)\subset \Fun_{/\m{\bF}_{G},\mathcal{O}_G^{\op}}(\m{\bF}_{G},\cC^{\otimes}) 
\]
for the \emph{$\infty$-category of $G$-$\bE_\infty$-algebras in $\cC$} . 
We also write  $\CAlg^{G}(\cC_{G/G})$ for this infinity category when it simplifies notation. We further write $\m{\CAlg}(\cC)$ or $\m{\CAlg}(\cC_{G/G})$ for the $G$-$\infty$-category of $G$-$\bE_\infty$-algebras in $\cC$ and simply  $\underline{\CAlg}:=\underline{\CAlg}(\Sp^{G})$. 
\end{defin}

\begin{remark}\label{rem:Erhoinfty-vs-C2-Einfty}
Recent work of Lenz--Linskens--P\"utchuck~\cite[Theorem 7.27]{LLP25} has closed a gap in the literature by proving that strictly commutative monoids in $G$-spectra, which the authors call ultra-commutative $G$-ring spectra, are equivalent to normed $G$-algebras and $G$-$\bE_\infty$-algebras; see \cite[Remark~7.19]{LLP25}. 
\end{remark}

\begin{defin}
    If $G$ is a finite group, $H \subseteq G$ a subgroup, and $p: BH \to BG$ the induced map of classifying spaces. We define the \emph{Hill--Hopkins--Ravenel norm} by $N_H^G := p_\otimes: \Sp^H \to \Sp^G$ as in the paragraph before \cite[Remark~9.9]{BH21}. More generally, given a $G$-$\bE_\infty$-ring, we can define a relative norm $_{A}N_H^G : \Mod(\Sp^{H};A) \to \Mod(\Sp^G;A)$ as in~\cite{ABGHLM18}; see~\cite[Construction~A.3.4]{CHLL24} for a construction that adheres to our choice of foundations.
\end{defin}

\subsection{The parametrized Tate construction}\label{sec:equivariant-Tate}
In this section, we recall the parametrized Tate construction introduced in~\cite{GM95}. 
We follow the treatment in the setting of parametrized $\infty$-categories from~\cite{QS21b}. 

Let $\psi = [ K\to \widehat{G}\to G]$
be an extension of groups with $G$ finite. Then there is an associated fibration $B\widehat{G}\to BG$ which corresponds to a 
functor $BG\to \Cat_{\infty}$ sending the unique $0$-simplex to $BK$. 
The right Kan extension of this functor along the canonical inclusion $BG\to \cO_{G}^{\op}$ produces a functor 
\[ 
B_{G}^\psi K \in \Fun(\cO_{G}^{\op},\Top) \,,
\] 
i.e.,  $G$-space. When $\psi$ is understood, we omit it from the notation and simply write $B_G K$. 

\begin{exm}\label{exm:parametrized-CPinfty}
Let $O(2)\overset{\textup{det}}{\longrightarrow} C_2$ denote the determinant map. 
The extension of groups 
\[ 
1\longrightarrow S^1 \longrightarrow O(2) \overset{\textup{det}}{\longrightarrow} C_2 \longrightarrow 1
\]
gives rise to the $C_2$-space $B_{C_2}S^1$, which we denote by $\mathbb{C} P_{C_2}^{\infty}$. 
\end{exm}

\begin{exm}
The extension of groups 
\[ 
1\to \mu_{p^n} \to D_{2p^n}\to C_2 \to 1 \,,
\]
where $\mu_{p^n}\subset S^1$ denotes the cyclic group of order $p^n$ sitting inside $S^1$ as the $p^n$-th roots of unity, 
gives rise to the $C_2$-space $B_{C_2}\mu_{p^n}$. 
Similarly, the extension of groups
\[ 
1\to \mu_{p^\infty} \to D_{2p^{\infty}}\to C_2 \to 1\,,
\]
where $\mu_{p^{\infty}}=\colim_n\mu_{p^n}\subset S^1$, gives rise to the $C_2$-space $B_{C_2}\mu_{p^{\infty}}$.
\end{exm} 

\begin{defin}\label{Def:Twisted}
Let $\psi = [K \to \widehat{G} \to G]$ be a group extension and $B_G^\psi K$ the associated $G$-space. We define
\[
\Sp^{h_GK} := \Fun_{G}(B_G^\psi K, \m{\Sp}^G)
\]
to be the \emph{$\infty$-category of $G$-spectra with $\psi$-twisted $K$-action}. 
Since the extension $\psi$ will always be one of the extensions in the previous two examples, 
we omit $\psi$ from the notation. 
In the case where $G=C_2$, we refer to $\Sp^{h_{C_2}K}$ simply as the $\infty$-category $C_2$-spectra with twisted $K$-action. 

Given an object $X \in \Sp^{h_GK}$, we define the parametrized homotopy orbits and parametrized homotopy fixed points by
\[
X_{h_GK} := \colim_{B_G K} X \quad \text{ and } \quad  X^{h_GK} := \lim_{B_GK} X \, 
\]
respectively. We define the parametrized Tate construction by
\[
X^{t_GK} := \cof(\Nm: X_{h_GK} \to X^{h_GK})
\]
if $K$ is finite and write
\[
X^{t_{C_2}S^1} := \cof(\Nm: \mathbb{S}^\sigma \otimes X_{h_{C_2}S^1} \to X^{h_{C_2}S^1}) \,,
\]
where $\Nm$ is the parametrized assembly map of~\cite[Theorem~D]{QS21b} and $\sigma$ is the sign representation of $C_2$. 
Above, $\colim$ and $\lim$ denote the $G$-colimit and $G$-limit, so $X_{h_GK}$, $X^{h_{G}K}$, 
and $X^{t_GK}$ are all $G$-spectra.
\end{defin}

\begin{remark}
\cite[Theorem~E]{QS21b} implies that the parametrized Tate construction~$(-)^{t_GK}$ and the natural transformation~$(-)^{h_GK} \to (-)^{t_GK}$ uniquely admit the structure of a lax $G$-symmetric monoidal functor and morphism thereof. 
In particular, if $X \in \CAlg_G(\Sp^{h_GK})$, then $X^{t_GK} \in \CAlg_G$ and the natural map~$X^{h_GK} \to X^{t_GK}$ is a 
map of $G$-$\bE_\infty$-rings. 
\end{remark}

\bibliographystyle{alpha}
\bibliography{rs}

@article{AKACHR25,
 author = {Angelini-Knoll, Gabriel and Ausoni, Christian and Culver, Dominic Leon and H{\"o}ning, Eva and Rognes, John},
 title = {Algebraic {{\(K\)}}-theory of elliptic cohomology},
 fjournal = {Geometry \& Topology},
 journal = {Geom. Topol.},
 issn = {1465-3060},
 volume = {29},
 number = {2},
 pages = {619--686},
 year = {2025},
 language = {English},
 doi = {10.2140/gt.2025.29.619},
 zbMATH = {8017539}
}

@article{DDIO24,
 author = {Dugger, Daniel and Dundas, Bj{\o}rn Ian and Isaksen, Daniel C. and {\O}stv{\ae}r, Paul Arne},
 title = {The multiplicative structures on motivic homotopy groups},
 fjournal = {Algebraic \& Geometric Topology},
 journal = {Algebr. Geom. Topol.},
 issn = {1472-2747},
 volume = {24},
 number = {3},
 pages = {1781--1786},
 year = {2024},
 language = {English},
 doi = {10.2140/agt.2024.24.1781},
 zbMATH = {7901556}
}

@article {AKGH25,
    AUTHOR = {Angelini-Knoll, Gabriel and Gerhardt, Teena and Hill, Michael
              A.},
     TITLE = {Real topological {H}ochschild homology via the norm and {R}eal
              {W}itt vectors},
   JOURNAL = {Adv. Math.},
  FJOURNAL = {Advances in Mathematics},
    VOLUME = {482},
      YEAR = {2025},
     PAGES = {Paper No. 110568},
      ISSN = {0001-8708,1090-2082},
   MRCLASS = {19D55 (13F35 16E40 16W10 55P91)},
  MRNUMBER = {4969860},
       DOI = {10.1016/j.aim.2025.110568},
       URL = {https://doi.org/10.1016/j.aim.2025.110568},
}

@Article{AS18,
 Author = {Angelini-Knoll, Gabe and Salch, Andrew},
 Title = {A {May}-type spectral sequence for higher topological {Hochschild} homology},
 FJournal = {Algebraic \& Geometric Topology},
 Journal = {Algebr. Geom. Topol.},
 ISSN = {1472-2747},
 Volume = {18},
 Number = {5},
 Pages = {2593--2660},
 Year = {2018},
 Language = {English},
 DOI = {10.2140/agt.2018.18.2593},
 zbMATH = {6935816},
 Zbl = {1410.18016}
}

@ARTICLE{AKAR23,
       author = {{Angelini-Knoll}, Gabriel and {Ausoni}, Christian and {Rognes}, John},
        title = "{Algebraic K-theory of real topological K-theory}",
      journal = {arXiv e-prints},
         year = 2023,
        month = sep,
          eid = {arXiv:2309.11463},
        pages = {arXiv:2309.11463},
          doi = {10.48550/arXiv.2309.11463},
archivePrefix = {arXiv},
       eprint = {2309.11463},
 primaryClass = {math.AT},
       adsurl = {https://ui.adsabs.harvard.edu/abs/2023arXiv230911463A}
}

@ARTICLE{AKHW24,
       author = {{Angelini-Knoll}, Gabriel and {Hahn}, Jeremy and {Wilson}, Dylan},
        title = "{Syntomic cohomology of Morava K-theory}",
      journal = {arXiv e-prints},
         year = 2024,
        month = oct,
          eid = {arXiv:2410.07048},
        pages = {arXiv:2410.07048},
          doi = {10.48550/arXiv.2410.07048},
archivePrefix = {arXiv},
       eprint = {2410.07048},
 primaryClass = {math.KT},
       adsurl = {https://ui.adsabs.harvard.edu/abs/2024arXiv241007048A}
}

@Article{ABGHLM18,
 Author = {Angeltveit, Vigleik and Blumberg, Andrew J. and Gerhardt, Teena and Hill, Michael A. and Lawson, Tyler and Mandell, Michael A.},
 Title = {Topological cyclic homology via the norm},
 FJournal = {Documenta Mathematica},
 Journal = {Doc. Math.},
 ISSN = {1431-0635},
 Volume = {23},
 Pages = {2101--2163},
 Year = {2018},
 Language = {English},
 DOI = {10.25537/dm.2018v23.2101-2163},
 zbMATH = {7013683},
 Zbl = {1417.55015}
}

@Article{Ati66,
 Author = {Atiyah, Michael F.},
 Title = {K-theory and reality},
 FJournal = {The Quarterly Journal of Mathematics. Oxford Second Series},
 Journal = {Q. J. Math., Oxf. II. Ser.},
 ISSN = {0033-5606},
 Volume = {17},
 Pages = {367--386},
 Year = {1966},
 Language = {English},
 DOI = {10.1093/qmath/17.1.367},
 zbMATH = {3235829},
 Zbl = {0146.19101}
}

@article{AR08,
	author = {Ausoni, Christian and Rognes, John},
	journal = {Enseign. Math.},
	number = {2},
	pages = {9-11},
	title = {The chromatic red-shift in algebraic {K}-theory},
	volume = {54},
	year = {2008}
}

@InCollection{Boa99,
 Author = {Boardman, J. Michael},
 Title = {Conditionally convergent spectral sequences},
 BookTitle = {Homotopy invariant algebraic structures. A conference in honor of J. Michael Boardman. AMS special session on homotopy theory, Baltimore, MD, USA, January 7--10, 1998},
 ISBN = {0-8218-1057-X},
 Pages = {49--84},
 Year = {1999},
 Publisher = {Providence, RI: American Mathematical Society},
 Language = {English},
 zbMATH = {1390036},
 Zbl = {0947.55020}
}

@ARTICLE{PR,
       author = {{Pstr{\k{a}}gowski}, Piotr and {Raksit}, Arpon},
        title = "{Motivic cohomology of {$E_2$}-ring spectra}",
      journal = {In progress},
        year = {}
}

@Article{Ara79,
 Author = {Araki, Shoro},
 Title = {Orientations in {{\(\tau\)}}-cohomology theories},
 FJournal = {Japanese Journal of Mathematics. New Series},
 Journal = {Jpn. J. Math., New Ser.},
 ISSN = {0289-2316},
 Volume = {5},
 Pages = {403--430},
 Year = {1979},
 Language = {English},
 zbMATH = {3692261},
 Zbl = {0443.55003}
}

@ARTICLE{HNS,
       author = {{Harpaz}, Yonaton and {Nikolaus}, Thoomas and {Shah}, Jay},
        title = "{Normal L-theory and topological cyclic homology}",
      journal = {In progress},
        year = {}
}

@article{AKCH21,
 author = {Angelini-Knoll, Gabriel and Culver, Dominic Leon and H{\"o}ning, Eva},
 title = {Topological {Hochschild} homology of truncated {Brown}-{Peterson} spectra. {I}.},
 fjournal = {Algebraic \& Geometric Topology},
 journal = {Algebr. Geom. Topol.},
 issn = {1472-2747},
 volume = {24},
 number = {5},
 pages = {2509--2536},
 year = {2024},
 language = {English},
 doi = {10.2140/agt.2024.24.2509},
 zbMATH = {7922360},
 Zbl = {1560.55015}
}

@Article{AMN22,
 Author = {Antieau, Benjamin and Mathew, Akhil and Morrow, Matthew and Nikolaus, Thomas},
 Title = {On the {Beilinson} fiber square},
 FJournal = {Duke Mathematical Journal},
 Journal = {Duke Math. J.},
 ISSN = {0012-7094},
 Volume = {171},
 Number = {18},
 Pages = {3707--3806},
 Year = {2022},
 Language = {English},
 DOI = {10.1215/00127094-2022-0037},
 zbMATH = {7628867},
 Zbl = {1508.14017}
}

@Article{AR02,
 Author = {Ausoni, Christian and Rognes, John},
 Title = {Algebraic {{\(K\)}}-theory of topological {{\(K\)}}-theory},
 FJournal = {Acta Mathematica},
 Journal = {Acta Math.},
 ISSN = {0001-5962},
 Volume = {188},
 Number = {1},
 Pages = {1--39},
 Year = {2002},
 Language = {English},
 DOI = {10.1007/BF02392794},
 zbMATH = {1786645},
 Zbl = {1019.18008}
}

@Article{AR12,
 Author = {Ausoni, Christian and Rognes, John},
 Title = {Algebraic {{\(K\)}}-theory of the first Morava {{\(K\)}}-theory},
 FJournal = {Journal of the European Mathematical Society (JEMS)},
 Journal = {J. Eur. Math. Soc. (JEMS)},
 ISSN = {1435-9855},
 Volume = {14},
 Number = {4},
 Pages = {1041--1079},
 Year = {2012},
 Language = {English},
 DOI = {10.4171/JEMS/326},
 zbMATH = {6058742},
 Zbl = {1253.19001}
}

@Book{BH21,
 Author = {Bachmann, Tom and Hoyois, Marc},
 Title = {Norms in motivic homotopy theory},
 FSeries = {Ast{\'e}risque},
 Series = {Ast{\'e}risque},
 ISSN = {0303-1179},
 Volume = {425},
 ISBN = {978-2-85629-939-5},
 Year = {2021},
 Publisher = {Paris: Soci{\'e}t{\'e} Math{\'e}matique de France (SMF)},
 Language = {English},
 DOI = {10.24033/ast.1147},
 zbMATH = {7403459}
}

@Article{BGS20,
 Author = {Barwick, Clark and Glasman, Saul and Shah, Jay},
 Title = {Spectral {Mackey} functors and equivariant algebraic {{\(K\)}}-theory. {II}.},
 FJournal = {Tunisian Journal of Mathematics},
 Journal = {Tunis. J. Math.},
 ISSN = {2576-7658},
 Volume = {2},
 Number = {1},
 Pages = {97--146},
 Year = {2020},
 Language = {English},
 DOI = {10.2140/tunis.2020.2.97},
 zbMATH = {7074072},
 Zbl = {1461.18009}
}

@ARTICLE{BL14,
       author = {{Barwick}, Clark and {Lawson}, Tyler},
        title = "{Regularity of structured ring spectra and localization in K-theory}",
      journal = {arXiv e-prints},
         year = 2014,
        month = feb,
          eid = {arXiv:1402.6038},
        pages = {arXiv:1402.6038},
          doi = {10.48550/arXiv.1402.6038},
archivePrefix = {arXiv},
       eprint = {1402.6038},
 primaryClass = {math.KT},
       adsurl = {https://ui.adsabs.harvard.edu/abs/2014arXiv1402.6038B}
}

@Article{BGL22,
 Author = {Bhattacharya, Prasit and Guillou, Bertrand and Li, Ang},
 Title = {An {{\(R\)}}-motivic {{\(v_1\)}}-self-map of periodicity {{\(1\)}}},
 FJournal = {Homology, Homotopy and Applications},
 Journal = {Homology Homotopy Appl.},
 ISSN = {1532-0073},
 Volume = {24},
 Number = {1},
 Pages = {299--324},
 Year = {2022},
 Language = {English},
 DOI = {10.4310/HHA.2022.v24.n1.a15},
 zbMATH = {7536036},
 Zbl = {1509.55010}
}

@ARTICLE{BW17,
       author = {{Behrens}, Mark and {Wilson}, Dylan},
        title = "{A $C_2$-equivariant analog of Mahowald's Thom spectrum theorem}",
      journal = {arXiv e-prints},
         year = 2017,
        month = jul,
          eid = {arXiv:1707.02582},
        pages = {arXiv:1707.02582},
          doi = {10.48550/arXiv.1707.02582},
archivePrefix = {arXiv},
       eprint = {1707.02582},
 primaryClass = {math.AT},
       adsurl = {https://ui.adsabs.harvard.edu/abs/2017arXiv170702582B}
}

@article {BKSO15,
    AUTHOR = {Berrick, A. J. and Karoubi, M. and Schlichting, M. and \O
              stv\ae r, P. A.},
     TITLE = {The {H}omotopy {F}ixed {P}oint {T}heorem and the
              {Q}uillen-{L}ichtenbaum conjecture in {H}ermitian
              {$K$}-theory},
   JOURNAL = {Adv. Math.},
  FJOURNAL = {Advances in Mathematics},
    VOLUME = {278},
      YEAR = {2015},
     PAGES = {34--55},
      ISSN = {0001-8708,1090-2082},
   MRCLASS = {19G38},
  MRNUMBER = {3341783},
MRREVIEWER = {Mohamed\ Elhamdadi},
       DOI = {10.1016/j.aim.2015.01.018},
       URL = {https://doi.org/10.1016/j.aim.2015.01.018},
}

@InCollection{BM94,
 Author = {B{\"o}kstedt, M. and Madsen, I.},
 Title = {Topological cyclic homology of the integers},
 BookTitle = {\(K\)-theory. Contributions of the international colloquium, Strasbourg, France, June 29-July 3, 1992},
 Pages = {57--143},
 Year = {1994},
 Publisher = {Paris: Soci{\'e}t{\'e} Math{\'e}matique de France},
 Language = {English},
 zbMATH = {742911},
 Zbl = {0816.19001}
}

@Article{Bar17,
 Author = {Barwick, Clark},
 Title = {Spectral {Mackey} functors and equivariant algebraic {{\(K\)}}-theory. {I}.},
 FJournal = {Advances in Mathematics},
 Journal = {Adv. Math.},
 ISSN = {0001-8708},
 Volume = {304},
 Pages = {646--727},
 Year = {2017},
 Language = {English},
 DOI = {10.1016/j.aim.2016.08.043},
 zbMATH = {6642269},
 Zbl = {1348.18020}
}

@Article{BMS19,
 Author = {Bhatt, Bhargav and Morrow, Matthew and Scholze, Peter},
 Title = {Topological {Hochschild} homology and integral {{\(p\)}}-adic {Hodge} theory},
 FJournal = {Publications Math{\'e}matiques},
 Journal = {Publ. Math., Inst. Hautes {\'E}tud. Sci.},
 ISSN = {0073-8301},
 Volume = {129},
 Pages = {199--310},
 Year = {2019},
 Language = {English},
 DOI = {10.1007/s10240-019-00106-9},
 zbMATH = {7059677},
 Zbl = {1478.14039}
}

@misc{BHS22,
      title={Galois reconstruction of Artin-Tate $\mathbb{R}$-motivic spectra}, 
      author={Robert Burklund and Jeremy Hahn and Andrew Senger},
      year={2022},
      eprint={2010.10325},
      archivePrefix={arXiv},
      primaryClass={math.AT}
}

@ARTICLE{CHLL24,
       author = {{Cnossen}, Bastiaan and {Haugseng}, Rune and {Lenz}, Tobias and {Linskens}, Sil},
        title = "{Normed equivariant ring spectra and higher Tambara functors}",
      journal = {arXiv e-prints},
         year = 2024,
        month = jul,
          eid = {arXiv:2407.08399},
        pages = {arXiv:2407.08399},
          doi = {10.48550/arXiv.2407.08399},
archivePrefix = {arXiv},
       eprint = {2407.08399},
 primaryClass = {math.AT},
       adsurl = {https://ui.adsabs.harvard.edu/abs/2024arXiv240708399C}
}

@Article{DS75,
 Author = {Dennis, R. Keith and Stein, Michael R.},
 Title = {{{\(K_2\)}} of discrete valuation rings},
 FJournal = {Advances in Mathematics},
 Journal = {Adv. Math.},
 ISSN = {0001-8708},
 Volume = {18},
 Pages = {182--238},
 Year = {1975},
 Language = {English},
 DOI = {10.1016/0001-8708(75)90157-7},
 zbMATH = {3498043},
 Zbl = {0318.13017}
}

@Article{DMP24,
 Author = {Dotto, Emanuele and Moi, Kristian and Patchkoria, Irakli},
 Title = {On the geometric fixed points of real topological cyclic homology},
 FJournal = {Journal of the London Mathematical Society. Second Series},
 Journal = {J. Lond. Math. Soc., II. Ser.},
 ISSN = {0024-6107},
 Volume = {109},
 Number = {2},
 Pages = {68},
 Note = {Id/No e12862},
 Year = {2024},
 Language = {English},
 DOI = {10.1112/jlms.12862},
 zbMATH = {7811233}
}

@Article{DMPR21,
 Author = {Dotto, Emanuele and Moi, Kristian and Patchkoria, Irakli and Reeh, Sune Precht},
 Title = {Real topological {Hochschild} homology},
 FJournal = {Journal of the European Mathematical Society (JEMS)},
 Journal = {J. Eur. Math. Soc. (JEMS)},
 ISSN = {1435-9855},
 Volume = {23},
 Number = {1},
 Pages = {63--152},
 Year = {2021},
 Language = {English},
 DOI = {10.4171/JEMS/1007},
 zbMATH = {7328108},
 Zbl = {1473.16005}
}

@Book{DGM13,
 Author = {Dundas, Bj{\o}rn Ian and Goodwillie, Thomas G. and McCarthy, Randy},
 Title = {The local structure of algebraic {K}-theory},
 FSeries = {Algebra and Applications},
 Series = {Algebr. Appl.},
 ISSN = {1572-5553},
 Volume = {18},
 ISBN = {978-1-4471-4392-5; 978-1-4471-4393-2},
 Year = {2013},
 Publisher = {London: Springer},
 Language = {English},
 DOI = {10.1007/978-1-4471-4393-2},
 zbMATH = {6061719},
 Zbl = {1272.55002}
}

@Article{ES21,
 Author = {Elmanto, Elden and Shah, Jay},
 Title = {Scheiderer motives and equivariant higher topos theory},
 FJournal = {Advances in Mathematics},
 Journal = {Adv. Math.},
 ISSN = {0001-8708},
 Volume = {382},
 Pages = {116},
 Note = {Id/No 107651},
 Year = {2021},
 Language = {English},
 DOI = {10.1016/j.aim.2021.107651},
 zbMATH = {7337355},
 Zbl = {1484.14046}
}

@Article{Elm83,
 Author = {Elmendorf, A. D.},
 Title = {Systems of fixed point sets},
 FJournal = {Transactions of the American Mathematical Society},
 Journal = {Trans. Am. Math. Soc.},
 ISSN = {0002-9947},
 Volume = {277},
 Pages = {275--284},
 Year = {1983},
 Language = {English},
 DOI = {10.2307/1999356},
 zbMATH = {3824751},
 Zbl = {0521.57027}
}

@Article{EM06,
 Author = {Elmendorf, A. D. and Mandell, M. A.},
 Title = {Rings, modules, and algebras in infinite loop space theory},
 FJournal = {Advances in Mathematics},
 Journal = {Adv. Math.},
 ISSN = {0001-8708},
 Volume = {205},
 Number = {1},
 Pages = {163--228},
 Year = {2006},
 Language = {English},
 DOI = {10.1016/j.aim.2005.07.007},
 zbMATH = {5054351},
 Zbl = {1117.19001}
}

@Misc{FM87,
 Author = {Fontaine, Jean-Marc and Messing, William},
 Title = {{{\(p\)}}-adic periods and {{\(p\)}}-adic {\'e}tale cohomology},
 Year = {1987},
 Language = {English},
 HowPublished = {Current trends in arithmetical algebraic geometry, {Proc}. {Summer} {Res}. {Conf}., {Arcata}/{Calif}. 1985, {Contemp}. {Math}. 67, 179-207 (1987).},
 zbMATH = {4027649},
 Zbl = {0632.14016}
}

@Article{GWX21,
 Author = {Gheorghe, Bogdan and Wang, Guozhen and Xu, Zhouli},
 Title = {The special fiber of the motivic deformation of the stable homotopy category is algebraic},
 FJournal = {Acta Mathematica},
 Journal = {Acta Math.},
 ISSN = {0001-5962},
 Volume = {226},
 Number = {2},
 Pages = {319--407},
 Year = {2021},
 Language = {English},
 DOI = {10.4310/ACTA.2021.v226.n2.a2},
 zbMATH = {7378147},
 Zbl = {1478.55006}
}

@Article{GIKR22,
 Author = {Gheorghe, Bogdan and Isaksen, Daniel C. and Krause, Achim and Ricka, Nicolas},
 Title = {{{\(\mathbb{C}\)}}-motivic modular forms},
 FJournal = {Journal of the European Mathematical Society (JEMS)},
 Journal = {J. Eur. Math. Soc. (JEMS)},
 ISSN = {1435-9855},
 Volume = {24},
 Number = {10},
 Pages = {3597--3628},
 Year = {2022},
 Language = {English},
 DOI = {10.4171/JEMS/1171},
 zbMATH = {7547822},
 Zbl = {1498.14050}
}

@InCollection{Gre18,
 Author = {Greenlees, J. P. C.},
 Title = {Four approaches to cohomology theories with reality},
 BookTitle = {An alpine bouquet of algebraic topology. Alpine algebraic and applied topology conference, Saas-Almagell, Switzerland, August 15--21, 2016. Proceedings},
 ISBN = {978-1-4704-2911-9; 978-1-4704-4774-8},
 Pages = {139--156},
 Year = {2018},
 Publisher = {Providence, RI: American Mathematical Society (AMS)},
 Language = {English},
 DOI = {10.1090/conm/708/14261},
 zbMATH = {6950985},
 Zbl = {1412.55011}
}

@Book{GM95,
 Author = {Greenlees, J. P. C. and May, J. P.},
 Title = {Generalized {Tate} cohomology},
 FSeries = {Memoirs of the American Mathematical Society},
 Series = {Mem. Am. Math. Soc.},
 ISSN = {0065-9266},
 Volume = {543},
 ISBN = {978-0-8218-2603-4; 978-1-4704-0122-1},
 Year = {1995},
 Publisher = {Providence, RI: American Mathematical Society (AMS)},
 Language = {English},
 DOI = {10.1090/memo/0543},
 zbMATH = {734600},
 Zbl = {0876.55003}
}

@ARTICLE{GM11,
       author = {{Guillou}, Bertrand and {May}, J.~P.},
        title = "{Models of G-spectra as presheaves of spectra}",
      journal = {arXiv e-prints},
         year = 2011,
        month = oct,
          eid = {arXiv:1110.3571},
        pages = {arXiv:1110.3571},
          doi = {10.48550/arXiv.1110.3571},
archivePrefix = {arXiv},
       eprint = {1110.3571},
 primaryClass = {math.AT},
       adsurl = {https://ui.adsabs.harvard.edu/abs/2011arXiv1110.3571G}
}

@Article{GM17,
 Author = {Guillou, Bertrand J. and May, J. Peter},
 Title = {Equivariant iterated loop space theory and permutative {{\(G\)}}-categories},
 FJournal = {Algebraic \& Geometric Topology},
 Journal = {Algebr. Geom. Topol.},
 ISSN = {1472-2747},
 Volume = {17},
 Number = {6},
 Pages = {3259--3339},
 Year = {2017},
 Language = {English},
 DOI = {10.2140/agt.2017.17.3259},
 zbMATH = {6791649},
 Zbl = {1394.55008}
}

@Article{HW20,
 Author = {Hahn, Jeremy and Wilson, Dylan},
 Title = {Eilenberg-{Mac} {Lane} spectra as equivariant {Thom} spectra},
 FJournal = {Geometry \& Topology},
 Journal = {Geom. Topol.},
 ISSN = {1465-3060},
 Volume = {24},
 Number = {6},
 Pages = {2709--2748},
 Year = {2020},
 Language = {English},
 DOI = {10.2140/gt.2020.24.2709},
 zbMATH = {7305778},
 Zbl = {1459.55008}
}

@ARTICLE{HRW22,
       author = {{Hahn}, Jeremy and {Raksit}, Arpon and {Wilson}, Dylan},
        title = "{A motivic filtration on the topological cyclic homology of commutative ring spectra}",
      journal = {To appear in Ann. Math.},
        year = {2025},
}

@Article{HLN21,
 Author = {Hebestreit, Fabian and Land, Markus and Nikolaus, Thomas},
 Title = {On the homotopy type of {{\(\mathrm{L}\)}}-spectra of the integers},
 FJournal = {Journal of Topology},
 Journal = {J. Topol.},
 ISSN = {1753-8416},
 Volume = {14},
 Number = {1},
 Pages = {183--214},
 Year = {2021},
 Language = {English},
 DOI = {10.1112/topo.12180},
 zbMATH = {7367120},
 Zbl = {1484.19008}
}

@Article{HK01,
 Author = {Hu, Po and Kriz, Igor},
 Title = {Real-oriented homotopy theory and an analogue of the {Adams}-{Novikov} spectral sequence},
 FJournal = {Topology},
 Journal = {Topology},
 ISSN = {0040-9383},
 Volume = {40},
 Number = {2},
 Pages = {317--399},
 Year = {2001},
 Language = {English},
 DOI = {10.1016/S0040-9383(99)00065-8},
 zbMATH = {1587571},
 Zbl = {0967.55010}
}

@Article{Hil22,
 Author = {Hill, Michael A.},
 Title = {Freeness and equivariant stable homotopy},
 FJournal = {Journal of Topology},
 Journal = {J. Topol.},
 ISSN = {1753-8416},
 Volume = {15},
 Number = {2},
 Pages = {359--397},
 Year = {2022},
 Language = {English},
 DOI = {10.1112/topo.12227},
 zbMATH = {7738209},
 Zbl = {1530.55007}
}

@Article{HM97,
 Author = {Hesselholt, Lars and Madsen, Ib},
 Title = {On the {{\(K\)}}-theory of finite algebras over {Witt} vectors of perfect fields},
 FJournal = {Topology},
 Journal = {Topology},
 ISSN = {0040-9383},
 Volume = {36},
 Number = {1},
 Pages = {29--101},
 Year = {1997},
 Language = {English},
 DOI = {10.1016/0040-9383(96)00003-1},
 zbMATH = {938856},
 Zbl = {0866.55002}
}

@ARTICLE{Bru01,
 Author = {Brun, Morten},
 Title = {Filtered topological cyclic homology and relative {\(K\)}-theory of nilpotent ideals},
 FJournal = {Algebraic \& Geometric Topology},
 Journal = {Algebr. Geom. Topol.},
 ISSN = {1472-2747},
 Volume = {1},
 Pages = {201--230},
 Year = {2001},
 Language = {English},
 DOI = {10.2140/agt.2001.1.201},
 zbMATH = {1589378},
 Zbl = {0984.19001}
}

@Article{HW22,
 Author = {Hahn, Jeremy and Wilson, Dylan},
 Title = {Redshift and multiplication for truncated {Brown}-{Peterson} spectra},
 FJournal = {Annals of Mathematics. Second Series},
 Journal = {Ann. Math. (2)},
 ISSN = {0003-486X},
 Volume = {196},
 Number = {3},
 Pages = {1277--1351},
 Year = {2022},
 Language = {English},
 DOI = {10.4007/annals.2022.196.3.6},
 zbMATH = {7611907}
}

@ARTICLE{HH18,
       author = {{Hill}, Michael A. and {Hopkins}, Michael J.},
        title = "{Real Wilson Spaces I}",
      journal = {arXiv e-prints},
         year = 2018,
        month = jun,
          eid = {arXiv:1806.11033},
        pages = {arXiv:1806.11033},
          doi = {10.48550/arXiv.1806.11033},
archivePrefix = {arXiv},
       eprint = {1806.11033},
 primaryClass = {math.AT},
       adsurl = {https://ui.adsabs.harvard.edu/abs/2018arXiv180611033H}
}

@InCollection{HHR14,
 Author = {Hill, Michael A. and Hopkins, Michael J. and Ravenel, Douglas C.},
 Title = {On the non-existence of elements of {Kervaire} invariant one},
 BookTitle = {Proceedings of the International Congress of Mathematicians (ICM 2014), Seoul, Korea, August 13--21, 2014. Vol. II: Invited lectures},
 ISBN = {978-89-6105-805-6; 978-89-6105-803-2},
 Pages = {1219--1243},
 Year = {2014},
 Publisher = {Seoul: KM Kyung Moon Sa},
 Language = {English},
 zbMATH = {6797703},
 Zbl = {1373.55023}
}

@Article{HM17,
 Author = {Hill, Michael A. and Meier, Lennart},
 Title = {The {{\(C_2\)}}-spectrum {{\(\mathrm{Tmf}_1(3)\)}} and its invertible modules},
 FJournal = {Algebraic \& Geometric Topology},
 Journal = {Algebr. Geom. Topol.},
 ISSN = {1472-2747},
 Volume = {17},
 Number = {4},
 Pages = {1953--2011},
 Year = {2017},
 Language = {English},
 DOI = {10.2140/agt.2017.17.1953},
 zbMATH = {6762682},
 Zbl = {1421.55002}
}

@Article{Hil12,
 Author = {Hill, Michael A.},
 Title = {The equivariant slice filtration: a primer},
 FJournal = {Homology, Homotopy and Applications},
 Journal = {Homology Homotopy Appl.},
 ISSN = {1532-0073},
 Volume = {14},
 Number = {2},
 Pages = {143--166},
 Year = {2012},
 Language = {English},
 DOI = {10.4310/HHA.2012.v14.n2.a9},
 zbMATH = {6132932},
 Zbl = {1403.55003}
}

@Article{HL18,
 Author = {Hopkins, Michael J. and Lawson, Tyler},
 Title = {Strictly commutative complex orientation theory},
 FJournal = {Mathematische Zeitschrift},
 Journal = {Math. Z.},
 ISSN = {0025-5874},
 Volume = {290},
 Number = {1-2},
 Pages = {83--101},
 Year = {2018},
 Language = {English},
 DOI = {10.1007/s00209-017-2009-6},
 zbMATH = {7031329},
 Zbl = {1417.55012}
}

@InCollection{HO18,
 Author = {Heller, J. and Ormsby, K.},
 Title = {The stable {Galois} correspondence for real closed fields},
 BookTitle = {New directions in homotopy theory. Second Mid-Atlantic Topology Conference, Johns Hopkins University, Baltimore, MD, USA, March 12--13, 2016. Proceedings},
 ISBN = {978-1-4704-3774-9; 978-1-4704-4772-4},
 Pages = {1--9},
 Year = {2018},
 Publisher = {Providence, RI: American Mathematical Society (AMS)},
 Language = {English},
 DOI = {10.1090/conm/707/14250},
 zbMATH = {6949787},
 Zbl = {1398.14030}
}

@ARTICLE{HP23,
       author = {{Hornbostel}, Jens and {Park}, Doosung},
        title = "{Real Topological Hochschild Homology of Perfectoid Rings}",
      journal = {arXiv e-prints},
         year = 2023,
        month = oct,
          eid = {arXiv:2310.11183},
        pages = {arXiv:2310.11183},
          doi = {10.48550/arXiv.2310.11183},
archivePrefix = {arXiv},
       eprint = {2310.11183},
 primaryClass = {math.KT},
       adsurl = {https://ui.adsabs.harvard.edu/abs/2023arXiv231011183H}
}

@ARTICLE{HHKWZ20,
       author = {{Hahn}, Jeremy and {Horev}, Asaf and {Klang}, Inbar and {Wilson}, Dylan and {Zou}, Foling},
        title = "{Equivariant nonabelian Poincar{\'e} duality and equivariant factorization homology of Thom spectra}",
      journal = {arXiv e-prints},
         year = 2020,
        month = jun,
          eid = {arXiv:2006.13348},
        pages = {arXiv:2006.13348},
          doi = {10.48550/arXiv.2006.13348},
archivePrefix = {arXiv},
       eprint = {2006.13348},
 primaryClass = {math.AT},
       adsurl = {https://ui.adsabs.harvard.edu/abs/2020arXiv200613348H}
}

@Article{Joy02,
 Author = {Joyal, A.},
 Title = {Quasi-categories and {Kan} complexes},
 FJournal = {Journal of Pure and Applied Algebra},
 Journal = {J. Pure Appl. Algebra},
 ISSN = {0022-4049},
 Volume = {175},
 Number = {1-3},
 Pages = {207--222},
 Year = {2002},
 Language = {English},
 DOI = {10.1016/S0022-4049(02)00135-4},
 zbMATH = {1839068},
 Zbl = {1015.18008}
}

@article{KV71,
 author = {Karoubi, M. and Villamayor, O.},
 title = {K-th{\'e}orie hermitienne},
 fjournal = {Comptes Rendus Hebdomadaires des S{\'e}ances de l'Acad{\'e}mie des Sciences, S{\'e}rie A},
 journal = {C. R. Acad. Sci., Paris, S{\'e}r. A},
 issn = {0366-6034},
 volume = {272},
 pages = {1237--1240},
 year = {1971},
 language = {French},
 zbMATH = {3342112},
 Zbl = {0215.36902}
}

@Misc{Kat87,
 Author = {Kato, Kazuya},
 Title = {On p-adic vanishing cycles. ({Application} of ideas of {Fontaine}-{Messing})},
 Year = {1987},
 Language = {English},
 HowPublished = {Algebraic geometry, {Proc}. {Symp}., {Sendai}/{Jap}. 1985, {Adv}. {Stud}. {Pure} {Math}. 10, 207-251 (1987).},
 zbMATH = {4051778},
 Zbl = {0645.14009}
}

@Article{Kee25,
 Author = {Keenan, Liam},
 Title = {The {May} filtration on {THH} and faithfully flat descent},
 FJournal = {Journal of Pure and Applied Algebra},
 Journal = {J. Pure Appl. Algebra},
 ISSN = {0022-4049},
 Volume = {229},
 Number = {1},
 Pages = {30},
 Note = {Id/No 107806},
 Year = {2025},
 Language = {English},
 DOI = {10.1016/j.jpaa.2024.107806},
 zbMATH = {7975204}
}

@article {KRO20,
    AUTHOR = {Kylling, Jonas Irgens and R\"ondigs, Oliver and \O stv\ae r,
              Paul Arne},
     TITLE = {Hermitian {$K$}-theory, {D}edekind {$\zeta$}-functions, and
              quadratic forms over rings of integers in number fields},
   JOURNAL = {Camb. J. Math.},
  FJOURNAL = {Cambridge Journal of Mathematics},
    VOLUME = {8},
      YEAR = {2020},
    NUMBER = {3},
     PAGES = {505--607},
      ISSN = {2168-0930,2168-0949},
   MRCLASS = {19F27 (11R42 11R70 14F42 19E15)},
  MRNUMBER = {4192569},
MRREVIEWER = {Lloyd\ D.\ Simons},
       DOI = {10.4310/CJM.2020.v8.n3.a3},
       URL = {https://doi.org/10.4310/CJM.2020.v8.n3.a3},
}

@ARTICLE{LL23,
       author = {{Jongwon Lee}, David and {Levy}, Ishan},
        title = "{Topological Hochschild homology of the image of j}",
      journal = {arXiv e-prints},
         year = 2023,
        month = jul,
          eid = {arXiv:2307.04248},
        pages = {arXiv:2307.04248},
          doi = {10.48550/arXiv.2307.04248},
archivePrefix = {arXiv},
       eprint = {2307.04248},
 primaryClass = {math.AT},
       adsurl = {https://ui.adsabs.harvard.edu/abs/2023arXiv230704248J}
}

@ARTICLE{Lan22,
       author = {{Land}, Markus},
        title = "{Remarks on chromatically localised hermitian K-theory}",
      journal = {arXiv e-prints},
         year = 2022,
        month = feb,
          eid = {arXiv:2202.05036},
        pages = {arXiv:2202.05036},
          doi = {10.48550/arXiv.2202.05036},
archivePrefix = {arXiv},
       eprint = {2202.05036},
 primaryClass = {math.KT},
       adsurl = {https://ui.adsabs.harvard.edu/abs/2022arXiv220205036L}
}

@ARTICLE{LLP25,
       author = {{Lenz}, Tobias and {Linskens}, Sil and {P{\"u}tzst{\"u}ck}, Phil},
        title = "{Norms in equivariant homotopy theory}",
      journal = {arXiv e-prints},
         year = 2025,
        month = mar,
          eid = {arXiv:2503.02839},
        pages = {arXiv:2503.02839},
          doi = {10.48550/arXiv.2503.02839},
archivePrefix = {arXiv},
       eprint = {2503.02839},
 primaryClass = {math.AT},
       adsurl = {https://ui.adsabs.harvard.edu/abs/2025arXiv250302839L}
}

@ARTICLE{Lev25,
       author = {{Levin}, Jordan},
        title = "{Chromatic Purity in Hermitian K-Theory at $p=2$}",
      journal = {arXiv e-prints},
         year = 2025,
        month = jan,
          eid = {arXiv:2501.09633},
        pages = {arXiv:2501.09633},
          doi = {10.48550/arXiv.2501.09633},
archivePrefix = {arXiv},
       eprint = {2501.09633},
 primaryClass = {math.KT},
       adsurl = {https://ui.adsabs.harvard.edu/abs/2025arXiv250109633L}
}

@Article{LM07,
 Author = {Lewis, L. Gaunce jun. and Mandell, Michael A.},
 Title = {Modules in monoidal model categories},
 FJournal = {Journal of Pure and Applied Algebra},
 Journal = {J. Pure Appl. Algebra},
 ISSN = {0022-4049},
 Volume = {210},
 Number = {2},
 Pages = {395--421},
 Year = {2007},
 Language = {English},
 DOI = {10.1016/j.jpaa.2006.10.002},
 zbMATH = {5167551},
 Zbl = {1123.18010}
}

@Book{Lur09,
 Author = {Lurie, Jacob},
 Title = {Higher topos theory},
 FSeries = {Annals of Mathematics Studies},
 Series = {Ann. Math. Stud.},
 Volume = {170},
 ISBN = {978-0-691-14049-0; 978-0-691-14048-3},
 Year = {2009},
 Publisher = {Princeton, NJ: Princeton University Press},
 Language = {English},
 DOI = {10.1515/9781400830558},
 zbMATH = {5497319},
 Zbl = {1175.18001}
}

@Article{HS20,
 Author = {Hahn, Jeremy and Shi, XiaoLin Danny},
 Title = {Real orientations of {Lubin}-{Tate} spectra},
 FJournal = {Inventiones Mathematicae},
 Journal = {Invent. Math.},
 ISSN = {0020-9910},
 Volume = {221},
 Number = {3},
 Pages = {731--776},
 Year = {2020},
 Language = {English},
 DOI = {10.1007/s00222-020-00960-z},
 zbMATH = {7233318},
 Zbl = {1447.55011}
}

@article{KSW16,
 author = {Karoubi, Max and Schlichting, Marco and Weibel, Charles},
 title = {The {Witt} group of real algebraic varieties},
 fjournal = {Journal of Topology},
 journal = {J. Topol.},
 issn = {1753-8416},
 volume = {9},
 number = {4},
 pages = {1257--1302},
 year = {2016},
 language = {English},
 doi = {10.1112/jtopol/jtw024},
 zbMATH = {6761414},
 Zbl = {1377.19001}
}

@Article{KLW17,
 Author = {Kitchloo, Nitu and Lorman, Vitaly and Wilson, W. Stephen},
 Title = {Landweber flat real pairs and {{\(ER(n)\)}}-cohomology},
 FJournal = {Advances in Mathematics},
 Journal = {Adv. Math.},
 ISSN = {0001-8708},
 Volume = {322},
 Pages = {60--82},
 Year = {2017},
 Language = {English},
 DOI = {10.1016/j.aim.2017.10.003},
 zbMATH = {6806879},
 Zbl = {1385.55004}
}

@article{LLQ22,
 author = {Li, Guchuan and Lorman, Vitaly and Quigley, J. D.},
 title = {Tate blueshift and vanishing for real oriented cohomology theories},
 fjournal = {Advances in Mathematics},
 journal = {Adv. Math.},
 issn = {0001-8708},
 volume = {411},
 pages = {51},
 note = {Id/No 108780},
 year = {2022},
 language = {English},
 doi = {10.1016/j.aim.2022.108780},
 zbMATH = {7629538},
 Zbl = {1528.55003}
}

@misc{Lur17,
    title={Higher Algebra},
    author={Lurie, Jacob},
    year={2017},
    url={https://www.math.ias.edu/~lurie/papers/HA.pdf}
}

@misc{Lur18,
    title={Spectral Algebraic Geometry},
    author={Lurie, Jacob},
    year={2018},
    url={http://www.math.ias.edu/~lurie/}
}

@Article{MNN17,
 Author = {Mathew, Akhil and Naumann, Niko and Noel, Justin},
 Title = {Nilpotence and descent in equivariant stable homotopy theory},
 FJournal = {Advances in Mathematics},
 Journal = {Adv. Math.},
 ISSN = {0001-8708},
 Volume = {305},
 Pages = {994--1084},
 Year = {2017},
 Language = {English},
 DOI = {10.1016/j.aim.2016.09.027},
 zbMATH = {6652659},
 Zbl = {1420.55024}
}

@Article{MV99,
 Author = {Morel, Fabien and Voevodsky, Vladimir},
 Title = {{{\(\mathbb{A}^1\)}}-homotopy theory of schemes},
 FJournal = {Publications Math{\'e}matiques},
 Journal = {Publ. Math., Inst. Hautes {\'E}tud. Sci.},
 ISSN = {0073-8301},
 Volume = {90},
 Pages = {45--143},
 Year = {1999},
 Language = {English},
 DOI = {10.1007/BF02698831},
 URL = {https://eudml.org/doc/104163},
 zbMATH = {1596160},
 Zbl = {0983.14007}
}

@ARTICLE{NS22,
       author = {{Nardin}, Denis and {Shah}, Jay},
        title = "{Parametrized and equivariant higher algebra}",
      journal = {arXiv e-prints},
         year = 2022,
        month = feb,
          eid = {arXiv:2203.00072},
        pages = {arXiv:2203.00072},
          doi = {10.48550/arXiv.2203.00072},
archivePrefix = {arXiv},
       eprint = {2203.00072},
 primaryClass = {math.AT},
       adsurl = {https://ui.adsabs.harvard.edu/abs/2022arXiv220300072N}
}

@Article{NS18,
 Author = {Nikolaus, Thomas and Scholze, Peter},
 Title = {On topological cyclic homology},
 FJournal = {Acta Mathematica},
 Journal = {Acta Math.},
 ISSN = {0001-5962},
 Volume = {221},
 Number = {2},
 Pages = {203--409},
 Year = {2018},
 Language = {English},
 DOI = {10.4310/ACTA.2018.v221.n2.a1},
 zbMATH = {7009201},
 Zbl = {1457.19007}
}

@ARTICLE{Hog16,
       author = {{H{\o}genhaven}, Amalie},
        title = "{Real topological cyclic homology of spherical group rings}",
      journal = {arXiv e-prints},
         year = 2016,
        month = nov,
          eid = {arXiv:1611.01204},
        pages = {arXiv:1611.01204},
          doi = {10.48550/arXiv.1611.01204},
archivePrefix = {arXiv},
       eprint = {1611.01204},
 primaryClass = {math.AT},
       adsurl = {https://ui.adsabs.harvard.edu/abs/2016arXiv161101204H}
}

@ARTICLE{Par23,
       author = {{Park}, Doosung},
        title = "{Syntomic cohomology and real topological cyclic homology}",
      journal = {arXiv e-prints},
         year = 2023,
        month = nov,
          eid = {arXiv:2311.06593},
        pages = {arXiv:2311.06593},
          doi = {10.48550/arXiv.2311.06593},
archivePrefix = {arXiv},
       eprint = {2311.06593},
 primaryClass = {math.KT},
       adsurl = {https://ui.adsabs.harvard.edu/abs/2023arXiv231106593P}
}

@Article{Pet24,
 Author = {Petersen, Sarah},
 Title = {The {{\(H \underline{\mathbb{F}}_2\)}}-homology of {{\(C_2\)}}-equivariant {Eilenberg}-{Mac} {Lane} spaces},
 FJournal = {Algebraic \& Geometric Topology},
 Journal = {Algebr. Geom. Topol.},
 ISSN = {1472-2747},
 Volume = {24},
 Number = {8},
 Pages = {4487--4518},
 Year = {2024},
 Language = {English},
 DOI = {10.2140/agt.2024.24.4487},
 zbMATH = {7965297}
}

@Article{CBDHHLMNNS23,
 Author = {Calm{\`e}s, Baptiste and Dotto, Emanuele and Harpaz, Yonatan and Hebestreit, Fabian and Land, Markus and Moi, Kristian and Nardin, Denis and Nikolaus, Thomas and Steimle, Wolfgang},
 Title = {Hermitian {K}-theory for stable {{\(\infty\)}}-categories. {I}: {Foundations}},
 FJournal = {Selecta Mathematica. New Series},
 Journal = {Sel. Math., New Ser.},
 ISSN = {1022-1824},
 Volume = {29},
 Number = {1},
 Pages = {269},
 Note = {Id/No 10},
 Year = {2023},
 Language = {English},
 DOI = {10.1007/s00029-022-00758-2},
 zbMATH = {7625214},
 Zbl = {1522.19001}
}

@article{CBDHHLMNNSII,
 Author = {Calm{\`e}s, Baptiste and Dotto, Emanuele and Harpaz, Yonatan and Hebestreit, Fabian and Land, Markus and Moi, Kristian and Nardin, Denis and Nikolaus, Thomas and Steimle, Wolfgang},
 Title = {Hermitian {K}-theory for stable {{\(\infty\)}}-categories. {II}: {Cobordism categories and additivity}},
 journal = {To appear in Acta Math.},
 Year = {2025},
}

@article{CBDHHLMNNSIII,
 Author = {Calm{\`e}s, Baptiste and Dotto, Emanuele and Harpaz, Yonatan and Hebestreit, Fabian and Land, Markus and Moi, Kristian and Nardin, Denis and Nikolaus, Thomas and Steimle, Wolfgang},
 Title = {Hermitian {K}-theory for stable {{\(\infty\)}}-categories. {III}: {Grothendieck-Witt groups of rings}},
 journal = {To appear in Ann. Math.},
 Year = 2025,
}

@Article{May66,
 Author = {May, J. P.},
 Title = {The cohomology of restricted {Lie} algebras and of {Hopf} algebras},
 FJournal = {Journal of Algebra},
 Journal = {J. Algebra},
 ISSN = {0021-8693},
 Volume = {3},
 Pages = {123--146},
 Year = {1966},
 Language = {English},
 DOI = {10.1016/0021-8693(66)90009-3},
 zbMATH = {3260855},
 Zbl = {0163.03102}
}

@Book{Mil74,
 Author = {Milnor, John W. and Stasheff, James D.},
 Title = {Characteristic classes},
 FSeries = {Annals of Mathematics Studies},
 Series = {Ann. Math. Stud.},
 Volume = {76},
 Year = {1974},
 Publisher = {Princeton University Press, Princeton, NJ},
 Language = {English},
 DOI = {10.1515/9781400881826},
 zbMATH = {3468033},
 Zbl = {0298.57008}
}

@article{Nak12,
 author = {Nakaoka, Hiroyuki},
 title = {Ideals of {Tambara} functors},
 fjournal = {Advances in Mathematics},
 journal = {Adv. Math.},
 issn = {0001-8708},
 volume = {230},
 number = {4-6},
 pages = {2295--2331},
 year = {2012},
 language = {English},
 doi = {10.1016/j.aim.2012.04.021},
 zbMATH = {6060970},
 Zbl = {1246.19001}
}

@Article{DMP22,
 Author = {Dotto, Emanuele and Moi, Kristian Jonsson and Patchkoria, Irakli},
 Title = {Witt vectors, polynomial maps, and real topological {Hochschild} homology},
 FJournal = {Annales Scientifiques de l'{\'E}cole Normale Sup{\'e}rieure. Quatri{\`e}me S{\'e}rie},
 Journal = {Ann. Sci. {\'E}c. Norm. Sup{\'e}r. (4)},
 ISSN = {0012-9593},
 Volume = {55},
 Number = {2},
 Pages = {473--535},
 Year = {2022},
 Language = {English},
 DOI = {10.24033/asens.2500},
 zbMATH = {7548803},
 Zbl = {1502.18011}
}

@Misc{Lic73,
 Author = {Lichtenbaum, Stephen},
 Title = {Values of zeta-functions, {\'e}tale cohomology, and algebraic {{\(K\)}}-theory},
 Year = {1973},
 Language = {English},
 HowPublished = {Algebraic {{\(K\)}}-{Theory} {II}, {Proc}. {Conf}. {Battelle} {Inst}. 1972, {Lect}. {Notes} {Math}. 342, 489-501 (1973).},
 zbMATH = {3445348},
 Zbl = {0284.12005}
}

@Misc{Qui75,
 Author = {Quillen, Daniel},
 Title = {Higher algebraic {K}-theory},
 Year = {1975},
 Language = {English},
 HowPublished = {Proc. int. {Congr}. {Math}., {Vancouver} 1974, {Vol}. 1, 171-176 (1975).},
 zbMATH = {3559805},
 Zbl = {0359.18014}
}

@ARTICLE{QS21a,
       author = {{Quigley}, J.~D. and {Shah}, Jay},
        title = "{On the equivalence of two theories of real cyclotomic spectra}",
      journal = {arXiv e-prints},
         year = 2021,
        month = dec,
          eid = {arXiv:2112.07462},
        pages = {arXiv:2112.07462},
          doi = {10.48550/arXiv.2112.07462},
archivePrefix = {arXiv},
       eprint = {2112.07462},
 primaryClass = {math.AT},
       adsurl = {https://ui.adsabs.harvard.edu/abs/2021arXiv211207462Q}
}

@ARTICLE{QS21b,
       author = {{Quigley}, J.~D. and {Shah}, Jay},
        title = "{On the parametrized Tate construction}",
      journal = {arXiv e-prints},
         year = 2021,
        month = oct,
          eid = {arXiv:2110.07707},
        pages = {arXiv:2110.07707},
          doi = {10.48550/arXiv.2110.07707},
archivePrefix = {arXiv},
       eprint = {2110.07707},
 primaryClass = {math.AT},
       adsurl = {https://ui.adsabs.harvard.edu/abs/2021arXiv211007707Q}
}

@article{Ran78,
 author = {Ranicki, Andrew},
 title = {On the algebraic {{\(L\)}}-theory of semisimple rings},
 fjournal = {Journal of Algebra},
 journal = {J. Algebra},
 issn = {0021-8693},
 volume = {50},
 pages = {242--243},
 year = {1978},
 language = {English},
 doi = {10.1016/0021-8693(78)90185-0},
 zbMATH = {3579020},
 Zbl = {0372.18005}
}

@Article{RW77,
 Author = {Ravenel, Douglas C. and Wilson, W. Stephen},
 Title = {The {Hopf} ring for complex cobordism},
 FJournal = {Journal of Pure and Applied Algebra},
 Journal = {J. Pure Appl. Algebra},
 ISSN = {0022-4049},
 Volume = {9},
 Pages = {241--280},
 Year = {1977},
 Language = {English},
 DOI = {10.1016/0022-4049(77)90070-6},
 zbMATH = {3581291},
 Zbl = {0373.57020}
}

@phdthesis{Ull-thesis,
  title        = {On the Regular Slice Spectral Sequence},
  author       = {Ullman, John},
  year         = 2013,
  month        = {July},
  address      = {Cambridge, MA},
  note         = {Available at \url{https://dspace.mit.edu/bitstream/handle/1721.1/83701/864163657-MIT.pdf?sequence=2&isAllowed=y}},
  school       = {Massachusetts Institute of Technology},
  type         = {PhD thesis}
}

@Article{Ull13,
 Author = {Ullman, John},
 Title = {On the slice spectral sequence},
 FJournal = {Algebraic \& Geometric Topology},
 Journal = {Algebr. Geom. Topol.},
 ISSN = {1472-2747},
 Volume = {13},
 Number = {3},
 Pages = {1743--1755},
 Year = {2013},
 Language = {English},
 DOI = {10.2140/agt.2013.13.1743},
 zbMATH = {6177483},
 Zbl = {1271.55015}
}

@Article{Wil73,
 Author = {Wilson, W. Stephen},
 Title = {The {{\(\Omega\)}}-spectrum for {Brown}-{Peterson} cohomology. {I}},
 FJournal = {Commentarii Mathematici Helvetici},
 Journal = {Comment. Math. Helv.},
 ISSN = {0010-2571},
 Volume = {48},
 Pages = {45--55},
 Year = {1973},
 Language = {English},
 DOI = {10.1007/BF02566110},
 URL = {https://eudml.org/doc/139536},
 zbMATH = {3405220},
 Zbl = {0256.55007}
}

@Book{Wil82,
 Author = {Wilson, W. Stephen},
 Title = {Brown-{Peterson} homology: {An} introduction and sampler},
 FSeries = {Regional Conference Series in Mathematics},
 Series = {Reg. Conf. Ser. Math.},
 ISSN = {0160-7642},
 Volume = {48},
 ISBN = {0-8219-1699-3},
 Year = {1982},
 Publisher = {Providence, RI: American Mathematical Society (AMS)},
 Language = {English},
 zbMATH = {3820701},
 Zbl = {0518.55001}
}

@Book{Rav86,
 Author = {Ravenel, Douglas C.},
 Title = {Complex cobordism and stable homotopy groups of spheres},
 FSeries = {Pure and Applied Mathematics (Academic Press)},
 Series = {Pure Appl. Math., Academic Press},
 ISSN = {0079-8169},
 Volume = {121},
 Year = {1986},
 Publisher = {Academic Press, New York, NY},
 Language = {English},
 zbMATH = {3984103},
 Zbl = {0608.55001}
}

@Article{Rog99,
 Author = {Rognes, John},
 Title = {Topological cyclic homology of the integers at two},
 FJournal = {Journal of Pure and Applied Algebra},
 Journal = {J. Pure Appl. Algebra},
 ISSN = {0022-4049},
 Volume = {134},
 Number = {3},
 Pages = {219--286},
 Year = {1999},
 Language = {English},
 DOI = {10.1016/S0022-4049(97)00155-2},
 zbMATH = {1308294},
 Zbl = {0929.19003}
}

@Article{Rog99b,
 Author = {Rognes, John},
 Title = {Algebraic {{\(K\)}}-theory of the two-adic integers},
 FJournal = {Journal of Pure and Applied Algebra},
 Journal = {J. Pure Appl. Algebra},
 ISSN = {0022-4049},
 Volume = {134},
 Number = {3},
 Pages = {287--326},
 Year = {1999},
 Language = {English},
 DOI = {10.1016/S0022-4049(97)00156-4},
 zbMATH = {1308293},
 Zbl = {0929.19004}
}

@misc{Rog00,
	author = {Rognes, John},
	note = {Talk at Oberwolfach},
	title = {Algebraic {{\(K\)}}-theory of finitely presented ring spectra},
	year = {2000}
}

@ARTICLE{Ste25,
       author = {{Stewart}, Natalie},
        title = "{On tensor products with equivariant commutative operads}",
      journal = {arXiv e-prints},
         year = 2025,
        month = apr,
          eid = {arXiv:2504.02143},
        pages = {arXiv:2504.02143},
archivePrefix = {arXiv},
       eprint = {2504.02143},
 primaryClass = {math.AT},
       adsurl = {https://ui.adsabs.harvard.edu/abs/2025arXiv250402143S}
}

@incollection {Wal82,
    AUTHOR = {Waldhausen, Friedhelm},
     TITLE = {Algebraic {$K$}-theory of spaces, a manifold approach},
 BOOKTITLE = {Current trends in algebraic topology, {P}art 1 ({L}ondon,
              {O}nt., 1981)},
    SERIES = {CMS Conf. Proc.},
    VOLUME = {2},
     PAGES = {141--184},
 PUBLISHER = {Amer. Math. Soc., Providence, RI},
      YEAR = {1982},
      ISBN = {0-8218-6001-1},
   MRCLASS = {18F25 (57R52)},
  MRNUMBER = {686115},
}

@incollection {Wal73,
    AUTHOR = {Wall, C. T. C.},
     TITLE = {Foundations of algebraic {$L$}-theory},
 BOOKTITLE = {Algebraic {$K$}-theory, {III}: {H}ermitian {$K$}-theory and
              geometric applications ({P}roc. {C}onf., {B}attelle {M}emorial
              {I}nst., {S}eattle, {W}ash., 1972)},
    SERIES = {Lecture Notes in Math.},
    VOLUME = {Vol. 343},
     PAGES = {266--300},
 PUBLISHER = {Springer, Berlin-New York},
      YEAR = {1973},
   MRCLASS = {18F25 (10C05 15A63 57D65)},
  MRNUMBER = {357550},
MRREVIEWER = {A.\ A.\ Ranicki},
}

@ARTICLE{Wil16,
       author = {{Wilson}, Dylan},
        title = "{Power operations for $\text{H}\underline{\mathbb{F}}_2$ and a cellular construction of $\text{BP}\mathbf{R}$}",
      journal = {arXiv e-prints},
         year = 2016,
        month = nov,
          eid = {arXiv:1611.06958},
        pages = {arXiv:1611.06958},
          doi = {10.48550/arXiv.1611.06958},
archivePrefix = {arXiv},
       eprint = {1611.06958},
 primaryClass = {math.AT},
       adsurl = {https://ui.adsabs.harvard.edu/abs/2016arXiv161106958W}
}

@ARTICLE{Yan25,
       author = {{Yang}, Lucy},
        title = "{A filtered Hochschild-Kostant-Rosenberg theorem for real Hochschild homology}",
      journal = {arXiv e-prints},
         year = 2025,
        month = mar,
          eid = {arXiv:2503.03024},
        pages = {arXiv:2503.03024},
          doi = {10.48550/arXiv.2503.03024},
archivePrefix = {arXiv},
       eprint = {2503.03024},
 primaryClass = {math.AT},
       adsurl = {https://ui.adsabs.harvard.edu/abs/2025arXiv250303024Y}
}

\end{document}